\documentclass[letterpaper,11pt]{article}
\usepackage[margin=1in]{geometry}  %margin of page
\usepackage{setspace}  %double-space
\usepackage{natbib}
 \bibpunct[, ]{(}{)}{,}{a}{}{,}%

\usepackage{amsfonts}
\usepackage{amsthm}
\usepackage{mathrsfs}
\usepackage{amsmath}
\usepackage{amssymb}
\usepackage{latexsym}
\usepackage{indentfirst}
\usepackage{footmisc}
\usepackage{endnotes}
\usepackage{algorithm}
\usepackage{algorithmic}
\usepackage{float}
\usepackage{graphicx}
\usepackage{epstopdf}
\usepackage{longtable}
\usepackage{enumerate}
\usepackage{enumitem}
\usepackage{authblk}
\usepackage{bbm}
\usepackage{booktabs}
\usepackage{url}

\usepackage{multirow}
\usepackage{bm}
\usepackage{slashbox}
\usepackage{array}
\usepackage{subfigure}
\usepackage{caption}
\usepackage{color}

\usepackage{pifont}  %to draw pentagon and asterisk
\newcommand{\pfs}{\mathrm{PFS}}
\newcommand{\pcs}{\mathrm{PCS}}

\newcommand{\Erm}{\mathrm{E}}
\newcommand{\Arm}{\mathrm{A}}
\newcommand{\Mrm}{\mathrm{M}}
\def\st{\mathrm{s.t.}}

\def\F{\mathcal{F}}

\def\X{\mathbf{X}}
\def\Y{\mathbf{Y}}

\def\bbf{\mathbf{f}}

\def\x{\mathbf{x}}
\def\z{\mathbf{z}}

\def\uf{\mathrm{f}}

\def\Ccal{\mathcal{C}}

\def\Ncal{\mathcal{N}}
\def\Xcal{\mathcal{X}}

\def\Gcal{\mathcal{G}}
\def\Rcal{\mathcal{R}}

\def\Ucal{\mathcal{U}}
\def\Vcal{\mathcal{V}}

\def\1{\mathbbm{1}}

\def\bbeta{\boldsymbol{\beta}}

\def\alphab{\boldsymbol{\alpha}}

\def\bSigmap{\boldsymbol{\Sigma}^{(p)}}
\def\blam{\boldsymbol{\lambda}}

\def\Dcal{\mathcal{D}}
\def\bepsilon{\boldsymbol{\epsilon}}
\def\btheta{\boldsymbol{\theta}}
\def\dfrak{\tilde{\mathrm{d}}}
\def\sfrak{\tilde{\mathrm{s}}}
\def\v{\mathbf{v}}

\def\PP{\mathbb{P}}
\def\EE{\mathbb{E}}
\def\RR{\mathbb{R}}

\newtheorem{assumption}{ASSUMPTION}
\newtheorem{lemma}{LEMMA}
\newtheorem{remark}{REMARK}
\newtheorem{theorem}{THEOREM}
\newtheorem{proposition}{Proposition}

\allowdisplaybreaks

\newcommand{\argmin}{\arg\min}

\title{A Contextual Ranking and Selection Method for Personalized Medicine}

\author[a]{Jianzhong Du}
\author[b]{Siyang Gao}
\author[c]{Chun-Hung Chen}
\affil[a]{School of Management, Fudan University, Shanghai, China, jianzhodu2-c@my.cityu.edu.hk}
\affil[b]{Department of Systems Engineering and School of Data Science, City University of Hong Kong, Hong Kong, China, siyangao@cityu.edu.hk}
\affil[c]{Department of Systems Engineering \& Operations Research, George Mason University, Fairfax, VA 22030, USA, cchen9@gmu.edu}

\date{\vspace{-8ex}}

\begin{document}
%%%%%%%%%%%%%%%%

\maketitle

\begin{abstract}
\textbf{Problem definition:} Personalized medicine (PM) seeks the best treatment for each patient among a set of available treatment methods. Since a specific treatment does not work well on all patients, traditionally, the best treatment was selected based on the doctor's personal experience and expertise, which is subject to human errors. In the meantime, stochastic models have been well developed in the literature for a lot of major diseases. This gives rise to a simulation-based solution for PM, which uses the simulation tool to evaluate the performance for pairs of treatment and patient biometric characteristics, and based on that, selects the best treatment for each patient characteristics. \textbf{Methodology/results:} In this research, we extend the ranking and selection (R\&S) model in simulation-based decision making to solving PM. The biometric characteristics of a patient is treated as a context for R\&S, and we call it contextual ranking and selection (CR\&S). We consider two formulations of CR\&S with small and large context spaces respectively and develop new techniques for solving them and identifying the rate-optimal budget allocation rules. Based on them, two selection algorithms are proposed, which can be shown to be numerically superior via a set of tests on abstract and real-world examples. \textbf{Managerial implications:} This research provides a systematic way of conducting simulation-based decision-making for PM. To improve the overall decision quality for the possible contexts, more simulation efforts should be devoted to contexts in which it is difficult to distinguish between the best treatment and non-best treatments, and our results quantify the optimal tradeoff of the simulation efforts between the pairs of contexts and treatments.
% Enter your abstract

\end{abstract}

% Sample

% Fill in data. If unknown, outcomment the field

%%%%%%%%%%%%%%%%%%%%%%%%%%%%%%%%%%%%%%%%%%%%%%%%%%%%%%%%%%%%%%%%%%%%%%

\section{Introduction}
\label{sec:1}

Personalized medicine (PM) (also known as precision medicine or P4 medicine) is an emerging healthcare problem. Benefiting from the advance of medical knowledge and technology, patients usually have access to a set of competing and sometimes complementary medical treatment methods for their diseases. However, the treatment used for a patient should be carefully chosen, because the effectiveness of the treatment might heavily depend on the patient's biometric characteristics. For instance, the highly active antiretroviral therapy (a standard treatment for AIDS) has substantially different profiles in efficacy and toxicity across subgroups of patients, influenced by the virus level at the time of receiving treatments and the gender and behavior pattern of the patient \citep{cai2011}. In this research, we call such biometric characteristics contexts. PM aims to determine the best treatments for contexts that might appear in practice, and thus provides tailored treatment for each patient. This is substantially different from the traditional evaluation of treatment performance (World Health Organization \citeyear{who2003}, Chap. 1).

PM involves evaluating the effectiveness of medical treatments under different contexts. Typically there are two ways to do it, by trial-based and model-based approaches. The trial-based approach applies statistical analysis to a series of well-designed clinical trials, and is capable of supporting personalized medicine with a large set of contexts \citep{schork2015}. However, this approach suffers from major ethical issues. Statistical analysis and inference rely on comparing the results of the treatment group and control group. It is unethical if the patients in the control group have worsening progression and are not allowed to take experimental drugs \citep{mok2011}. In addition, this approach is further complicated by the prohibitively large amount of resource needed for following up the test results and making decisions. These drawbacks of the trial-based approach can limit its implementation in the real world \citep{hamburg2010}.

The model-based approach employs mathematical models to depict the progression of the disease, and based on it, assesses the effectiveness of the treatment \citep{Garnett2011}. It does not involve making experiments on humans, and thus can avoid the abovementioned ethical and resource-related issues in trial-based approaches. The evaluation model is generally stochastic due to the uncertainty in the model structure and transition and the estimation of the model parameters. From the personalized perspective, the effectiveness of a treatment is also random among individuals, even under the same patient context \citep{chickhealth2006}. There has been a rich body of literature on the application of stochastic models to healthcare problems, e.g., the epidemiological transmission dynamics \citep{chickhealth2001,chickhealth2008,alonso2007} and HIV preventions for susceptible populations \citep{tan2012}.

In this research, we will focus on the model-based approach for PM and study the problem of efficiently identifying the best treatment under all the possible patient contexts within a finite computing time. In view that stochastic models for practical problems can be large-scaled, complex and not analytical, we will use the generic tool of simulation to evaluate the performance of them.

In simulation experiments, designs (a terminology in systems engineering, analogous to treatments in medical decision problems) are simulated for multiple replications. Their performance estimators (typically sample means) are compared and the estimated best design is selected. This practice imposes two challenges for the purpose of PM. First, for a given patient context, the probability that we correctly select the true best treatment is always less than one with a finite simulation budget. The randomness in the model will cause a non-best treatment to occasionally outperform the best, leading us to a suboptimal decision. Second, the difficulties in correctly selecting the best treatment vary among contexts. The best treatment under some contexts is easy to identify, while for some other contexts, it can be highly difficult.

To address these two challenges, in this research, we propose to utilize the simulation budget to maximize the chance of identifying the best treatment under each possible context. It is achieved by smartly controlling the number of simulation replications allocated to each pair of context and treatment, so as to concentrate the computing efforts on contexts where the best treatment is more difficult to identify. By doing so, the best treatment under each context can be correctly selected with a higher confidence in a limited time. We call this problem contextual ranking and selection (CR\&S).

CR\&S is closely related to two streams of literature. The first is ranking and selection (R\&S). R\&S is a well-established model in the field of simulation optimization. It aims to allocate the simulation budget to a set of competing system designs in order to efficiently select the best one. Representative R\&S methods include the optimal computing budget allocation (OCBA) \citep{chen2000,fu2007}, value of information procedures (VIP) \citep{frazier2008,chick2010} and indifference-zone (IZ) mehtods \citep{kim2001,nelson2001}. However, these procedures do not consider contexts, and thus cannot be applied to CR\&S \citep{goodwin2022real}.

Recently, \cite{shen2019} considered the R\&S problem in the presence of continuous contexts and used the linear models to predict the design performance. \cite{li2018data} generalized the method of \cite{shen2019} to handle high-dimensional context spaces. However, these two studies pursue a different goal from this research, which is to provide performance guarantees for the designs (treatments) selected, instead of optimizing the design performance. Compared to them, the OCBA-type method is less conservative, in the sense that it can achieve better design performance with a less simulation budget \citep{branke2007}, at the cost of losing the performance guarantee on the designs selected. Therefore, our model and method are more appropriate when efficiency is important (e.g., when the simulation budget available is relatively small).

In addition, \cite{hu2017sequential} and \cite{pearce2017efficient} considered the large-scale problem of CR\&S and used the method of Bayesian optimization to solve it. They employed the stochastic kriging model for predictions of the design performance and focused on how to search the design and context spaces, rather than developing the budget allocation rules. The performances of their algorithms were only studied empirically. \cite{ding2022knowledge} extended the algorithm in \cite{pearce2017efficient} and showed that the new algorithm is consistent. Compared to these work, this research aims at the budget allocations of the small-scale and large-scale problems and shows that the proposed budget allocation rules and selection algorithms are asymptotically optimal, which is a stronger result than consistency.

The second stream of literature related to CR\&S is the best arm identification (BAI). BAI is more studied in the fields of statistics and machine learning, but it is a very similar model as R\&S, aiming to identify the best arm from a finite set by adaptively pulling the arms and learning their rewards without consideration of contexts \citep{audibert2010best,kaufmann2016complexity,russo2020simple}. Recently, BAI has also been extended to the context environment, known as contextual bandits \citep{liz2022instance}. In bandit problems, the sampling object is typically the real system, while in R\&S, the sampling object is the simulation model. The different sampling objects do not cause any differences when solving BAI and R\&S problems, but make the structures of contextual bandits and CR\&S problems fundamentally different. In contextual bandits, contexts are associated with the real system and are thus out of the experimenter's control, i.e., the experimenter can only decide which arm to sample given the context that appears, but cannot decide which context to appear or sample \citep{tewari2017ads}. In contrast, CR\&S considers an entirely simulated environment, in which contexts are also input variables to the simulation model and are controlled by the experimenter. As a result, the experimenter needs to decide both the context and design (the context-design pair) to sample. It leads to a different and more complex decision problem.

Our contributions in this research are four-fold. First, we study three measures for evaluating the evidence of correct selection over the context space. These measures are extensions of the probability of correct selection (PCS) used in R\&S to the contextual setting, and are capable of depicting the quality of the estimated best treatment under all the possible contexts. We show that the three measures are asymptotically equivalent, in the sense that they have the same rate function.

Second, we propose two formulations for the PM problem. Both formulations optimize the rate function of the three measures under a simulation budget constraint. One formulation samples all treatment-context pairs and is suitable for a small context space. In the other formulation, treatment performance and context are assumed to have linear relationship. This is suitable for a large context space.

Third, for both formulations, we develop the rate-optimal selection rules and devise easily implementable selection algorithms, called CR\&S Algorithms 1 and 2. We show that the two algorithms can recover the rate-optimal selection rules.

Last, we conduct extensive numerical experiments to assess the performances of the two algorithms. We first test them on a set of benchmark functions and demonstrate their superiority in solving different types of problems. Next, we apply the algorithms to two real-world PM problems and obtain the medical decision maps for them.

The rest of the paper is organized as follows. Section \ref{sec:3} introduces the basic notation and assumptions. Section \ref{sec:4} studies three objective measures of CR\&S and their rate functions. Sections \ref{sec:discrete} and \ref{sec:lin} consider the PM problem with small and large context spaces respectively. They formulate and solve the selection problems, develop selection algorithms for implementation and theoretically study the their performances. Numerical examples and computational results are provided in Section \ref{sec:7}, followed by conclusions and discussion in Section \ref{sec:con}.

\section{Preliminaries}
\label{sec:3}

Suppose there are $k$ different treatments. The performance of each treatment depends on $\X=(X_1,\ldots,X_d)^\top$, a vector of random contexts with support $\Xcal\subseteq \mathbb{R}^d$. For each treatment $i=1,2,\ldots,k$, let $Y_{il} (\x)$ be the $l$th simulation sample from treatment $i$ and context $\x$, and $y_i (\x)$ be the mean performance of this treatment. We have $Y_{il} (\x)= y_i (\x)+\epsilon_{il} (\x)$, where $\epsilon_{il} (\x)$ is the random noise incurred in the simulation. Denote $n_i(\x)$ as the number of simulation replications for treatment $i$ and context $\x$. The sample mean $\bar{Y}_{i} (\x)=\frac{1}{n_i(\x)}\sum_{l=1}^{n_i(\x)}Y_{il} (\x)$. Without loss of generality, we let the best treatment $i^*(\x)$ under context $\x$ be the treatment with the smallest mean performance.

Throughout the paper, we assume that $\Xcal$ has a finite number of $m$ possible contexts $\x_1,\x_2,...,\x_m$. This setting aligns with context spaces that are finite in nature. For infinite context spaces (continuous or discrete and unbounded), we usually do not need to find the best treatment for each context; instead, a common practice is to classify the values of context variables into a number of categories/levels. For example, when treating diabetic patients, a key context variable is the body mass index (BMI) of the patients and it takes real values. Two possible ways to process BMI is to classify it into categories $<$18.5 underweight, 18.5-24.9 normal weight, 25.0-29.9 overweight and $\geq$30.0 obesity \citep{whobmi}, or more accurately, into levels $<$18, 18, 19, ..., 29, 30 and $>$30. To this end, the finite setting provides great flexibility in the level of contextual discrepancy we want to distinguish when formulating the problem.

Specifically, we consider two cases for the context space. The first case is when the context space is small, and we have time to simulate all the treatment-context pairs. The second case is when the context space is large, and we only have time to simulate treatments under a fraction of contexts. In this case, we further assume that all contexts lie on a grid, and the relationship between the treatment performance and contexts can be described by linear models, so the performances of treatments under un-simulated contexts can be interpolated. In this research, we call them \emph{small-scale problem} and \emph{large-scale problem} respectively.

Suppose $n$ is our total simulation budget (number of simulation replications), and $n_{i,j}$ is the number of simulation replications we allocate to treatment $i$ under context $\x_j$. Let $\alpha_{i,j}=n_{i,j}/n$ and $\alphab = (\alpha_{1,1},\alpha_{2,1},\dots,\alpha_{k,1}, \alpha_{1,2}, \alpha_{2,2}, \dots, \alpha_{k,2},\dots, \alpha_{1,m}, \alpha_{2,m}, \dots, \alpha_{k,m})$ be the vector of $\alpha_{i,j}$. We make the following technical assumptions in our analysis.

\begin{assumption}\label{ass:1}
	The best treatment $i^*(\x)$ is unique for all $\x\in\Xcal$.
\end{assumption}

\begin{assumption}\label{ass:2}
	$Y_{il} (\x)$'s are independent across different $i$, $l$ and $\x$.
\end{assumption}

\begin{assumption}\label{ass:3}
	$Y_{il}(\x)$'s are normally distributed with mean $y_i(\x)$ and variance $\sigma^2(\x)$.
\end{assumption}

Assumption \ref{ass:1} assumes that the best treatment under each of the $m$ contexts is unique, because two treatments with the same mean performance cannot be distinguished. The assumptions of independence and normality of samples in Assumptions \ref{ass:2} and \ref{ass:3} are standard in the simulation optimization literature \citep{law2000}. The independence between simulation samples can be achieved by using independent sequences of random numbers in different simulation runs. The normality assumption is typically satisfied in simulation because the output is obtained from an average performance or batch means. According to the Central Limit Theorem, it is approximately normal.

\section{Objective Measures}
\label{sec:4}

In this section, we discuss three objective measures for PM. Next, we analyze the rate functions of the three measures and establish their equivalence.

Suppose performance $y_i(\x)$ of treatment $i$ under context $\x$ is estimated by $\hat{y}_i(\x)$. For context $\x$, a correct selection happens when the estimated best treatment $\hat{i}^*(\x)$ is identical to the real best treatment $i^*(\x)$. However, the correct selection can never be guaranteed in practice with a finite simulation budget. Under a fixed context $\x$, traditional R\&S typically assesses the quality of the selection for the best treatment by the probability of correct selection (PCS)
\begin{equation*}
	\pcs(\x)=\PP (\hat{i}^*(\x)=i^*(\x))=\PP\left(\bigcap_{i=1,i\neq i^*(\x)}^k \Big(\hat{y}_{i^*(\x)} (\x)<\hat{y}_{i} (\x)\Big)\right),
\end{equation*}
and seeks to either maximize this probability or guarantee a pre-specified level for it. The probability here is taken with respect to the random noises in the simulation samples.

In CR\&S, each context $\x$ is associated to a R\&S problem. We want to provide the best treatments for all the $m$ contexts, and therefore need measures for evaluating the quality of the selection over the entire context space $\Xcal$. To fulfill this need, we consider the following three measures based on PCS:
\begin{align*}
	\pcs_{\Erm}&=\EE[\pcs(\X)]=\sum_{j=1}^m p_j \pcs(\x_j),\\
	\pcs_{\Mrm}&=\min_{\x\in\Xcal}\pcs(\x),\\
	\pcs_{\Arm}&=\PP\left(\bigcap_{j=1}^m\bigcap_{i=1,i\neq i^*(\x_j)}^k \Big(\hat{y}_{i^*(\x_j)} (\x_j)<\hat{y}_{i} (\x_j)\Big)\right).
\end{align*}

In $\pcs_{\Erm}$, $p_j$ is the probability of $\X=\x_j$, $j=1,2,...,m$. $\pcs_{\Erm}$ describes the expected probability of correct selection over $\Xcal$, where the expectation is taken with respect to the randomness of $\X$. $\pcs_{\Mrm}$ shows the worst-case performance of $\pcs(\x)$ over $\Xcal$. This measure is, in some sense, similar to the worst-case performance in robust optimization \citep{bertsimas2011} and R\&S with input uncertainty \citep{fan2018}.

$\pcs_{\Arm}$ is defined in a different way from the two measures above. It is not based on $\pcs(\x)$; instead, it requires correctness for all the comparisons of interest, i.e., comparisons between the estimated best treatment and the alternatives under all the possible contexts. $\pcs_{\Arm}$ sets the highest standard for the quality of the selection among the three, and is appropriate to be used by conservative decision makers. It is obvious that $\pcs_{\Arm}\leq\pcs_{\Mrm}\leq\pcs_{\Erm}$. Intuitively, $\pcs_{\Erm}$ and $\pcs_{\Mrm}$ are the average and the minimum probabilities of the best treatment being identified among all the patient contexts; $\pcs_{\Arm}$ is the probability of the best treatment being identified for all the contexts. Note that $\pcs_{\Arm}$ is newly proposed for CR\&S, while $\pcs_{\Erm}$ and $\pcs_{\Mrm}$ have been used and discussed in \cite{shen2019} as measures for R\&S with covariates.

Due to the lack of analytical expressions of $\pcs_{\Erm}$, $\pcs_{\Mrm}$ and $\pcs_{\Arm}$, it is challenging to find the exact optimizers of them. As a result, it is common to instead pursue their asymptotic optimizers (optimizers as $n\rightarrow\infty$) in the R\&S literature \citep{chen2000,frazier2008,ryzhov2016}. Asymptotic optimizers become close to the real optimizers when the simulation budget $n$ is large, and often demonstrate very good empirical performance when $n$ is small \citep{branke2007}. To find asymptotic optimizers of $\pcs_{\Erm}$, $\pcs_{\Mrm}$ and $\pcs_{\Arm}$, we can look for solutions that maximize the asymptotic performance of the three measures, i.e., solutions that maximize the rates at which they converge to 1. The following theorem characterizes these rates of the three measures.

\begin{theorem}\label{theo:1}
	Define probabilities of false selection $\pfs_{\Erm}=1-\pcs_{\Erm}$, $\pfs_{\Mrm}=1-\pcs_{\Mrm}$ and $\pfs_{\Arm}=1-\pcs_{\Arm}$. Under Assumptions \ref{ass:1}-\ref{ass:3}, the three measures $\pfs_{\Erm}$, $\pfs_{\Mrm}$ and $\pfs_{\Arm}$ converge exponentially and have the same rate function $\Rcal(\alphab)$. That is,
	\begin{align*}
		\lim_{n\rightarrow\infty}\frac{1}{n}\log \pfs_{\Erm}=
		\lim_{n\rightarrow\infty}\frac{1}{n}\log \pfs_{\Mrm}=
		\lim_{n\rightarrow\infty}\frac{1}{n}\log \pfs_{\Arm}=-\Rcal(\alphab).
	\end{align*}
	Moreover, it can be shown that
	$
	\Rcal(\alphab)=\min\limits_{j\in\{1,...,m\}}\min\limits_{i\in\{1,...,k\},i\neq i^*(\x_j)}-\lim\limits_{n\rightarrow\infty}\frac{1}{n}\log \PP\left(\hat{y}_{i^*(\x_j)}(\x_j)\geq\hat{y}_{i} (\x_j)\right).
	$
\end{theorem}

To interpret Theorem \ref{theo:1}, we pick $i_o$ and $j_o$ such that
\begin{equation*}
	(i_o,j_o)\in\argmin_{i\in\{1,...,k\},i\neq i^*(\x_j),j\in\{1,...,m\}} -\lim_{n\rightarrow\infty}\frac{1}{n}\log \PP\left(\hat{y}_{i^*(\x_j)}(\x_j)\geq\hat{y}_{i} (\x_j)\right).
\end{equation*}
The theorem shows that the three measures, though defined from different perspectives, converge at the same exponential rate $\Rcal(\alphab)$, where $\alphab$ is the sampling rate of each treatment-context pair. The rate function $\Rcal(\alphab)$ is characterized by the most difficult comparison among comparisons between the best treatment and non-best treatments under each context, i.e., the comparison of sample means between treatments $i^*(\x_{j_o})$ and $i_o$ under context $j_o$. The reason for this effect is that, the most difficult comparison has the slowest convergence rate, which dominates the convergence rates of the other comparisons, and thus represents the rate these measures converge at. Theorem \ref{theo:1} lays the foundation of this paper: instead of considering the three measures separately, we can solve them once and for all by directly optimizing the rate function $\Rcal(\alphab)$.

\section{Small-Scale Problem}
\label{sec:discrete}

In this section, we consider the small-scale problem, where our simulation budget is sufficient for simulating all the treatment-context pairs. The estimate $\hat{y}_i(\x_j)$ for the performance $y_i(\x_j)$ of treatment $i$ and context $\x_j$ is the sample mean $\bar{Y}_i(\x_j)$, $i=1,2,...,k$ and $j=1,2,...,m$.

\subsection{Rate-Optimal Budget Allocation Rule}

For the small-scale problem, optimization of the rate function $\Rcal(\alphab)$ is given by
\begin{align}\label{eq:1}
	\min \ -\Rcal(\alphab) \quad
	\st \  \sum_{i=1}^k \sum_{j=1}^m  \alpha_{i,j}=1, \
	\alpha_{i,j}\geq0, \ i=1,2,...,k, \ j=1,2,...,m.
\end{align}
The simulation budget constraint $\sum_{i=1}^k \sum_{j=1}^m  \alpha_{i,j}=1$ is equivalent to $\sum_{i=1}^k \sum_{j=1}^m  n_{i,j}=n$. This is an OCBA-like formulation \citep{chen2000}, which finds a simulation budget allocation strategy to optimize the measure of interest, i.e., the rate function $\Rcal(\alphab)$ in our problem.

Before we solve (\ref{eq:1}), we carry out more analysis on the rate function $\Rcal(\alphab)$. According to Theorem \ref{theo:1}, $\Rcal(\alphab)=\min_{j\in\{1,2,...,m\}}\min_{i\in\{1,...,k\},i\neq i^*(\x_j)}-\lim_{n\rightarrow\infty}\frac{1}{n}\log \PP\left(\hat{y}_{i^*(\x_j)}(\x_j)\geq\hat{y}_{i} (\x_j)\right)$. We denote
\begin{equation*}
	-\lim_{n\rightarrow\infty}\frac{1}{n}\log \PP\left(\bar{Y}_{i^*(\x_j)}(\x_j)\geq\bar{Y}_{i} (\x_j)\right) \doteq \Gcal_{i^*(\x_j),i,j}(\alpha_{i^*(\x_j),j},\alpha_{i,j}).
\end{equation*}
From the G\"{a}rtner-Ellis Theorem \citep{dembo1998}, for i.i.d. normal samples $Y_{il} (\x)$,
\begin{align*}
	\Gcal_{i^*(\x_j),i,j}(\alpha_{i^*(\x_j),j},\alpha_{i,j}) = \frac{(y_{i^*(\x_j)}(\x_j) - y_{i}(\x_j))^2}{ 2(\sigma_{i^*(\x_j)}^2(\x_j)/\alpha_{i^*(\x_j),j} + \sigma_{i}^2(\x_j)/\alpha_{i,j}) }.
\end{align*}
Then, an equivalent formulation of problem (\ref{eq:1}) is given by
\begin{align}
	\max & \ z \nonumber\\
	\st & \ \Gcal_{i^*(\x_j),i,j}(\alpha_{i^*(\x_j),j},\alpha_{i,j})\geq z, \ \ i=1,2,...,k \text{ and }i\neq i^*(\x_j), j=1,2,...,m,\nonumber\\
	& \sum_{i=1}^k \sum_{j=1}^m \alpha_{i,j}=1, \
	\alpha_{i,j}\geq0, \ \ i=1,2,...,k, \ j=1,2,...,m. \label{eq:2}
\end{align}

Note that $\Gcal_{i^*(\x_j),i,j}(\alpha_{i^*(\x_j),j},\alpha_{i,j})$ is a concave function, so $\Gcal_{i^*(\x_j),i,j}(\alpha_{i^*(\x_j),j},\alpha_{i,j})\geq z$ forms a convex set, and problem (\ref{eq:2}) is a convex optimization model. We can investigate the KKT conditions \citep{boyd2004} of this model to solve it.

\begin{theorem}\label{theo:3}
	The optimal solution to problem (\ref{eq:2}) is given by
	\begin{align}
		&\frac{\alpha_{i^*(\x_j),j}^2}{\sigma_{i^*(\x_j)}^2(\x_j)}=
		\sum_{i=1,i\neq i^*(\x_j)}^k \frac{\alpha_{i,j}^2} {\sigma_{i}^2(\x_j)}, \ \ j=1,2,...,m, \label{eq:add4}\\
		&\frac{(y_i (\x_j)-y_{i^*(\x_j)} (\x_j))^2} {\sigma_{i^*(\x_j)}^2(\x_j)/\alpha_{i^*(\x_j),j}+\sigma_i^2(\x_j)/\alpha_{i,j}}=
		\frac{(y_{i'} (\x_{j'})-y_{i^*(\x_{j'})} (\x_{j'}))^2} {\sigma_{i^*(\x_{j'})}^2(\x_{j'})/\alpha_{i^*(\x_{j'}),j'}+\sigma_{i'}^2(\x_{j'})/\alpha_{i',j'}}, \ j,j'=1,...,m,\nonumber\\
		& i,i'=1,...,k, \ i\neq i^*(\x_j) \text{ and } i'\neq i^*(\x_{j'}). \label{eq:add6}
	\end{align}
\end{theorem}

Theorem \ref{theo:3} indicates that the solution satisfying conditions (\ref{eq:add4})-(\ref{eq:add6})
corresponds to the budget allocation rule that maximize the convergence rate of $\pfs_{\Erm}$, $\pfs_{\Mrm}$ and $\pfs_{\Arm}$. Condition (\ref{eq:add4}) establishes for each context $\x_j$ a certain balance between the proportions of simulation replications allocated to the best treatment $\alpha_{i^*(\x_j),j}$ and those allocated to non-best treatments $\alpha_{i,j}$ for $i\neq i^*(\x_j)$, in the sense that $\frac{\alpha_{i^*(\x_j),j}^2}{\sigma_{i^*(\x_j)}^2(\x_j)}$ (represent the simulation replications allocated to the best treatment) should be equal to $\sum_{i=1,i\neq i^*(\x_j)}^k \frac{\alpha_{i,j}^2} {\sigma_{i}^2(\x_j)}$ (represent the simulation replications allocated to the non-best treatments). Condition (\ref{eq:add6}) further adjusts the ratios of the simulation replications allocated to any two non-best treatments under the same context and across different contexts. This condition suggests that the difficulty of correctly identifying a non-best treatment $i$ under context $\x_j$ as non-best can be reflected by the index $\frac{(y_i (\x_j)-y_{i^*(\x_j)} (\x_j))^2} {\sigma_{i^*(\x_j)}^2(\x_j)/\alpha_{i^*(\x_j),j}+\sigma_i^2(\x_j)/\alpha_{i,j}}$, which represents a comparison between the non-best treatment $i$ and the optimal treatment $i^*(\x_j)$ under context $\x_j$. To optimize the rate function, we should allocate the simulation budget to the treatment-context pairs such that this index remains equal for all the treatment-context pairs.

\subsection{Selection Algorithm}
\label{sec:6}

In this section, we develop a selection algorithm based on optimality conditions (\ref{eq:add4}) and (\ref{eq:add6}) for implementation and analyze its asymptotic performance.

For simplicity of presentation, define
\begin{align*}
	&\Ucal_j^b=\frac{\alpha_{i^*(\x_j),j}^2}{\sigma_{i^*(\x_j)}^2(\x_j)}, \quad \Ucal_j^{non}=\sum_{i=1,i\neq i^*(\x_j)}^k \frac{\alpha_{i,j}^2} {\sigma_{i}^2(\x_j)}, \ \ j=1,2,...,m, \\
	&\Vcal_{i,j}=\frac{(y_i (\x_j)-y_{i^*(\x_j)} (\x_j))^2} {\sigma_{i^*(\x_j)}^2(\x_j)/\alpha_{i^*(\x_j),j}+\sigma_i^2(\x_j)/\alpha_{i,j}}, \ \ j=1,2,...,m, \ i=1,2,...,k \text{ and } i\neq i^*(\x_j).
\end{align*}
Note that $\Ucal_j^b$ represents the simulation replications allocated to the best treatment $i^*(\x_j)$ under context $\x_j$, $\Ucal_j^{non}$ represents the simulation replications allocated to the non-best treatments $i$ under the same context, and $\Vcal_{i,j}$ represents the difficulty of correctly identifying the non-best treatment $i$ under context $\x_j$ as non-best. Then, conditions (\ref{eq:add4}) and (\ref{eq:add6}) can be re-written as
\begin{align}
	&\Ucal_j^b=\Ucal_j^{non}, \ \ j=1,2,...,m, \label{eq:14}\\
	&\Vcal_{i,j}=\Vcal_{i',j'}, \ j,j'=1,...,m, \ i,i'=1,...,k, \ i\neq i^*(\x_j) \text{ and } i'\neq i^*(\x_{j'}). \label{eq:16}
\end{align}
Since equations (\ref{eq:14}) and (\ref{eq:16}) do not have an analytical solution, we will design the algorithm in a simple and cost-effective manner that gradually reduces the error terms $|\Ucal_j^b-\Ucal_j^{non}|$ and $|\Vcal_{i,j}-\Vcal_{i',j'}|$ in (\ref{eq:14}) and (\ref{eq:16}).

Let
$
(i_*,j_*)\in\argmin_{j\in\{1,2,...,m\},i\in\{1,...,k\}\setminus\{i^*(\x_j)\}}\Vcal_{i,j}.
$
Note that
\begin{align*}
	&\frac{d \Ucal_j^b}{d \alpha_{i^*(\x_j),j}}=\frac{2\alpha_{i^*(\x_j),j}}{\sigma_{i^*(\x_j)}^2(\x_j)}>0,
	\ \ j=1,2,...,m;\\
	&\frac{\partial \Ucal_j^{non}}{\partial \alpha_{i,j}}=\frac{2\alpha_{i,j}}{\sigma_{i}^2(\x_j)}>0, \ \ j=1,2,...,m, \ i=1,2,...,k \text{ and } i\neq i^*(\x_j); \\
	&\frac{\partial \Vcal_{i,j}}{\partial \alpha_{i^*(\x_j),j}}= \frac{(y_i (\x_j)-y_{i^*(\x_j)} (\x_j))^2} {(\sigma_{i^*(\x_j)}^2(\x_j)/\alpha_{i^*(\x_j),j}+\sigma_i^2(\x_j)/\alpha_{i,j})^2}
	\frac{\sigma_{i^*(\x_j)}^2(\x_j)}{\alpha_{i^*(\x_j),j}^2}>0, \\
	&\frac{\partial \Vcal_{i,j}}{\partial \alpha_{i,j}}= \frac{(y_i (\x_j)-y_{i^*(\x_j)} (\x_j))^2} {(\sigma_{i^*(\x_j)}^2(\x_j)/\alpha_{i^*(\x_j),j}+\sigma_i^2(\x_j)/\alpha_{i,j})^2}
	\frac{\sigma_{i}^2(\x_j)}{\alpha_{i,j}^2}>0, \ \ j=1,2,...,m, \ i=1,2,...,k \text{ and } i\neq i^*(\x_j).
\end{align*}
That is, we can choose to increase the values of $\Ucal_j^b$ and $\Ucal_j^{non}$ by allocating more replications to treatment $i^*(\x_j)$ and treatment $i$ for any $i\neq i^*(\x_j)$ under context $\x_j$. We can also choose to increase the value of $\Vcal_{i,j}$ by allocating additional replications to either treatment $i^*(\x_j)$ or treatment $i$ under context $\x_j$, $i\in\{1,2,...,k\}$ and $i\neq i^*(\x_j)$.

To design a selection algorithm based on (\ref{eq:14}) and (\ref{eq:16}), suppose for a budget allocation, (\ref{eq:16}) cannot be fulfilled. To fix it, we will provide a small incremental budget to improve $\Vcal_{i_*,j_*}$ so that the gap between $\min_{j\in\{1,2,...,m\}}\min_{i\in\{1,...,k\}\setminus\{i^*(\x_j)\}}\Vcal_{i,j}$ and  $\max_{j\in\{1,2,...,m\}}\max_{i\in\{1,...,k\}\setminus\{i^*(\x_j)\}}\Vcal_{i,j}$ can be reduced. As discussed above, allocating more replications to treatment $i^*(\x_{j_*})$ or $i_*$ under context $\x_{j_*}$ both achieve this goal. To further decide
which of treatments $i^*(\x_{j_*})$ and $i_*$ receives the incremental budget, we check condition (\ref{eq:14}). If $\Ucal_{j_*}^b<\Ucal_{j_*}^{non}$, the additional replications should be allocated to the best treatment $i^*(\x_{j_*})$ in order to balance the equation; otherwise, the additional replications should be allocated to the non-best treatment $i_*$. This idea is summarized in CR\&S Algorithm 1 below.

\begin{algorithm}[htbp]
	\caption{CR\&S Algorithm 1.}
	\label{alg:ocba}
	\begin{algorithmic}
		\STATE {\bfseries Input:} Specify the number of contexts $m$, number of treatments $k$, total simulation budget $n$ and initial number of simulation replications $n_0$. Iteration counter $r\leftarrow0$.
		\STATE {\bfseries Initial Sampling:} Perform $n_0$ replications for treatment $i$ under context $\x_j$, $i=1,2,...,k$ and $j=1,2,...,m$, and calculate sample means $\bar{Y}_{i} (\x_j)$ and sample variances $\hat{\sigma}^2_{i} (\x_{j})$. Set $\hat{n}_{i,j}=n_0$, $n^{(r)}=\sum_{j=1}^m \sum_{i=1}^k \hat{n}_{i,j}$ and $\hat{\alpha}_{i,j}=\hat{n}_{i,j}/n^{(r)}$.
		\REPEAT
		\STATE {\bfseries Step 1:} Obtain $\hat{\Ucal}_j^b$, $\hat{\Ucal}_j^{non}$, $\hat{\Vcal}_{i,j}$ and $\hat{i}^*(\x_j)$ for $i=1,2,...,k$, $i\neq \hat{i}^*(\x_j)$ and $j=1,2,...,m$. Let $(\hat{i}_*,j^r)\in\argmin_{j\in\{1,2,...,m\},i\in\{1,...,k\}\setminus\{\hat{i}^*(\x_j)\}}\hat{\Vcal}_{i,j}$.
		\STATE {\bfseries Step 2:} If $\hat{\Ucal}_{j^r}^b<\hat{\Ucal}_{j^r}^{non}$, $i^r=\hat{i}^*(\x_{j^r})$; otherwise $i^r=\hat{i}_*$. Provide one more replication to treatment $i^r$ under context $\x_{j^r}$. Update $\bar{Y}_{i^r} (\x_{j^r})$ and $\hat{\sigma}^2_{i^r} (\x_{j^r})$.
		\STATE {\bfseries Step 3:} Update $\hat{n}_{i,j}$, $n^{(r+1)}$ and $\hat{\alpha}_{i,j}$. $r\leftarrow r+1$.
		\UNTIL{$n^{(r)}=n$.}
	\end{algorithmic}
\end{algorithm}

At the beginning of the algorithm, we simulate each treatment-context pair for the same number of replications and acquire initial estimates for their means and variances. In each of the subsequent iterations, we sample more on a certain treatment-context pair determined by $\hat{\Ucal}_j^b$, $\hat{\Ucal}_j^{non}$ and $\hat{\Vcal}_{i,j}$, and update its sample mean and sample variance. Although we have set the incremental budget $\Delta n=1$ in this generic algorithm, in practice, $\Delta n$ can be larger than 1, to reduce the number of iterations. The algorithm terminates when the total simulation budget is exhausted.

This idea for designing CR\&S Algorithm 1 does not involve solving the set of nonlinear equations (\ref{eq:add4}) and (\ref{eq:add6}) and is thus cost-effective; more importantly, this algorithm can recover the optimality conditions (\ref{eq:add4}) and (\ref{eq:add6}). It can be established in the following theorem.

\begin{theorem}\label{theo:4}
	Suppose Assumptions \ref{ass:1}-\ref{ass:3} hold. For $\hat{\alpha}_{i,j}$ generated by CR\&S Algorithm 1, $i=1,2,...,k$ and $j=1,2,...,m$, we have almost surely that
	\begin{align*}
		&\lim_{r \to \infty} \left| \frac{\hat{\alpha}_{i^*(\x_j),j}^2}{\sigma_{i^*(\x_j)}^2(\x_j)}-
		\sum_{i=1,i\neq i^*(\x_j)}^k \frac{\hat{\alpha}_{i,j}^2} {\sigma_{i}^2(\x_j)} \right| = 0, \ \ j=1,2,...,m, \\
		&\lim_{r \to \infty} \left| \frac{(y_i (\x_j)-y_{i^*(\x_j)} (\x_j))^2} {\sigma_{i^*(\x_j)}^2(\x_j)/\hat{\alpha}_{i^*(\x_j),j}+\sigma_i^2(\x_j)/\hat{\alpha}_{i,j}} -
		\frac{(y_{i'} (\x_j)-y_{i^*(\x_j)} (\x_j))^2} {\sigma_{i^*(\x_j)}^2(\x_j)/\hat{\alpha}_{i^*(\x_j),j}+\sigma_{i'}^2(\x_j)/\hat{\alpha}_{i',j}} \right| = 0,  \ \ j,j'=1,...,m,\nonumber\\
		& i,i'=1,...,k, \ i\neq i^*(\x_j) \text{ and } i'\neq i^*(\x_{j'}).
	\end{align*}
\end{theorem}
In other words, when $\alpha_{i,j}$'s are replaced by the sample allocation $\hat{\alpha}_{i,j}$'s generated by CR\&S Algorithm 1, conditions (\ref{eq:add4}) and (\ref{eq:add6}) still hold almost surely as the iteration $r \rightarrow \infty$. A byproduct of this theorem is that the number of simulation replications $\hat{n}_{i,j} = \hat{\alpha}_{i,j}n$ allocated to treatment $i$ under context $j$ by the algorithm will go to infinity as the total budget $n$ goes to infinity. It ensures that all the estimators in this algorithm, such as $\bar{Y}_{i} (\x_j)$, $\hat{\sigma}^2_{i} (\x_{j})$, $\hat{i}^*(\x_j)$, etc., will converge to their true values. Particularly, the estimated best treatment $\hat{i}^*(\x_j)$ will converge to the true best $i^*(\x_j)$ in the long term for all patient contexts $\x_j$, $j=1,...,m$.

\section{Large-Scale Problem}
\label{sec:lin}

In this section, we consider the large-scale problem. Suppose the contexts lie on a grid, and the relationship between treatment performance $y_i(\x)$ and context $\x$ can be described by the linear model
\begin{align*}
	y_i(\x) = \bbf(\x)^\top \bbeta_i, \quad i = 1,2,\dots,k,
\end{align*}
where $\bbeta_i=(\beta_{i1},\ldots,\beta_{iq})^\top$ is a vector of unknown parameters that need to be estimated and $\bbf(\x)=(\uf_{1}(\x),\ldots,\uf_{q}(\x))^\top$ is a vector of known basis functions. A common choice of $\uf_{i}(\x)$'s is $\uf_{i}(\x)=\x$, $i=1,2,...,q$. They can also be set as other functional forms to achieve a potential better fit. Although the linear models are simple and require the knowledge of adequate basis functions, they are robust to model misspecification and often have good performance in prediction \citep{thompson1982regression,james2013introduction}.

The large-scale problem looks similar to but is in essence different from a class of BAI problems known as linear bandits \citep{soare2014best,jedra2020optimal}. In linear bandits, it is assumed that treatment $i$ can be represented by a characteristic vector $\v_i$, and the mean performance of treatment $i$ is given by $\v_i^\top \bbeta_L$, where $\bbeta_L$ is a vector of unknown parameters. In other words, linear bandits are only concerned with one linear model, with independent variables of the model representing information of the treatments. No contexts are involved in linear bandits. Our large-scale problem is concerned with $k$ linear models, with independent variables of the models representing information of the contexts.

For the large-scale problem, we only need to simulate the treatments under a small fraction of contexts. Suppose the contexts we simulate are $\x_h^\circ$, $h=1,\dots,p$ and $p\ll m$. As before, $n_{i,h}$ denotes the number of simulation replications we allocate to the treatment-context pair $(i,\x_h^\circ)$. Let $\alpha_{i,h} = n_{i,h} / n$, $\alphab$ be the vector of $\alpha_{i,h}$'s and $\F=(\bbf(\x_1^\circ),\ldots,\bbf(\x_p^\circ))^\top$ be the $p \times q$ design matrix. For treatment $i$, let $\overline{\Y}_i = \left(\bar{Y}_i(\x_1^\circ),\ldots, \bar{Y}_i(\x_p^\circ)\right)^\top$ be the sample means of the treatments under the $p$ simulated contexts, and let $\overline{\bepsilon}_i = \left(\overline{\epsilon}_i(\x_1^\circ),\ldots, \overline{\epsilon}_i(\x_p^\circ)\right)^\top$ be the averaged observation errors, where $\overline{\epsilon}_i(\x_h^\circ) = \frac{1}{n_{i,h}}\sum_{l=1}^{n_{i,h}} \epsilon_{il}(\x_h^\circ)$.

We use the method of least squares to estimate ${\bbeta}_i$, i.e., $\widehat{\bbeta}_i = \left( \F^\top \F \right)^{-1} \F^\top \overline{\Y}_i$. Then, the estimate $\hat{y}_{i}(\x)$ for the mean performance $y_i(\x)$ of treatment $i$ under context $\x$ is $\bar{Y}_i^L(\x) = \bbf(\x)^\top \widehat{\bbeta}_i$.

\subsection{Rate-Optimal Budget Allocation Rule}

For the large-scale problem, optimization of the rate function $\Rcal(\alphab)$ is given by
\begin{align}\label{eq:linearR}
	\min \ -\Rcal(\alphab) \quad
	\st \  \sum_{i=1}^k \sum_{h=1}^p  \alpha_{i,h}=1, \
	\alpha_{i,h}\geq0, \ i=1,2,...,k, \ h=1,2,...,p.
\end{align}
Model \eqref{eq:linearR} has the same structure as \eqref{eq:1}. They both optimize the rate function $\Rcal(\alphab)$ subject to the simulation budget constraint. The difference is that in \eqref{eq:linearR}, mean performance $y_i(\x)$ is predicted by $\bar{Y}_i^L(\x)$ from the linear models.

By Theorem \ref{theo:1}, $\Rcal(\alphab)=\min\limits_{j\in\{1,...,m\}}\min\limits_{i\in\{1,...,k\},i\neq i^*(\x_j)}-\lim\limits_{n\rightarrow\infty}\frac{1}{n}\log \PP\left(\bar{Y}^L_{i^*(\x_j)}(\x_j)\geq\bar{Y}^L_{i} (\x_j)\right)$. Here we denote
\begin{align*}
	-\lim_{n\rightarrow\infty}\frac{1}{n}\log \PP\left(\bar{Y}^L_{i^*(\x_j)}(\x_j)\geq\bar{Y}^L_{i} (\x_j)\right) \doteq \Gcal_{i^*(\x_j),i,j}^{L}(\alphab).
\end{align*}
Obviously, $\Gcal_{i^*(\x_j),i,j}^{L}(\alphab)$ is different from the rate function $\Gcal_{i^*(\x_j),i,j}(\alpha_{i^*(\x_j),j},\alpha_{i,j})$ in Section \ref{sec:discrete}. We next derive $\Gcal_{i^*(\x_j),i,j}^{L}(\alphab)$ in the following lemma.

\begin{lemma}\label{lem:linrate}
	Suppose Assumptions \ref{ass:1}-\ref{ass:3} hold. With the linear models, the rate function of the three measures $\pfs_{\Erm}$, $\pfs_{\Mrm}$ and $\pfs_{\Arm}$ is
	$
	\Rcal(\alphab) = \min_{j\in\{1,...,m\}}\min_{i\in\{1,...,k\},i\neq i^*(\x_j)} \Gcal_{i^*(\x_j),i,j}^{L}(\alphab),
	$
	where
	\begin{align*}
		\Gcal_{i^*(\x_j),i,j}^{L}(\alphab) =  \frac{\left[ \bbf(\x_j)^\top (\bbeta_i - \bbeta_{i^*(\x_j)}) \right]^2 }{2 \bbf(\x_j)^\top \left( \F^\top \F \right)^{-1}  \F^\top \left(\bSigmap_{\epsilon,i}+\bSigmap_{\epsilon,i^*(\x_j)}\right) \F \left( \F^\top \F \right)^{-1}  \bbf(\x_j) }
	\end{align*}
	and $\bSigmap_{\epsilon,i_1}$ is the  diagonal matrix with $\left(\bSigmap_{\epsilon,i_1}\right)_{hh} = \frac{\sigma_{i_1}^2(\x_h^\circ)}{\alpha_{i_1,h}}$, $i_1=i,i^*(\x_j)$.
\end{lemma}

A model equivalent to \eqref{eq:linearR} is given by
\begin{equation}\label{eq:reprob}
	\begin{aligned}
		\max ~ & z  \\
		s.t. ~ & \Gcal_{i^*(\x_j),i,j}^{L}(\alphab) \ge z, \  i=1,2,...,k \text{ and }i\neq i^*(\x_j), j=1,2,...,m,  \\
		& \sum_{i=1}^k \sum_{h=1}^p \alpha_{i,h}=1, \  \alpha_{i,h}\geq0, ~ i=1,2,...,k, \ h=1,2,...,p.
	\end{aligned}
\end{equation}
Although \eqref{eq:reprob} is a convex optimization model, its KKT conditions cannot be easily analyzed as for its counterpart (\ref{eq:2}) in the small-scale problem. To solve \eqref{eq:reprob}, we will consider the dual problem of it. For simplicity of notation, let $\overline{\Y} = \left( \overline{\Y}^\top_1,\dots,\overline{\Y}^\top_k \right)^\top$, $\hat{\boldsymbol{\sigma}}^2_i = \left( \sigma^2_i(\x_1^\circ), \dots, \sigma^2_i(\x_p^\circ) \right)^\top$, $\hat{\boldsymbol{\sigma}}^2 = \left( \hat{\boldsymbol{\sigma}}^{2,\top}_1, \dots, \hat{\boldsymbol{\sigma}}^{2,\top}_k \right)^\top$, and $\blam$ be the vector of $\lambda_{i,j}$'s for $i=1,\dots,k$ and $i \ne i^*(\x_j)$ and $j = 1,\dots,m$. Let the mean of $(\overline{\Y},\hat{\boldsymbol{\sigma}}^2)$ be $(\boldsymbol{y},\boldsymbol{\sigma}^2)$.
\begin{theorem}\label{th:linall}
	The optimal solution to \eqref{eq:reprob} is
	\begin{align}\label{eq:alpocba}
		\alpha_{i,h} = \frac{\sqrt{\chi_{i,h}(\blam, \boldsymbol{y}, \boldsymbol{\sigma}^2)}}{ \sum_{i=1}^k \sum_{h=1}^p \sqrt{\chi_{i,h}(\blam, \boldsymbol{y}, \boldsymbol{\sigma}^2)}}, \ i=1,\dots,k, \ h=1,\dots,p.
	\end{align}
	In (\ref{eq:alpocba}), $\blam$ is the optimal solution to
	\begin{equation}\label{eq:probdual}
		\begin{aligned}
			\min_{\blam} ~ & a(\blam, \boldsymbol{y}, \boldsymbol{\sigma}^2) =  -\sum_{i=1}^k \sum_{h=1}^p \sqrt{ \chi_{i,h}(\blam, \boldsymbol{y}, \boldsymbol{\sigma}^2)  }  \\
			s.t. ~ & \sum_{j=1}^m \sum_{i=1, i\ne i^*(\x_j)}^k\lambda_{i,j}=1, \ \lambda_{i,j} \ge 0, ~i=1,\dots,k \text{ and } i \ne i^*(\x_j), \ j = 1,\dots,m,
		\end{aligned}
	\end{equation}
	where $\chi_{i,h}(\blam, \boldsymbol{y}, \boldsymbol{\sigma}^2)$ is defined as
	\begin{align}\label{eq:defchi}
		&\chi_{i,h}(\blam, \boldsymbol{y}, \boldsymbol{\sigma}^2)  \\
		=&\sigma^2_{i}(\x_{h}^\circ) \left( \sum_{j \in \Ccal_{i}} \sum_{i'=1,i'\ne i}^{k} \frac{2\lambda_{i',j} \left( \bbf(\x_{j})^\top \left( \F^\top \F \right)^{-1} \bbf(\x_{h}^\circ) \right)^2}{ \left[ \bbf(\x_j)^\top (\bbeta_i - \bbeta_{i'}) \right]^2 }   + \sum_{j \notin \Ccal_{i}} \frac{2\lambda_{i,j} \left( \bbf(\x_{j})^\top \left( \F^\top \F \right)^{-1} \bbf(\x_{h}^\circ) \right)^2}{ \left[ \bbf(\x_j)^\top (\bbeta_{i} - \bbeta_{i^*(\x_j)}) \right]^2 }    \right) \nonumber
	\end{align}
	for $i=1,\dots,k$ and $h=1,\dots,p$, and $\Ccal_i = \{j: i^*(\x_j) = i\}$.
\end{theorem}

Intuitively, allocating more replications to treatment $i$ under context $\x^\circ_h$ can increase the accuracy of estimate $\widehat{\bbeta}_i$, and the more accurate $\widehat{\bbeta}_i$ and $\widehat{\bbeta}_{i'}$ are, the more likely $\bar{Y}^L_{i}(\x_j) = \bbf(\x_j)^\top \widehat{\bbeta}_i < \bbf(\x_j)^\top \widehat{\bbeta}_{i'} = \bar{Y}^L_{i'}(\x_j)$ given $y_{i}(\x_j) < y_{i'}(\x_j)$, for all $j =1,\dots,m$. Each term in the summation of \eqref{eq:defchi} can be seen as the contribution of allocating replications to treatment $i$ under context $\x^\circ_h$ to the correct comparison between the best and non-best treatments under different contexts. Therefore, $\sqrt{\chi_{i,h}(\blam, \boldsymbol{y}, \boldsymbol{\sigma}^2)}$ can be seen as the total contribution of allocating replications to treatment $i$ under context $\x^\circ_h$ to maximizing the rate function $\Rcal(\alphab)$, and Theorem \ref{th:linall} indicates that the number of replications allocated to treatment $i$ under context $\x_h^\circ$ should be proportional to $\sqrt{\chi_{i,h}(\blam, \boldsymbol{y}, \boldsymbol{\sigma}^2)}$.

\subsection{Selection Algorithm}

In this section, we devise a selection algorithm for the large-scale problem based on Theorem \ref{th:linall}.

The parameters $\boldsymbol{y}$ and $\boldsymbol{\sigma}^2$ in Theorem \ref{th:linall} are unknown in practice and can be estimated by $\overline{\Y}$ and $\hat{\boldsymbol{\sigma}}^2$. Given $\overline{\Y}$ and $\hat{\boldsymbol{\sigma}}^2$, \eqref{eq:probdual} is a convex optimization problem, and we develop a gradient descent algorithm to find its optimal solution $\hat{\blam}$. In each iteration, we compute a descent direction $\dfrak$ and a descent stepsize $\sfrak$, and update $\hat{\blam}$ by letting it move along the direction $\dfrak$ with the stepsize $\sfrak$. Different from most gradient descent algorithms which conduct this movement for multiple times, our algorithm only conducts the movement once, and then plugs the updated $\hat{\blam}$,  $\overline{\Y}$, $\hat{\boldsymbol{\sigma}}^2$ into $\alpha_{i,h} = \frac{\sqrt{\chi_{i,h}(\blam, \boldsymbol{y}, \boldsymbol{\sigma}^2)}}{ \sum_{i=1}^k \sum_{h=1}^p \sqrt{\chi_{i,h}(\blam, \boldsymbol{y}, \boldsymbol{\sigma}^2)}}$ to compute the estimated optimal allocation $\hat{\alpha}_{i,h}$. This algorithm design considerably reduces the computation associated with gradient descent, while still ensuring that $\overline{\Y}$, $\hat{\boldsymbol{\sigma}}^2$ and $\hat{\alpha}_{i,h}$ converge to the correct values. Next, we provide a small incremental budget and allocate it to the treatment-context pairs based on $\hat{\alpha}_{i,h}$, and update $\overline{\Y}$ and $\hat{\boldsymbol{\sigma}}^2$ of the treatment-context pairs that receive additional replications. Then, the algorithm proceeds to the next iteration. This process is repeated until the simulation budget is consumed.

\begin{algorithm}[htbp]
	\caption{CR\&S Algorithm 2.}
	\label{alg:ocbaforrs}
	\begin{algorithmic}
		\STATE {\bfseries Input:} Specify the number of contexts $m$, number of treatments $k$, total simulation budget $n$ and initial number of simulation replications $n_0$. Calculate $\bbf(\x_j)$ for each context $\x_j$ and $\left( \F^\top \F \right)^{-1}$. Choose a small constant $\kappa_0$ and $\eta < \frac{1}{(k-1)m}$. Iteration counter $r \gets 0$.
		\STATE {\bfseries Initial Sampling:} Perform $n_0$ replications on each pair of treatment $i$ and context $\x_h^\circ$, calculate sample means and sample variances $\bar{Y}_{i} (\x_h^\circ)$ and $\hat{\sigma}^2_{i} (\x_{h}^\circ)$, and estimate $\bbeta_{i}$ by $\widehat{\bbeta}_i = \left( \F^\top \F \right)^{-1} \F^\top \overline{\Y}_i$. Let $\hat{n}_{i,h}=n_0$, $n^{(r)}= \sum_{i=1}^k \sum_{h=1}^p \hat{n}_{i,h}$ and $\hat{\alpha}_{i,h}=\hat{n}_{i,h}/n^{(r)}$. Find the best treatment $\hat{i}^*(\x_j) = \arg \min_i \bbf(\x_j)^\top \widehat{\bbeta}_i$ for each $\x_j$. Set $\hat{\lambda}_{i,j} = \frac{1}{(k-1)m}$, $i=1,\dots,k$, $i\ne \hat{i}^*(\x_j)$, $j=1,\dots,m$.
		\REPEAT
		\STATE {\bfseries Step 1:} $r \gets r+1$. Obtain  $\chi_{i,h} (\hat{\blam}, \overline{\Y}, \hat{\boldsymbol{\sigma}}^2)$ by plugging $\hat{\blam}$, $\overline{\Y}$, and $\hat{\boldsymbol{\sigma}}^2$ into $\chi_{i,h}(\blam, \boldsymbol{y}, \boldsymbol{\sigma}^2)$.
		\STATE {\bfseries Step 2:} Randomly choose a $(i^{r*},j^{r*})$ from $\{ (i,j): \hat{\lambda}_{i,j} \text{ exists and } \  \hat{\lambda}_{i,j} \ge \eta \}$.
		\STATE {\bfseries Step 3:} Compute the descent direction $\dfrak^{(r)} = \mathop{\arg \min}\limits_{\dfrak \in \Dcal^{(i^{r*},j^{r*})}(\hat{\blam})} \sfrak^{\max}(\dfrak,\hat{\blam}) \nabla a(\hat{\blam}, \overline{\Y}, \hat{\boldsymbol{\sigma}}^2)^\top \dfrak$, where $\Dcal^{(i',j')}(\hat{\blam}) = \{e_{i,j}-e_{i',j'}: i \ne i' \text{ or } j \ne j'\} \bigcup \{e_{i',j'} - e_{i,j}: i \ne i' \text{ or } j \ne j', \hat{\lambda}_{i,j} > 0\}$, $e_{i,j}$ is obtained by letting $\hat{\lambda}_{i,j}$ equal to one and all the other elements of $\hat{\blam}$ equal to zero, and $\sfrak^{\max}(\dfrak,\hat{\blam}) = \hat{\lambda}_{i_2,j_2}$ for $\dfrak = e_{i_1,j_1} - e_{i_2,j_2} \in \Dcal^{(i',j')}(\hat{\blam})$. Let $W^{(r)} = \nabla a(\hat{\blam}, \overline{\Y}, \hat{\boldsymbol{\sigma}}^2)^\top \dfrak^{(r)}$.
		\STATE {\bfseries Step 4:} If $W^{(r)}$ satisfies $W^{(r)} < \max\{ -\kappa_0, -(\frac{\log r}{r})^{1/4} \}$ and $\sfrak^{\max}(\dfrak^{(r)},\hat{\blam})W^{(r)} < \max\{ -\kappa_0, -(\frac{\log r}{r})^{1/2} \}$, choose $\sfrak^{(r)} = \text{LineSearch}(\dfrak^{(r)},\sfrak^{\max}(\dfrak^{(r)},\hat{\blam}),\hat{\blam}, \overline{\Y}, \hat{\boldsymbol{\sigma}}^2)$ and let $\hat{\blam} = \hat{\blam} + \sfrak^{(r)} \dfrak^{(r)}$. Otherwise, $\hat{\blam}$ remains unchanged.
		\STATE {\bfseries Step 5:} Update $\chi_{i,h}(\hat{\blam}, \overline{\Y}, \hat{\boldsymbol{\sigma}}^2)$. Compute $\hat{\alpha}^*_{i,h}$ using \eqref{eq:alpocba} with $\chi_{i,h}(\hat{\blam}, \overline{\Y}, \hat{\boldsymbol{\sigma}}^2)$.
		\STATE {\bfseries Step 6:} Choose $(i^r,h^r) = \arg\max_{(i,h)} \frac{\hat{\alpha}^*_{i,h}}{\hat{n}_{i,h}}$. Provide one more replication to treatment $i^r$ under context $\x_{h^r}^\circ$. Update $\bar{Y}_{i^r} (\x_{h^r}^\circ)$, $\hat{\sigma}^2_{i^r} (\x_{h^r}^\circ)$, $\widehat{\bbeta}_i$, and $\hat{i}^*(\x_j)$. If $\hat{i}^*(\x_j)$ is changed for any $\x_j$, set $\hat{\lambda}_{i,j} = \frac{1}{(k-1)m}$, $i=1,\dots,k$, $i\ne \hat{i}^*(\x_j)$, $j=1,\dots,m$.
		\STATE {\bfseries Step 7:} Update $\hat{n}_{i,h}$, $n^{(r)}$ and $\hat{\alpha}_{i,h}$.
		\UNTIL{$n^{(r)} = n$.}
	\end{algorithmic}
\end{algorithm}

This idea is summarized in CR\&S Algorithm \ref{alg:ocbaforrs}. Note that this way of algorithm design has appeared in the literature, e.g., in \cite{zhou2021}. The calculation of the stepsize $\sfrak$ in Step 4 of CR\&S Algorithm \ref{alg:ocbaforrs} calls for a line search, which is provided in Algorithm \ref{alg:linesear} below. For the input parameters $\sfrak_1$ and $\sfrak_2$ in Algorithm \ref{alg:linesear}, the recommended values are $10^{-4}$ and $10^{-1}$ (Chapter 3, \cite{nocedal2006numerical}).

\begin{algorithm}[htbp]
	\caption{LineSearch$(\dfrak,\sfrak^{\max},\hat{\blam}, \overline{\Y}, \hat{\boldsymbol{\sigma}}^2)$.}
	\label{alg:linesear}
	\begin{algorithmic}
		\STATE {\bfseries Initialization:} Specify the descent direction $\dfrak$, maximum feasible stepsize $\sfrak^{\max}$, dual solution $\hat{\blam}$, estimate of coefficients $\hat{\bbeta}$, parameters for line search $\sfrak_1$, $\sfrak_2$, and $\tau \in (0,1)$. Let $\sfrak=\sfrak^{\max}$.
		\STATE {\bfseries while} Any of the conditions
		\begin{align}
			\min_i \min_j \chi_{i,j}(\hat{\blam},\overline{\Y}, \hat{\boldsymbol{\sigma}}^2) &> 0, \label{eq:lin:fea}\\
			a(\hat{\blam}+\sfrak \cdot \dfrak,\overline{\Y}, \hat{\boldsymbol{\sigma}}^2) &\le a(\hat{\blam},\overline{\Y}, \hat{\boldsymbol{\sigma}}^2) +\sfrak_1 \sfrak \nabla a(\hat{\blam},\overline{\Y}, \hat{\boldsymbol{\sigma}}^2)^\top \dfrak,  \label{eq:lin:obj}\\
			\nabla a(\hat{\blam}+\sfrak\cdot \dfrak,\overline{\Y}, \hat{\boldsymbol{\sigma}}^2)^\top \dfrak & \le \sfrak_2 |\nabla a(\hat{\blam},\overline{\Y}, \hat{\boldsymbol{\sigma}}^2)^\top \dfrak|,  \label{eq:lin:der}
		\end{align}
		is not satisfied, \textbf{do} $\sfrak \leftarrow \tau \sfrak$.
		\STATE {\bfseries Output:} Stepsize $\sfrak$.
	\end{algorithmic}
\end{algorithm}

Similar to CR\&S Algorithm 1, CR\&S Algorithm \ref{alg:ocbaforrs} can recover the optimal solution to problem \eqref{eq:reprob}. This is established in the following theorem.
\begin{theorem}
	Suppose Assumptions \ref{ass:1}-\ref{ass:3} hold. For $\hat{\alpha}_{i,h}$ generated by CR\&S Algorithm \ref{alg:ocbaforrs}, we have that $\hat{\alpha}_{i,h}$ converges to the optimal solution to problem \eqref{eq:reprob} almost surely, $i=1,\dots,k$ and $h=1,\dots,p$.
\end{theorem}

\section{Numerical Experiments}
\label{sec:7}

In this section, we conduct two sets of numerical experiments. The first set tests the performances of CR\&S Algorithms 1 and 2 on a series of benchmark functions, and the second set applies them to two real-world PM problems.

\subsection{Performance Comparison on the Benchmark Functions}
\label{sec:72}

In this test, we numerically assess the performances of the CR\&S Algorithms 1 and 2 on some benchmark functions. We use the following algorithms for comparison:
\begin{itemize}
	\item \emph{Equal Allocation}. The number of simulation replications allocated to any treatment-context pair is equal. This is a naive method, and can serve as a baseline against which improvement from other methods might be measured.
	
	\item \emph{Successive Rejection with equal allocation among contexts} (Equal SR).  The original SR was designed for a single context and has been shown to be highly efficient for BAI problems with bounded sampling distributions \citep{carpentier2016tight,gabillon2012best}. In this test, we apply SR to treatments under the same context while equally distributing the simulation budget among different contexts. Under each context, the simulation budget available $n/m$ is divided into $k-1$ phases. Every treatment that has not been rejected receives $n_{(i)}-n_{(i-1)}$ more replications and the estimated worst treatment is rejected in phase $i$, $i=1,\dots,k-1$.
	
	\item \emph{Optimal computing budget allocation with equal allocation among contexts} (Equal OCBA). Similarly as SR, the original OCBA was designed for a single context. In this test, we apply OCBA to treatments under the same context while equally distributing the simulation budget among different contexts:
	\begin{align*}
		&\frac{n_{i_1,j}}{n_{i_2,j}} =  \frac{ \sigma^2_{i_1}(\x_j) }{ \sigma^2_{i_2}(\x_j) } \frac{\big( y_{i^*(\x_j)}(\x_j) - y_{i_2}(\x_j) \big)^2 }{  \big( y_{i^*(\x_j)}(\x_j) - y_{i_1}(\x_j) \big)^2 } , \ i_1,i_2 \in \{1,2,\cdots,k\}\backslash \{ i^*(\x_j) \}, \ j = 1,2,\cdots,m,  \\
		&n_{i^*(\x_j),j} = \sigma_{i^*(\x_j)}(\x_j) \left( \sum_{i=1,i\ne i^*(\x_j)}^k  \left( \frac{n_{i,j}}{\sigma_{i}(\x_j)} \right)^2  \right)^{\frac{1}{2}}, \ j = 1,2,\cdots,m, \\
		&\sum_{i=1}^k n_{i,j_1} = \sum_{i=1}^k n_{i,j_2},  \ j_1,j_2 = 1,2,\cdots,m.
	\end{align*}
	
	\item \emph{The two-stage procedure} (TS, \cite{shen2019}). TS considers R\&S in the presence of contexts and also assumes linear relationship between treatment performance and contexts as CR\&S Algorithm 2. It allocates a small fraction of the simulation budget to some selected treatment-context pairs in the first stage and based on the sample estimates, decides the number of replications these treatment-context pairs should receive in the second stage. TS is based on the IZ method. The ultimate goal of it is to make a guarantee of the quality of the selected design over the context space, instead of maximizing the quality. When stopped, the total simulation budget consumed by TS is random. To add TS into comparison, we use the allocation $\alphab_{TS}$ obtained from the first stage of TS as a reference to allocate the remaining fixed simulation budget.
	
	\item \emph{Optimal Allocation Matching (OAM, \cite{hao2020adaptive})}. OAM is an algorithm for contextual bandit problems. Suppose $y_i(\x) = \bbf(\x)^\top \bbeta_i$, $i=1,\dots,k$. OAM shows that the optimal budget allocation of contextual bandits satisfies
	\begin{align*}
		&\inf_{\alpha_{i,j} \in [0,\infty]}  \ \sum_{j=1}^m \sum_{i=1}^k \tilde{\alpha}_{i,j} (y_{i}(\x_j) - y_{i^*(\x_j)} (\x_j)), \\
		&s.t. \ \bbf(\x_j)^\top \left( \sum_{j'=1}^m \alpha_{i,j'} \bbf(\x_{j'}) \bbf(\x_{j'})^\top  \right)^{-1} \bbf(\x_j) \le \frac{(y_{i}(\x_j) - y_{i^*(\x_j)} (\x_j))^2}{2}, \forall j \ne i^*(\x_j), \ i=1,\dots,k.
	\end{align*}
	Intuitively, the left-hand side of the constraint represents the width of the confidence interval to compare $\bbf(\x_j)^\top \widehat{\bbeta}_i$ and $\bbf(\x_j)^\top \widehat{\bbeta}_{i^*(\x_j)}$. In each iteration of OAM, with the given context, it decides which treatment to sample based on an approximated optimal allocation. As discussed in Section 1, contextual bandit algorithms (including OAM) do not decide which context to sample. In this test, we set each context to be sampled with the same probability of $1/m$ for OAM.
\end{itemize}

The test will be conducted on the benchmark functions below, where $\x=(x_1,\ldots,x_d)^\top\in \RR^d$ is the context, $\z=(z_1,\ldots,z_d)^\top\in \RR^d$ is the solution for the benchmark function (treatments in PM), and $\epsilon$ is a normally distributed noise that is independent across different solutions, contexts and simulation replications.
\begin{enumerate}
	\item Sphere function:
	$ Y(\z,\x) = f(\z,\x) + \epsilon =  \sum_{l=1}^{d} (z_l-x_l)^2+ \epsilon. $
	The global minimum of $f(\z,\x)$ is 0 obtained at $\z=\x$. We consider the one dimensional case ($d=1$) of this problem with 4 contexts $\x \in \{-0.45, \ -0.15, \  0.15, \ 0.45\}$ and 11 solutions $\z \in \{ -1.25, \ -1.00, \ -0.75, \  \cdots, 1.25 \}$. The noise $\epsilon$ follows the normal distribution $N(0, 0.05 )$.
	\item Rosenbrock function:
	$ Y(\z,\x) = f(\z,\x)+ \epsilon =  \sum_{l=1}^{d-1} \Big[ 100\big((z_{l+1}-x_{l+1}) - (z_l-x_l)^2\big)^2 + \big( 1-(z_l-x_l) \big)^2 \Big]  + \epsilon. $
	The global minimum of $f(\z,\x)$ is 0 obtained at $z_l = x_l + 1$, $l = 1,2,\cdots,d$, $d >1$. We consider the two dimensional case ($d=2$) of this problem with 25 contexts $\x \in \{-0.30, \ -0.15, \ 0, $ $0.15, \ 0.30 \} \times \{-0.30, \ -0.15, \ 0, \ 0.15, \ 0.30 \}$ and 9 solutions $\z \in \{ 0, \  0.75, \ 1.5 \} \times \{ 0, \ 0.75, \ 1.5 \}$. The noise $\epsilon$ follows the normal distribution $N(0, 2.25 )$.
	\item Randomly generated linear functions:
	$ Y(\z,\x) = f(\z,\x)+ \epsilon = \bbeta(\z)^\top (1,\x^\top)^\top + \epsilon, $
	where components of $\bbeta(\z)$ are randomly generated from $\text{Unif}[0,5]$ and $\epsilon$ follows the normal distribution $N(0, 1)$. We consider context space dimensions $d=1$ and $d=3$, $4^d$ contexts $\x \in \{0,\frac{1}{3},\frac{2}{3},1\}^d$ and 5 solutions $\z \in \{ 1, \ 2,  \cdots, 5 \}$.
\end{enumerate}

In the first two examples, we modified the original benchmark functions $f(\z)$ to $f(\z-\x)$ to incorporate context $\x$. These two examples align with the structure of the small-scale problem, and will be used to compare Equal Allocation, Equal SR, Equal OCBA, OAM and CR\&S Algorithm 1. The third example is not a typical benchmark function. It is built with a linear structure that aligns with the large-scale problem. It will be used to compare Equal Allocation, Equal Allocation, OAM, TS, and CR\&S Algorithms 1 and 2.

\begin{figure}
	\centering
	\caption{Comparison on the benchmark functions}
	\begin{minipage}{19.5cm}
		\hspace*{-2cm}\includegraphics[width=19.5cm,height = 5cm]{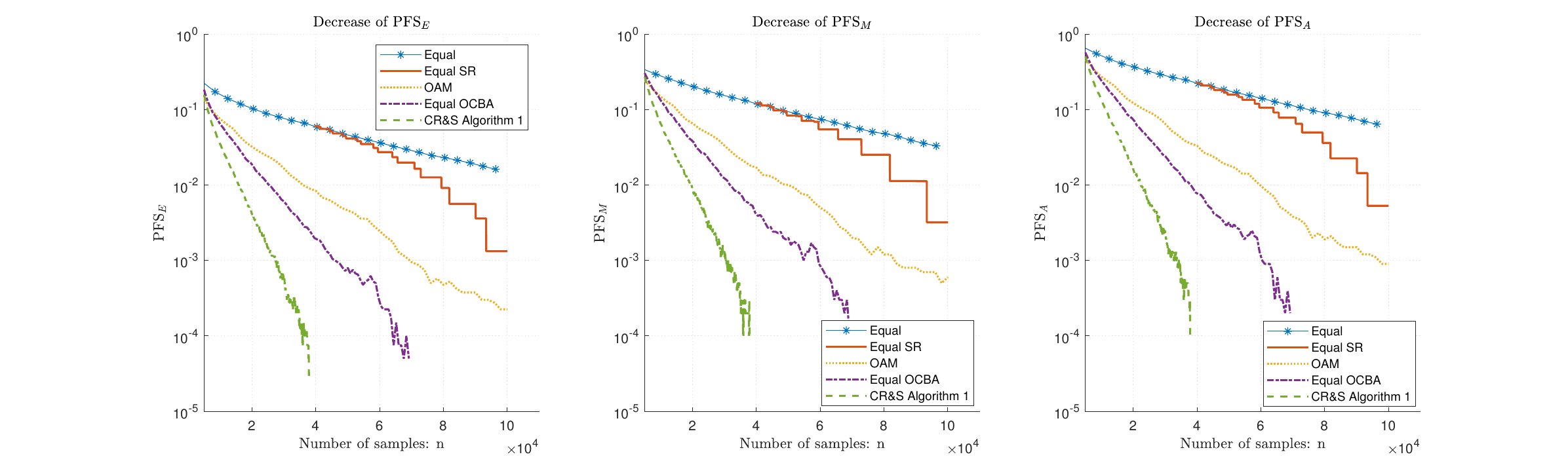}
	\end{minipage}
	
	\begin{minipage}{19.5cm}
		\hspace*{-2cm}\includegraphics[width=19.5cm,height = 5cm]{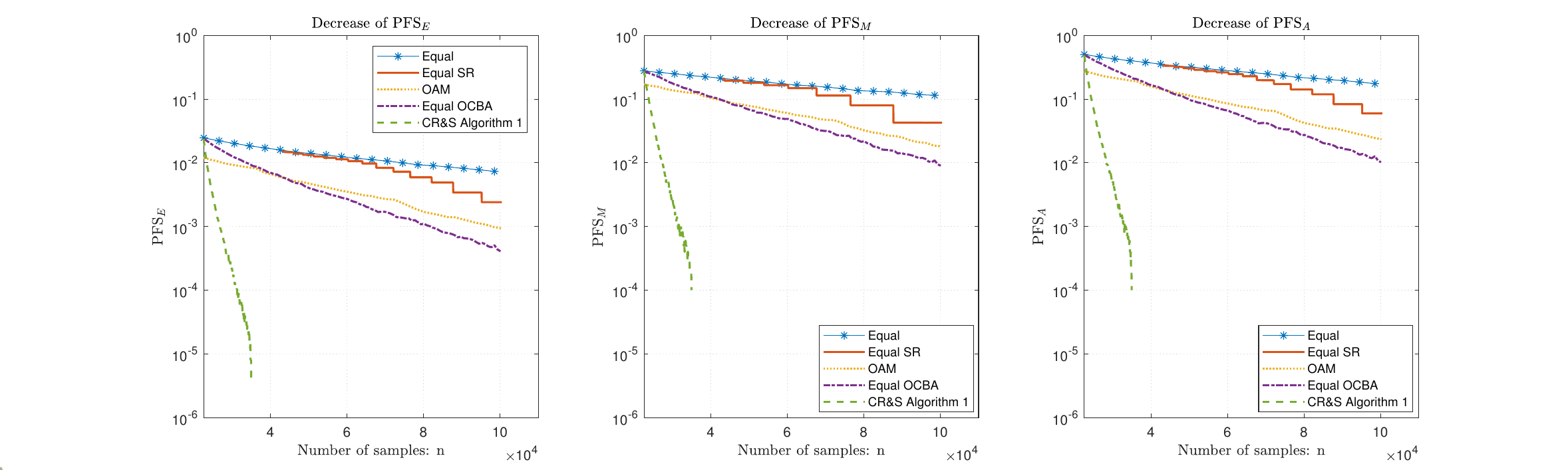}
	\end{minipage}
	
	\begin{minipage}{19.5cm}
		\hspace*{-2cm}\includegraphics[width=19.5cm,height = 5cm]{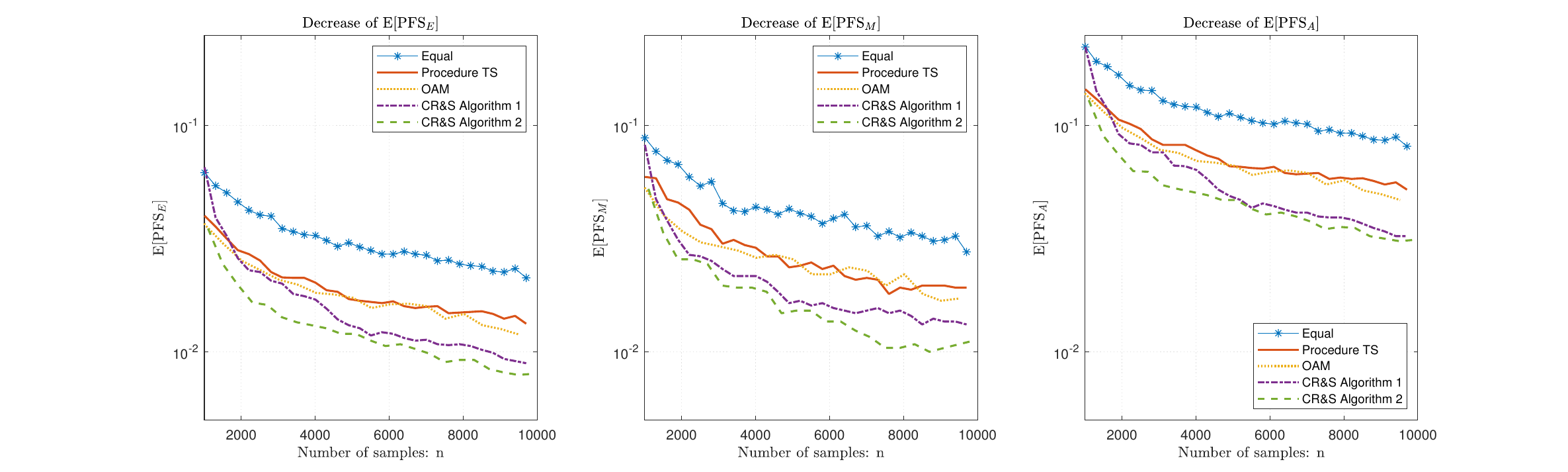}
	\end{minipage}
	
	\begin{minipage}{19.5cm}
		\hspace*{-2cm}\includegraphics[width=19.5cm,height = 5cm]{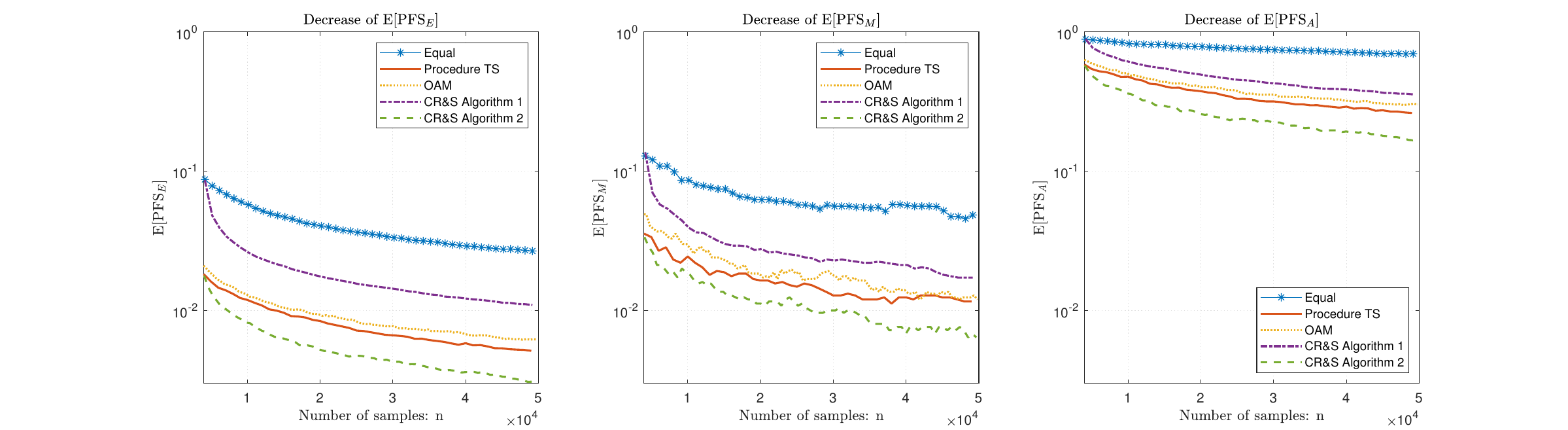}
	\end{minipage}
	
	\flushleft
	\vspace{0.2cm}
	{\small \emph{Notes. From top to bottom, benchmark functions being tested are the sphere function, Rosenbrock function, and 1-dimensional and 3-dimensional randomly generated linear functions.}}
	\label{CTest}
\end{figure}

We assess the average performances of the compared algorithms based on $10^4$ macro-replications for the sphere and Rosenbrock functions and $2,500$ macro-replications for the randomly generated linear functions. Figure \ref{CTest} shows the comparison result under different simulation budgets. The four rows in Figure \ref{CTest} correspond to the the sphere function, Rosenbrock function, and 1-dimensional and 3-dimensional randomly generated linear functions. The three columns correspond to the three measures under study. Since the linear functions in each macro-replication are randomly generated instead of being fixed, the average performances of the algorithms in the third and fourth rows are with respect to randomness from the function instances and simulation noises.

The proposed CR\&S Algorithm 1 performs the best under the sphere and Rosenbrock functions and the three measures, followed by Equal OCBA, OAM and Equal SR. The advantage of CR\&S Algorithm 1 is particularly big under the Rosenbrock function. Although Equal OCBA and Equal SR have been shown to be efficient for R\&S problems, they do not have any good mechanisms to balance the budgets allocated among contexts, causing the performances of them inferior to CR\&S Algorithm 1. The Equal Allocation performs the worst.

Under the two linear functions, CR\&S Algorithm 2 performs the best. When the context space is 1-dimensional, CR\&S Algorithm 1 outperforms TS and OAM. The goal of TS is to guarantee the quality of the estimated best treatment under each context. The budget allocation of it is not so effective in optimizing the quality of the estimated best treatments as CR\&S Algorithm 1. OAM lacks efficient mechanisms to balance the simulation budgets allocated among contexts. When the context space is 3-dimensional, TS and OAM outperform CR\&S Algorithm 1. Although the budget allocations of TS and OAM are not optimal for the large-scale problem, they have a major advantage over CR\&S Algorithm 1 in that they utilize prediction models. This advantage becomes more obvious when the total number of contexts is larger, as with the 3-dimensional context space. The Equal Allocation again performs the worst.

\subsection{Case Studies}

In this test, we apply our proposed algorithms to two real-world PM problems, namely the prevention of cervical cancer \citep{Levinvac2015} and treatment of chronic obstructive pulmonary disease \citep{hoogendoorn2019broadening,corro2020address}. Due to the space limitation, here we will only provide the numerical results for the cervical cancer example, and leave the test of the chronic obstructive pulmonary disease in Section \ref{sec:copd} of Appendix.

Cervical cancer is the fifth leading cause of cancer globally. Unlike most cancers, cervical cancer has only one direct cause: the human papillomavirus (HPV), and is thus preventable. While the widespread screening practice has led to a dramatic decrease in the cervical cancer mortality in developed countries, the cost of it is prohibitive, especially for women with low income in developing countries \citep{Levinvac2015}.

The incidence rate of cervical cancer evolves with the age and reaches the peak at around 45 \citep{globocan}. There are two ways to prevent the progression of it: the traditional screening and the newly-invented HPV vaccine. Traditional screening programs (distinguished by the frequency) conduct the examination at regular time points \citep{maclayvac2010}, including no screening, low-frequency screening (two times per lifetime at age 35 and 45), and high-frequency screening (one screening every three years from 30 to 60).

The HPV vaccine came to the market in recent years and is usually expensive. Despite of the high price, the vaccine could effectively prevent the infection of the most risky types of HPV (e.g. HPV 16/18), and the immunization period is life-long. The perfect time for HPV vaccination is before the start of any sexual behaviors (usually at age 12) \citep{westravac2011}. The decision on HPV vaccination is a tradeoff between the current economic loss and future risk. Vaccination or not, combined with the screening policy, forms six possible treatment methods: HPV vaccination alone, HPV vaccination with a low-frequency screening, HPV vaccination with a high-frequency screening, low-frequency screening alone, high-frequency screening alone, and no-prevention.

\begin{figure}[p]
	\centering
	\includegraphics[width=15cm]{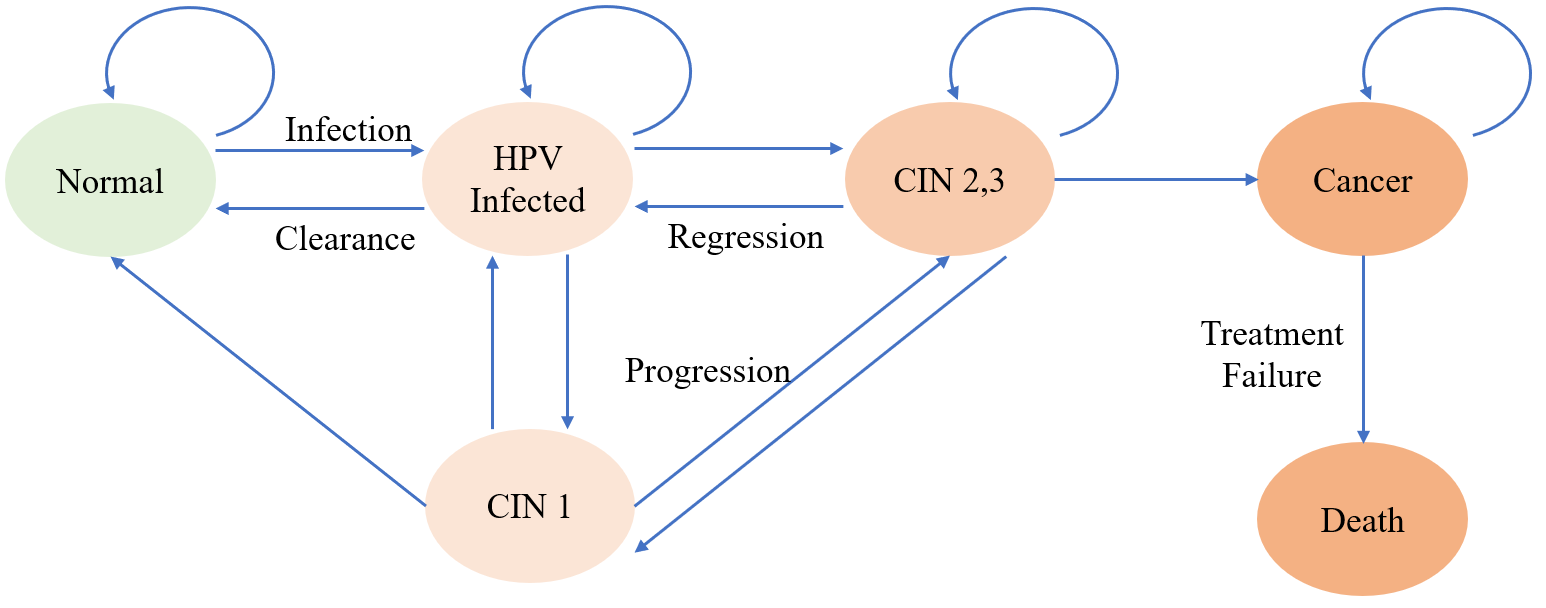}
	\caption{Simulation model for the cervical cancer} \label{vaccinediag}
	\flushleft
	\vspace{0.2cm}
	{\small \emph{Notes. This figure is adapted from \cite{Levinvac2015}. Each state may suffer from age-related all-cause mortality. ``Normal" state means the individual is not infected; ``CIN 1" means the individual has cervical intraepithelial neoplasia at grade 1; ``CIN 2,3" means the individual has cervical intraepithelial neoplasia at grade 2 or 3.}}
\end{figure}

\begin{figure}[p]
	\centering
	\includegraphics[width=15cm]{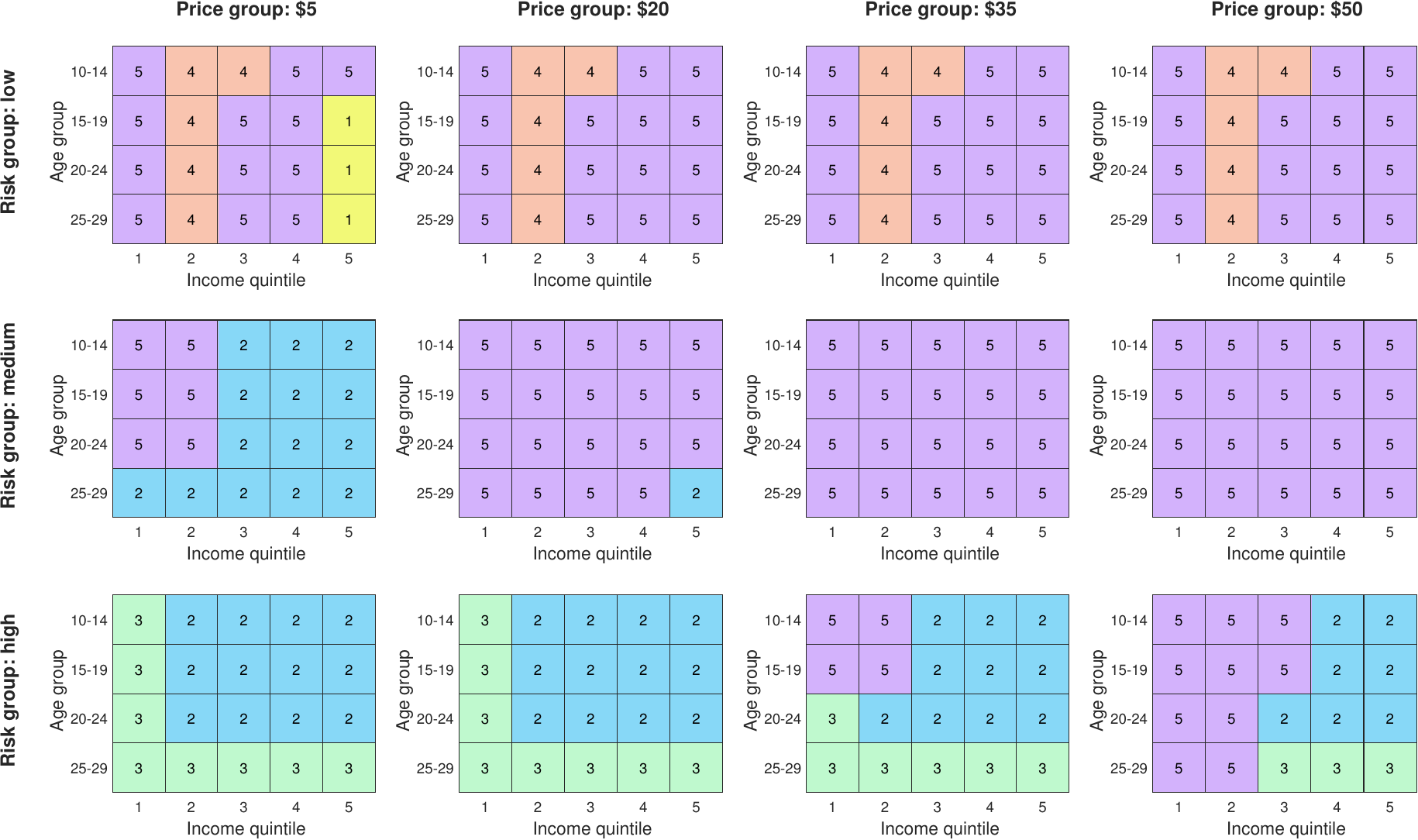}
	\caption{Medical decision map for the cervical cancer prevention problem } \label{vaccine}
	\flushleft
	\vspace{0.2cm}
	{\small \emph{Notes. In each sub-figure, the horizontal and vertical axes represent the income quintiles and age groups respectively. Numbers in cells show the best treatment method under different contexts. Specifically, numbers 1-6 mean HPV vaccination alone, HPV vaccination with a low-frequency screening, HPV vaccination with a high-frequency screening, low-frequency screening alone, high-frequency screening alone, and no-prevention.} }
\end{figure}

The simulation model of this problem is developed based on the Markov chain in Figure \ref{vaccinediag}. Cervical intraepithelial neoplasia at grade 1 (CIN 1) and its deteriorating grades (CIN 2,3) result from the human papillomavirus (HPV) infection, and they may regress to the normal state. However, once the lesions are at CIN 2,3 and are not detected, the illness would develop into cancer at substantial risk.

The context variables we consider include income, age and HPV progression risk of the patients and the price of HPV vaccine. The vaccine is assumed to have four possible prices: \$5, \$20, \$35, and \$50. Income is classified into five levels, representing the five income quintiles in a population. Age is classified into four five-year groups (11-15, 16-20, 21-25, 26-30). The HPV progression risk has low, medium and high levels, corresponding to different multipliers on the baseline progression rate. As a result, we have 240 possible contexts. The cancer treatment cost and state transition probabilities for each context are determined based on \cite{Levinvac2015}. The treatment performance is measured by the expected cost per quality-adjusted life years (QALY).

We apply CR\&S Algorithm 1 to this problem. The algorithm determines the number of simulation replications for each treatment-context pair and based on it, estimates the best treatment under each context. The result is reported in Figure \ref{vaccine}. It can be observed that the HPV vaccine with a price higher than \$20 is the best treatment for high-risky women only. The cost per QALY of it is too high for medium and low-risky women. The best treatment for individuals with high income is mostly vaccination-based. In terms of age, the best treatment for individuals at age 11-20 does not involve vaccination, while for individuals at age 21-30, the best treatment becomes vaccination-based. This is because a female individual is most likely to get the cervical cancer between 30-60. When they are at age 11-20, they are not exposed to the high risk of it, and there is no need for HPV vaccination. When they are at age 21-30, although there is a certain probability of failure in getting immunization from the vaccines, the cancer prone period that is coming soon makes the vaccination-based treatment methods the best choices for them.

\section{Conclusions and Discussion}
\label{sec:con}

In this study, we consider the problem of personalized medicine. We adopt the tool of simulation for assessing the performances of the treatment methods, and aim to efficiently utilize the computing time to select the best treatment for each patient context that might appear. To do so, we start by introducing three measures for evaluating the evidence of correct selection over the context space and showing that these measures have the same convergence rate function. Next, we propose two simulation budget allocation models that are appropriate for small and large context spaces. For the two models, we identify the rate-optimal budget allocation rules that optimize the rate function, develop convenient selection algorithms for implementation and show the consistency of the algorithms. A series of numerical experiments on benchmark functions and real-world problems demonstrate the superior empirical performances of the proposed algorithms.

In this research, we have focused on the one-time treatment, where only one treatment decision is expected to be made for the patients. In practice, there is a class of PM problems that require multiple decisions during the progression of the disease, and the goal is to find the optimal treatment policy that maximizes the cumulative rewards over the decision periods \citep{negoescu2018dynamic,lee2019optimal}. These problems are based on more complex context and decision structures, and our proposed CR\&S algorithms cannot be applied in general. This is a good future research direction. In terms of methodology, we have solved the PM problem based on the OCBA method. Recently, \cite{russo2020simple} proposed three simple context-free Bayesian algorithms under a top-two framework for BAI, which have been shown to have nice theoretical properties and empirical performances. We believe it is also a promising research avenue to extend the top-two framework and algorithms to the PM problems.

\newpage
\appendix
{\noindent \LARGE \textbf{Appendix}}

This document provides additional numerical results and proofs of the theorems for the paper ``A Contextual Ranking and Selection Method for Personalized Medicine".

\section{Case study: Chronic Obstructive Pulmonary Disease}\label{sec:copd}

More than 2\% of the total population worldwide suffers from chronic obstructive pulmonary disease (COPD). Symptoms of COPD include long-term breathlessness, cough, and sputum production. The progression of COPD is described in Figure \ref{copddiag}. A COPD patient faces three adverse events in the health state transition: exacerbation, pneumonia, and death. The transition is random and depends on the current health state of the patient. If the patient can survive an adverse event, it is still possible for him/her to face the recurrence of the same event. Thus, the occurrence of events divides a patient's life into irregular and random time intervals. Parameters in the distributions of the time intervals can be predicted by the patient's health state via regression models provided in \cite{hoogendoorn2019broadening,corro2020address}.

\begin{figure}[p]
	\centering
	\includegraphics[width=14cm]{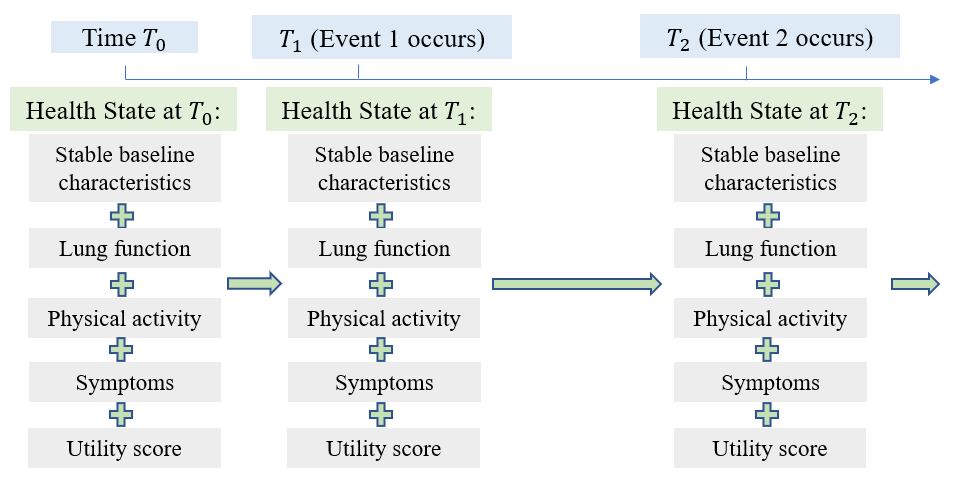}
	\caption{Simulation model for the chronic obstructive pulmonary disease } \label{copddiag}
	\flushleft
	\vspace{0.2cm}
	{\small \emph{Notes. This figure is adapted from \cite{hoogendoorn2019broadening}. The ``stable baseline characteristics" include age, number of packs smoked each year, BMI, and gender.}}
\end{figure}

\begin{figure}[p]
	\centering
	\includegraphics[width=15cm]{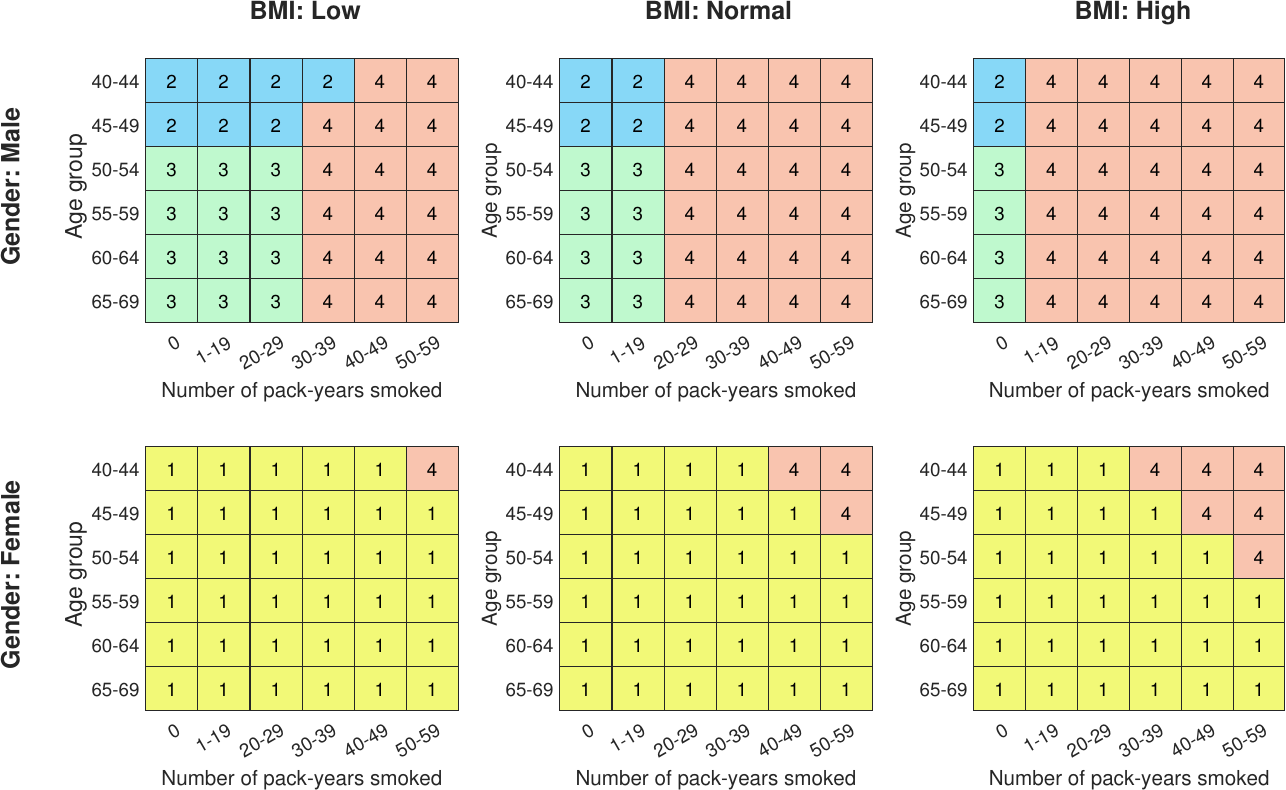}
	\caption{Medical decision map for the chronic obstructive pulmonary disease } \label{COPDfig}
	\flushleft
	\vspace{0.2cm}
	{\small \emph{Notes. In each sub-figure, the horizontal and vertical axes represent the number of pack-years smoked and age group. Numbers in cells show the best treatment method under different contexts. Specifically, numbers 1-4 mean reducing the decline rate in lung function by 10\%, increasing the time to exacerbation by 60\%, improving the physical activity level by 2.5 points, and reducing the probability of having cough/sputum by 50\%.} }
\end{figure}

So far, COPD has no cure, making proper health management especially important. Four treatment methods can be adopted to improve the patients' quality of life \citep{hoogendoorn2019broadening,corro2020address}: reducing the decline rate in lung function by 10\%, increasing the time to exacerbation by 60\%, improving the physical activity level by 2.5 points, and reducing the probability of having cough/sputum by 50\%. Let $\mathbf{X}=(X_1,X_2,X_3,X_4)^{\mathrm{T}}$ denote the context variables of the disease, where $X_1$ is the initial age of developing into COPD (an older age has a higher natural death rate), $X_2$ is the number of packs smoked each year, $X_3$ is the BMI (patients with a higher BMI usually have a better lung function), and $X_4$ is the gender. According to \cite{corro2020address}, the number of packs smoked each year could be 0 (corresponds to non-smokers), 1-19, 20-29, 30-39, 40-49, and 50-59. Age is partitioned into six five-year groups: 40-44, 45-49, 50-54, 55-59, 60-64, 65-69. The BMI has three classes: low, normal and high, and the gender can be male and female. In total, there are 216 contexts. The state transition probabilities are taken from \cite{hoogendoorn2019broadening} and \cite{corro2020address}. The treatment performance is measured by QALY.

We apply CR\&S Algorithm 2 to this problem. The algorithm determines the number of simulation replications for treatments under some selected constexts and based on it, estimates the best treatment under each context. The result is reported in Figure \ref{COPDfig}. It can be observed that for male patients, the best treatment varies a lot among the other contexts. For male patients who smoke more than 30 packs each year, the best treatment is reducing the probability of having cough/sputum, while for the rest male patients, the best treatment is increasing the time to exacerbation if the patient is younger (40-49), and is improving the physical activity level if the patient is older (50-69). For female patients, the treatment of reducing the decline rate in lung function is the best for almost all contexts, except that for female patients who are younger (40-44) and smoke more (the number of packs smoked each year is 50-59), the best treatment is reducing the probability of having cough/sputum. When the BMI of the patient if high, the best treatment is mostly reducing the probability of having cough/sputum. When the BMI is low or normal, the best treatment can be any of the four methods depending on the contexts.

\section{Proof of Theorems}

This section provides the proof of Lemma 1 and Theorems 1-5 in the main paper. 

\subsection{Proof of Theorem 1}

\begin{lemma}\label{lem:1}
	(Principle of the slowest term \citep{ganesh2004}) Consider positive sequences $a_{j}(n)$, $j=1,...,m$. If $\lim\limits_{n\rightarrow\infty}\frac{1}{n}\log a_{j}(n)$ exists for all $j$, then $\lim\limits_{n\rightarrow\infty}\frac{1}{n}\log ( \sum\limits_{j=1}^{m} a_{j}(n))=\max\limits_{j\in\{1,...,m\}}(\lim\limits_{n\rightarrow\infty}\frac{1}{n}\log a_{j}(n))$.
\end{lemma}

%lemma2 causes some filing problem

\begin{lemma}\label{lem:2}
	\citep{hunter2013} Consider positive sequences $a_i(n)$, $i=1,2,...,k$. If $\lim_{n\rightarrow\infty}\frac{1}{n}\log a_{i}(n)$ exists for all $i$, then $\max\limits_{i\in\{1,...,k\}}(\lim\limits_{n\rightarrow\infty}\frac{1}{n}\log a_{i}(n))=
	\lim\limits_{n\rightarrow\infty}\frac{1}{n}\log(\max\limits_{i\in\{1,...,k\}}a_{i}(n)).$
\end{lemma}

We first analyze $\pfs_{\Erm}$.
{\small \begin{align}
		&\lim_{n\rightarrow\infty}\frac{1}{n}\log \pfs_{\Erm}= \lim_{n\rightarrow\infty}\frac{1}{n}\log \Big(\sum_{j=1}^m p_j \pfs(\x_j)\Big)
		=\lim_{n\rightarrow\infty}\frac{1}{n}\log \left(\sum_{j=1}^m p_j \PP\Bigg(\bigcup_{i\neq i^*(\x_j)} \Big(\hat{y}_{i^*(\x_j)} (\x_j)\geq\hat{y}_{i} (\x_j)\Big)\Bigg)\right)\nonumber\\
		=&\max_{j\in\{1,...,m\}} \lim_{n\rightarrow\infty}\frac{1}{n}\log \left( p_j \PP\Bigg(\bigcup_{i\neq i^*(\x_j)} \Big(\hat{y}_{i^*(\x_j)} (\x_j)\geq\hat{y}_{i} (\x_j)\Big)\Bigg)\right)\nonumber\\
		=&\max_{j\in\{1,...,m\}} \lim_{n\rightarrow\infty}\frac{1}{n}\log \PP\Bigg(\bigcup_{i\neq i^*(\x_j)} \Big(\hat{y}_{i^*(\x_j)} (\x_j)\geq\hat{y}_{i} (\x_j)\Big)\Bigg).\label{eq:a1}
\end{align}}
The penultimate step is from Lemma \ref{lem:1}. Note that
{\small \begin{align*}
		1
		\leq \PP\Bigg(\bigcup_{i\neq i^*(\x_j)} \Big(\hat{y}_{i^*(\x_j)} (\x_j)\geq\hat{y}_{i} (\x_j)\Big)\Bigg) / \max_{i\neq i^*(\x_j)}\PP\Bigg(\hat{y}_{i^*(\x_j)} (\x_j)\geq\hat{y}_{i} (\x_j)\Bigg)
		\leq (k-1),
\end{align*}}
and that by Lemma \ref{lem:2},
{\small \begin{align}
		&\lim_{n\rightarrow\infty}\frac{1}{n}\log \left((k-1)\max_{i\neq i^*(\x_j)}\PP\Bigg(\hat{y}_{i^*(\x_j)} (\x_j)\geq\hat{y}_{i} (\x_j)\Bigg)\right) \nonumber \\
		=&\lim_{n\rightarrow\infty}\frac{1}{n}\log \max_{i\neq i^*(\x_j)}\PP\Bigg(\hat{y}_{i^*(\x_j)} (\x_j)\geq\hat{y}_{i} (\x_j)\Bigg)
		=\max_{i\neq i^*(\x_j)} \lim_{n\rightarrow\infty}\frac{1}{n}\log \PP\Bigg(\hat{y}_{i^*(\x_j)} (\x_j)\geq\hat{y}_{i} (\x_j)\Bigg) \nonumber \\
		=&\max_{i\neq i^*(\x_j)} -\Gcal_{i^*(\x_j),i,j}(\alpha_{i^*(\x_j),j},\alpha_{i,j})
		=-\min_{i\neq i^*(\x_j)}\Gcal_{i^*(\x_j),i,j}(\alpha_{i^*(\x_j),j},\alpha_{i,j}). \label{eq:loglim}
\end{align}}
Then,
{\small \begin{equation}\label{eq:a2}
		\lim_{n\rightarrow\infty}\frac{1}{n}\log \PP\Bigg(\bigcup_{i\neq i^*(\x_j)} \Big(\hat{y}_{i^*(\x_j)} (\x_j)\geq\hat{y}_{i} (\x_j)\Big)\Bigg)= -\min_{i\neq i^*(\x_j)}\Gcal_{i^*(\x_j),i,j}(\alpha_{i^*(\x_j),j},\alpha_{i,j}).
\end{equation}}
By (\ref{eq:a1}) and (\ref{eq:a2}), we have
{\small $
	\lim\limits_{n\rightarrow\infty}\frac{1}{n}\log \pfs_{\Erm}=-\min\limits_{j\in\{1,...,m\}}\min\limits_{i\neq i^*(\x_j)}\Gcal_{i^*(\x_j),i,j}(\alpha_{i^*(\x_j),j},\alpha_{i,j}).
	$}

We next consider $\pfs_{\Mrm}$. Since $\pfs_{\Mrm} = \max_{j\in\{1,...,m\}}\pfs(\x_j)$, we have
{\small \begin{align}
		&\lim_{n\rightarrow\infty}\frac{1}{n}\log \pfs_{\Mrm}
		=\lim_{n\rightarrow\infty}\frac{1}{n}\log \left(\max_{j\in\{1,...,m\}} \PP\Bigg(\bigcup_{i\neq i^*(\x_j)} \Big(\hat{y}_{i^*(\x_j)} (\x_j)\geq\hat{y}_{i} (\x_j)\Big)\Bigg)\right)\nonumber\\
		=&\max_{j\in\{1,...,m\}} \lim_{n\rightarrow\infty}\frac{1}{n}\log\PP\Bigg(\bigcup_{i\neq i^*(\x_j)} \Big(\hat{y}_{i^*(\x_j)} (\x_j)\geq\hat{y}_{i} (\x_j)\Big)\Bigg)
		=-\min_{j\in\{1,...,m\}} \min_{i\neq i^*(\x_j)}\Gcal_{i^*(\x_j),i,j}(\alpha_{i^*(\x_j),j},\alpha_{i,j}), \nonumber
\end{align}}
where the penultimate step is from Lemma \ref{lem:2} and the last step is from (\ref{eq:a2}).

Last, we consider
$
\pfs_{\Arm}=\PP\left(\bigcup\limits_{j=1}^m\bigcup\limits_{i\neq i^*(\x_j)} \Big(\hat{y}_{i^*(\x_j)} (\x_j)\geq \hat{y}_{i} (\x_j)\Big)\right).
$
Similarly as in the analysis for $\pfs_{\Erm}$,
\begin{align*}
	&\max_{j\in\{1,...,m\}}\max_{i\neq i^*(\x_j)}\PP\Bigg(\hat{y}_{i^*(\x_j)} (\x_j)\geq\hat{y}_{i} (\x_j)\Bigg)
	\leq\PP\left(\bigcup_{j=1}^m\bigcup_{i\neq i^*(\x_j)} \Big(\hat{y}_{i^*(\x_j)} (\x_j)\geq \hat{y}_{i} (\x_j)\Big)\right)\nonumber\\
	\leq& m(k-1)\max_{j\in\{1,...,m\}}\max_{i\neq i^*(\x_j)}\PP\Bigg(\hat{y}_{i^*(\x_j)} (\x_j)\geq\hat{y}_{i} (\x_j)\Bigg).
\end{align*}
Similar to \eqref{eq:a2}, we have
$
\lim\limits_{n\rightarrow\infty}\frac{1}{n}\log \pfs_{\Arm}
= -\min\limits_{j\in\{1,...,m\}}\min\limits_{i\neq i^*(\x_j)}  \Gcal_{i^*(\x_j),i,j}(\alpha_{i^*(\x_j),j},\alpha_{i,j}). \label{eq:a4}
$

\subsection{Proof of Theorem 2}

According to the KKT conditions, there exist constants $\theta$ and $\lambda_{i,j}$ for $j=1,2,...,m$, $i=1,2,...,k$ and $i\neq i^*(\x_j)$ such that
\begin{align}
	& 1-\sum_{j=1}^m \sum_{i=1,i\neq i^*(\x_j)}^k \lambda_{i,j}=0, \label{eq:a6}\\
	& \sum_{i=1,i\neq i^*(\x_j)}^k\lambda_{i,j} \frac{\partial \Gcal_{i^*(\x_j),i,j}(\alpha_{i^*(\x_j),j},\alpha_{i,j})}{\partial \alpha_{i^*(\x_j),j}}=\theta,
	\ \ \ j=1,2,...,m, \label{eq:a7}\\
	& \lambda_{i,j} \frac{\partial \Gcal_{i^*(\x_j),i,j}(\alpha_{i^*(\x_j),j},\alpha_{i,j})}{\partial \alpha_{i,j}}=\theta,
	\ \ \ j=1,2,...,m, i=1,2,...,k \text{ and } i\neq i^*(\x_j), \label{eq:a8}\\
	& \lambda_{i,j}(\Gcal_{i^*(\x_j),i,j}(\alpha_{i^*(\x_j),j},\alpha_{i,j})-z)=0, \ \ \ j=1,2,...,m, i=1,2,...,k \text{ and } i\neq i^*(\x_j). \label{eq:a9}
\end{align}

From (\ref{eq:a7}), all the $\lambda_{i,j}$'s are non-positive or non-negative at the same time, and from (\ref{eq:a6}), $\lambda_{i,j}\geq0$ for $j=1,2,...,m$, $i=1,2,...,k$ and $i\neq i^*(\x_j)$. If we assume that there exist some $j\in\{1,2,...,m\}$ and $i\in\{1,2,...,k\}\setminus \{i^*(\x_j)\}$ such that $\lambda_{i,j}=0$, from (\ref{eq:a8}), $\theta=0$, and then all the $\lambda_{i,j}$'s are equal to 0. This is a contradiction to (\ref{eq:a6}). As a result, $\lambda_{i,j}>0$ for $j=1,2,...,m$, $i=1,2,...,k$ and $i\neq i^*(\x_j)$. From (\ref{eq:a9}), $\Gcal_{i^*(\x_j),i,j}(\alpha_{i^*(\x_j),j},\alpha_{i,j})=z$, and the statement in (4) of the main paper can be concluded.

Next, from (\ref{eq:a8}), $\lambda_{i,j}= \frac{\theta}{\partial \Gcal_{i^*(\x_j),i,j}(\alpha_{i^*(\x_j),j},\alpha_{i,j})/\partial \alpha_{i,j}}$. Substitute it into (\ref{eq:a7}),
\begin{equation}\label{eq:a10}
	\sum_{i=1,i\neq i^*(\x_j)}^k \frac{\partial \Gcal_{i^*(\x_j),i,j}(\alpha_{i^*(\x_j),j},\alpha_{i,j})/\partial \alpha_{i^*(\x_j),j}} {\partial \Gcal_{i^*(\x_j),i,j}(\alpha_{i^*(\x_j),j},\alpha_{i,j})/\partial \alpha_{i,j}}=1, \ \ j=1,2,...,m.
\end{equation}
Result (3) of the main paper follows from (\ref{eq:a10}) because
\begin{align*}
	\frac{\partial \Gcal_{i^*(\x_j),i,j}(\alpha_{i^*(\x_j),j},\alpha_{i,j})} {\partial \alpha_{i^*(\x_j),j}}=
	&\frac{(y_i (\x_j)-y_{i^*(\x_j)} (\x_j))^2\sigma_{i^*(\x_j)}^2(\x_j)}
	{2\alpha_{i^*(\x_j),j}^2 (\sigma_{i^*(\x_j)}^2(\x_j)/\alpha_{i^*(\x_j),j}+\sigma_i^2(\x_j)/\alpha_{i,j})^2} \\
	\frac{\partial \Gcal_{i^*(\x_j),i,j}(\alpha_{i^*(\x_j),j},\alpha_{i,j})} {\partial \alpha_{i,j}}=
	&\frac{(y_i (\x_j)-y_{i^*(\x_j)} (\x_j))^2\sigma_{i}^2(\x_j)}
	{2\alpha_{i,j}^2 (\sigma_{i^*(\x_j)}^2(\x_j)/\alpha_{i^*(\x_j),j}+\sigma_i^2(\x_j)/\alpha_{i,j})^2}.
\end{align*}

\subsection{Proof of Theorem 3}

In this proof, we will append subscript $r$ to notations in CR\&S Algorithm 1 to indicate the iteration number, such as $\hat{n}_{i,j} \leftrightarrow \hat{n}_{(i,j),r}$ which is the total number of simulation replications that treatment $i$ under patient context $\x_j$ receives from iteration 0 to iteration $r$, $\hat{\alpha}_{i,j} \leftrightarrow \hat{\alpha}_{(i,j),r}$ which is $\hat{\alpha}_{(i,j),r} = \hat{n}_{(i,j),r}/n^{(r)}$ where $n^{(r)} = \sum_{j=1}^m \sum_{i=1}^k \hat{n}_{(i,j),r}$, and the sample variance $\hat{\sigma}^2_i (\boldsymbol{\mathrm{x}}_{j}) \leftrightarrow \hat{\sigma}^2_{i,r} (\boldsymbol{\mathrm{x}}_{j})$. The key estimators in CR\&S Algorithm 1 are summarized below:
\begin{align}
	&\hat{i}_r^*(\boldsymbol{\mathrm{x}}_j) \triangleq \min_{i = 1,2,\cdots,k} \bar{Y}_{i,r} (\boldsymbol{\mathrm{x}}_j),\quad \mathcal{\hat{U}}_{(j),r}^b \triangleq \frac{\hat{\alpha}_{(\hat{i}_r^*(\boldsymbol{\mathrm{x}}_j),j),r}^2}{\hat{\sigma}^2_{\hat{i}_r^*(\boldsymbol{\mathrm{x}}_j),r}(\boldsymbol{\mathrm{x}}_j)}, \quad \mathcal{\hat{U}}_{(j),r}^{non} \triangleq \sum_{i \ne \hat{i}_r^*(\boldsymbol{\mathrm{x}}_j)} \frac{\hat{\alpha}_{(i,j),r}^2}{\hat{\sigma}_{i,r}^2(\boldsymbol{\mathrm{x}}_j)} ,\quad \mathcal{S}_{(j),r}^b \triangleq \frac{\hat{\sigma}^2_{\hat{i}_r^*(\boldsymbol{\mathrm{x}}_j),r}(\boldsymbol{\mathrm{x}}_j)}{ \hat{n}_{(\hat{i}_r^*(\boldsymbol{\mathrm{x}}_j),j),r} } , \nonumber \\
	& \mathcal{S}_{(i,j),r} \triangleq \hat{\sigma}_{i,r}^2(\boldsymbol{\mathrm{x}}_j) / \hat{n}_{(i,j),r}, \quad \hat{\delta}_{(i,j),r} \triangleq (\bar{Y}_{i,r}(\boldsymbol{\mathrm{x}}_j) - \bar{Y}_{\hat{i}_r^*(\boldsymbol{\mathrm{x}}_j),r} (\boldsymbol{\mathrm{x}}_j))^2,  \nonumber  \\
	&\mathcal{\hat{\tau}}_{(i,j),r} \triangleq \frac{\hat{\delta}_{(i,j),r}}{\mathcal{S}_{(j),r}^b  + \mathcal{S}_{(i,j),r} }, \quad \mathcal{\hat{V}}_{(i,j),r} \triangleq  \frac{\hat{\tau}_{(i,j),r}}{n^{(r)}}, \quad j=1,2,\cdots,m, \quad i=1,2,\cdots,k \text{ and } i \ne \hat{i}_r^*(\boldsymbol{\mathrm{x}}_j),  \nonumber
\end{align}

The proof of Theorem 3 requires Lemmas \ref{leone}-\ref{lem::bdopttreat} and Propositions \ref{propos1} and \ref{propos2} below. The necessary condition we will frequently use for CR$\&$S Algorithm 1 to sample $\hat{i}^*_r(\boldsymbol{\mathrm{x}}_j)$ under context $\x_j$ has three equivalent forms
\begin{align*}
	&\mathcal{\hat{U}}_{(j),r}^b < \mathcal{\hat{U}}_{(j),r}^{non},
	\text{ or } \frac{\hat{n}_{(\hat{i}_r^*(\boldsymbol{\mathrm{x}}_j),j),r}^2}{\hat{\sigma}^2_{\hat{i}_r^*(\boldsymbol{\mathrm{x}}_j),r}(\boldsymbol{\mathrm{x}}_j)}< \sum_{i \ne \hat{i}_r^*(\boldsymbol{\mathrm{x}}_j)} \frac{\hat{n}_{(i,j),r}^2}{\hat{\sigma}_{i,r}^2(\boldsymbol{\mathrm{x}}_j)},
	\text{ or }
	\sum_{i \ne \hat{i}_r^*(\boldsymbol{\mathrm{x}}_j)} \frac{ \hat{n}_{(i,j),r}^2/\hat{\sigma}_{i,r}^2(\boldsymbol{\mathrm{x}}_j) }{\hat{n}_{(\hat{i}_r^*(\boldsymbol{\mathrm{x}}_j),j),r}^2/\hat{\sigma}^2_{\hat{i}_r^*(\boldsymbol{\mathrm{x}}_j),r}(\boldsymbol{\mathrm{x}}_j)} > 1.
\end{align*}
Moreover, $\hat{\Vcal}_{(i_1,j_1),r} \le  \hat{\Vcal}_{(i_2,j_2),r}$ for some $i_2 \ne \hat{i}_r^*(\boldsymbol{\mathrm{x}}_{j_2})$ under context $\x_{j_2}$ is a necessary condition for CR$\&$S Algorithm 1 to sample treatment $i_1 \ne \hat{i}_r^*(\boldsymbol{\mathrm{x}}_{j_1})$ under context $\x_{j_1}$. Note that $\hat{\Vcal}_{(i_1,j_1),r} \le \hat{\Vcal}_{(i_2,j_2),r}$ is equivalent to
\begin{align*}
	\hat{\tau}_{(i_1,j_1),r} \le \hat{\tau}_{(i_2,j_2),r}
	\text{ or } \left(\mathcal{S}_{(j_2),r}^b  + \mathcal{S}_{(i_2,j_2),r} \right) \hat{\delta}_{(i_1,j_1),r} \le \left(\mathcal{S}_{(j_1),r}^b  + \mathcal{S}_{(i_1,j_1),r} \right) \hat{\delta}_{(i_2,j_2),r}.
\end{align*}
\begin{lemma}\label{leone}
	Let $n^{(r)}_{(j)} = \sum_{i=1}^k \hat{n}_{(i,j),r}$, $j=1,2,\cdots,m$. For context $\boldsymbol{\mathrm{x}}_j$, if $ n^{(r)}_{(j)} \to \infty$, we have $\hat{n}_{(i,j),r} \to \infty $ almost surely for all $i = 1,2,\cdots,k$ as $r \to \infty$.
\end{lemma}
\begin{proof}
	For context $\boldsymbol{\mathrm{x}}_j$ and a fixed sample path $\omega$, define $A = \{ i|\hat{n}_{(i,j),r} \to \infty \}$. Since $ n^{(r)}_{(j)} \to \infty$, it is obvious that $A$ is non-empty. Suppose $A^c$ is also non-empty and $i_1 \in A^c$. Denote $K_1$ as the last time such that $(i_1,j)$ is sampled. It means
	$
	K_1 = \sup \{ r | (i^r,j^r) = (i_1,j) \}
	$ where $(i^r,j^r)$ is the treatment-context pair simulated at iteration $r$.
	
	Since $\hat{\sigma}^2_{i,r}(\boldsymbol{\mathrm{x}}_j)$ converges almost surely, we could find an upper bound $b_{vU}$ such that $\hat{\sigma}_{i,r}^2(\boldsymbol{\mathrm{x}}_j) < b_{vU}$ for all $i = 1,2,\cdots, k$. Then, as $ n^{(r)}_{(j)}$ increases, there must exist a finite time $K_2 > K_1$ such that
	\begin{equation}\label{in:contra1}
		\frac{\hat{n}_{(i_1,j),r}^2}{\hat{\sigma}_{i_1,r}^2(\boldsymbol{\mathrm{x}}_j)} < \frac{1}{b_{vU}} \sum_{i \ne i_1(\boldsymbol{\mathrm{x}}_j)}\hat{n}_{(i,j),r}^2 < \sum_{i \ne i_1(\boldsymbol{\mathrm{x}}_j)} \frac{\hat{n}_{(i,j),r}^2}{\hat{\sigma}_{i,r}^2(\boldsymbol{\mathrm{x}}_j)}
	\end{equation}
	holds when $r > K_2$. For all $r > K_2$, we claim $\hat{i}^*_r(\boldsymbol{\mathrm{x}}_j) \ne i_1$. Otherwise, since $n^{(r)}_{(j)} \to \infty$, we would be able to find some iteration $r > K_2$ where some treatment under context $\boldsymbol{\mathrm{x}}_j$ is sampled. Then we will have $(i^r,j^r) = (i_1,j)$ at this iteration $r$ in CR\&S Algorithm 1 because $\hat{\Ucal}_{(j),r}^b<\hat{\Ucal}_{(j),r}^{non}$ by \eqref{in:contra1}, contradicting the definitions of $K_1$ and $K_2$.
	
	The analysis above applies to any $i_1 \in A^c$. Since $A^c$ contains a finite number of treatments, there exists $K_3$ such that $i^r \in A$ and $\hat{i}^*_r(\boldsymbol{\mathrm{x}}_j) \in A$ for all $r > K_3$, where context $\boldsymbol{\mathrm{x}}_j$ is sampled.
	
	Meanwhile, $A$ should contain at least two treatments. Otherwise, if $A$ has only one element $i_0$, then $i^r = \hat{i}^*_r(\boldsymbol{\mathrm{x}}_j) = i_0$ for all $r > K_3$. It means $r^2\mathcal{\hat{U}}_{(j),r}^b =  \hat{n}_{(\hat{i}_r^*(\boldsymbol{\mathrm{x}}_j),j),r}^2/\hat{\sigma}^2_{\hat{i}_r^*(\boldsymbol{\mathrm{x}}_j),r}(\boldsymbol{\mathrm{x}}_j) > \hat{n}_{(\hat{i}_r^*(\boldsymbol{\mathrm{x}}_j),j),r}^2/ b_{vU}$  will go to infinity and $r^2\mathcal{\hat{U}}_{(j),r}^{non} = \sum_{i \ne \hat{i}_r^*(\boldsymbol{\mathrm{x}}_j)} \hat{n}_{(i,j),r}^2/\hat{\sigma}_{i,r}^2(\boldsymbol{\mathrm{x}}_j)$ will remain fixed when $r >K_3$ and $n^{(r)}_{(j)} \to \infty$. This leads to contradiction because we have to sample $(i^r,j^r) = (i_1,j)$ due to $\hat{\Ucal}_{(j),r}^b>\hat{\Ucal}_{(j),r}^{non}$, where $i_1 \in A^c$, at some iteration $r > K_3$ by CR$\&$S Algorithm 1.
	
	Then there exists $K_4 > K_3$ such that $\hat{i}^*_r(\boldsymbol{\mathrm{x}}_j)$ remains the same for all $r >K_4$ because $\bar{Y}_{i_0,r}(\boldsymbol{\mathrm{x}}_j)$ converges to $y_{i_0}(\x_j)$ for $i_0 \in A$ and $\bar{Y}_{i_1,r}(\boldsymbol{\mathrm{x}}_j)$ remains fixed for $i_1 \in A^c$. Then, when $r > K_4$, we have for $i_0 \in A$, $i_0 \ne \hat{i}^*_r(\boldsymbol{\mathrm{x}}_j)$ and $i_1 \in A^c$
	\begin{align*}
		&\mathcal{\hat{\tau}}_{(i_0,j),r} - \mathcal{\hat{\tau}}_{(i_1,j),r}
		=\frac{\hat{\delta}_{(i_0,j),r}}{\mathcal{S}_{(j),r}^b + \mathcal{S}_{(i_0,j),r}} - \frac{\hat{\delta}_{(i_1,j),r}}{\mathcal{S}_{(j),r}^b + \mathcal{S}_{(i_1,j),r}}    \\
		=&\frac{\big( \hat{\delta}_{(i_0,j),r} - \hat{\delta}_{(i_1,j),r} \big) \hat{\sigma}^2_{\hat{i}_r^*(\boldsymbol{\mathrm{x}}_j),r}(\boldsymbol{\mathrm{x}}_j) / \hat{n}_{(\hat{i}_r^*(\boldsymbol{\mathrm{x}}_j),j),r} + \hat{\delta}_{(i_0,j),r} \hat{\sigma}_{i_1,r}^2(\boldsymbol{\mathrm{x}}_j) / \hat{n}_{(i_1,j),r} - \hat{\delta}_{(i_1,j),r}\hat{\sigma}_{i_0,r}^2(\boldsymbol{\mathrm{x}}_j) / \hat{n}_{(i_0,j),r}     }{\big( \mathcal{S}_{(j),r}^b + \mathcal{S}_{(i_0,j),r} \big)\big( \mathcal{S}_{(j),r}^b + \mathcal{S}_{(i_1,j),r} \big)}  \\
		>& \frac{b_{v_1} / \hat{n}_{(\hat{i}_r^*(\boldsymbol{\mathrm{x}}_j),j),r} + b_{v_2} / \hat{n}_{(i_1,j),r} - b_{v_3} / \hat{n}_{(i_0,j),r}}{\big( \mathcal{S}_{(j),r}^b + \mathcal{S}_{(i_0,j),r} \big)\big( \mathcal{S}_{(j),r}^b + \mathcal{S}_{(i_1,j),r} \big)},
	\end{align*}
	where $b_{v_1}$ and $b_{v_2}$ are lower bounds of $\big( \hat{\delta}_{(i_0,j),r} - \hat{\delta}_{(i_1,j),r} \big) \hat{\sigma}^2_{\hat{i}_r^*(\boldsymbol{\mathrm{x}}_j),r}(\boldsymbol{\mathrm{x}}_j)$ and $\hat{\delta}_{(i_0,j),r} \hat{\sigma}_{i_1,r}^2(\boldsymbol{\mathrm{x}}_j)$ and $b_{v_3}$ is an upper bound of $\hat{\delta}_{(i_1,j),r} \hat{\sigma}_{i_0,r}^2(\boldsymbol{\mathrm{x}}_j)$. Moreover, $b_{v_2} > 0$ and $b_{v_3}>0$.
	
	Since $i_0$ and $\hat{i}_r^*(\boldsymbol{\mathrm{x}}_j)$ are in $A$, we know that both $\hat{n}_{(\hat{i}_r^*(\boldsymbol{\mathrm{x}}_j),j),r}$ and $\hat{n}_{(i_0,j),r}$ go to $\infty$ as $r$ increases. Then, there exists a $K_5 > K_4$ such that, when $r > K_5$,
	\begin{displaymath}
		\begin{aligned}
			\frac{|b_{v_1}|}{\hat{n}_{(\hat{i}_r^*(\boldsymbol{\mathrm{x}}_j),j),r}} < \frac{b_{v_2}}{2\hat{n}_{(i_1,j),K_4}} = \frac{b_{v_2}}{2\hat{n}_{(i_1,j),r}}, \quad \frac{b_{v_3}}{\hat{n}_{(i_0,j),r}} < \frac{b_{v_2}}{2\hat{n}_{(i_1,j),K_4}} = \frac{b_{v_2}}{2\hat{n}_{(i_1,j),r}}.
		\end{aligned}
	\end{displaymath}
	
	It means $\mathcal{\hat{\tau}}_{(i_0,j),r} - \mathcal{\hat{\tau}}_{(i_1,j),r} > 0$ when $r > K_5$. Since the number of treatments is finite, we could find $K_6 > K_5$ such that when $r > K_6$,
	$
	\max_{i_1 \in A^c} \mathcal{\hat{\tau}}_{(i_1,j),r} < \min_{i_0 \in A} \mathcal{\hat{\tau}}_{(i_0,j),r}.
	$
	
	Since $\hat{n}_{(j)}^{(r)} \to \infty$, some treatments under context $\boldsymbol{\mathrm{x}}_j$ will be sampled infinitely. However, based on the criteria of CR$\&$S algorithm 1, when $r > K_6$, context $\boldsymbol{\mathrm{x}}_j$ will be sampled only through $\mathcal{\hat{\tau}}_{(i_1,j),r}$, $i_1 \in A^c$. It means $i_0 \ne \hat{i}_r^*(\boldsymbol{\mathrm{x}}_j)$ where $i_0 \in A$ will not be sampled when $r > K_6$. This leads to contradiction because $\hat{n}_{(i_0,j),r} \to \infty$ by the definition of $A$. So $A^c = \emptyset$ and we have proved the lemma.
\end{proof}

\begin{lemma}\label{letwo}
	For all $\boldsymbol{\mathrm{x}}_j, j=1,\cdots,m$, we have $\lim_{r \to \infty}  n^{(r)}_{(j)} = \infty$ almost surely.
\end{lemma}
\begin{proof}
	Fix a sample path $\omega$ and define $B = \{ j | \lim_{r \to \infty} n^{(r)}_{(j)} = \infty \}$. Obviously, $B$ is non-empty. If $B^c$ is also non-empty, for every $j_1 \in B^c$, let
	$
	L_1 = \sup \{ r | j^r = j_1 \}.
	$
	Then context $\boldsymbol{\mathrm{x}}_{j_1}$ will not be sampled when $r > L_1$.
	Thus, $\mathcal{\hat{\tau}}_{(i,j_1),r}$ stays the same for $r > L_1, i=1,\cdots,k$.
	
	Since $B^c$ has a finite number of contexts, we could find $L_2 \ge L_1$ such that for any $j \in B^c$, context $\boldsymbol{\mathrm{x}}_j$ will not be sampled when $r > L_2$. Then, there exists a positive constant $b^u$ such that
	$
	\max_{j \in B^c,i=1,\cdots,k}\mathcal{\hat{\tau}}_{(i,j),r} < b^u
	$ because $\bar{Y}_{i,r} (\boldsymbol{\mathrm{x}}_j)$, $\hat{\sigma}_{i,r}^2(\boldsymbol{\mathrm{x}}_j)$, and $\hat{n}_{(i,j),r}$ remain fixed, $j \in B^c,i=1,\cdots,k$.
	
	However, by Lemma \ref{leone}, $\mathcal{\hat{\tau}}_{(i,j),r} \to \infty $ because $\hat{n}_{(i,j),r} \to \infty $ as $r \to \infty$ if $j \in B$ and $i = 1, \cdots,k$. From here, it is straightforward to see that we could find $L_3 > L_2$ such that
	$
	\min_{j \in A,i=1,\cdots,k}\mathcal{\hat{\tau}}_{(i,j),r} > \max_{j \in A^c,i=1,\cdots,k}\mathcal{\hat{\tau}}_{(i,j),r},
	$
	for $r > L_3$. That is, we have to simulate treatments under contexts $\boldsymbol{\mathrm{x}}_j$ where $j \in B^c$ in CR\&S Algorithm 1. This contradicts the definition of $L_2$. Therefore, $B^c=\emptyset$.
\end{proof}

\begin{lemma}\label{lethree}
	For all $\boldsymbol{\mathrm{x}}_j, \ j=1,\cdots,m$ and $i=1,\cdots,k$, we have $ \lim_{r \to \infty} \hat{n}_{(i,j),r} = \infty$ almost surely.
\end{lemma}
\begin{proof}
	This lemma is obvious by Lemmas \ref{leone} and \ref{letwo}.
\end{proof}

\begin{remark}
	Since $ \hat{n}_{(i,j),r} \to \infty$ as $r \to \infty$ for all contexts $\boldsymbol{\mathrm{x}}_j, j=1,\cdots,m$ and treatments $i=1,\cdots,k$, we will always have $\hat{i}^*_r(\boldsymbol{\mathrm{x}}_j) = i^*(\boldsymbol{\mathrm{x}}_j)$, i.e., the estimated best treatment is the true best under context $\boldsymbol{\mathrm{x}}_j$ on almost every sample path when $r$ is large enough. So, without loss of generality, when we say iteration $r$ is large enough in the rest of the proof, we mean that the estimate $\bar{Y}_i(\boldsymbol{\mathrm{x}}_j)$ is close enough to its real value $y_i(\boldsymbol{\mathrm{x}}_j)$ such that $\hat{i}^*_r(\boldsymbol{\mathrm{x}}_j) = i^*(\boldsymbol{\mathrm{x}}_j)$ for all $j=1,\cdots,m$ and $i=1,\cdots,k$.
	
	Moreover, since $ \hat{n}_{(i,j),r} \to \infty$ as $r \to \infty$, we can find constants $b_{L},b_{U},b_{vL}$, and $b_{vU}$ such that $0<b_{L} < \hat{\delta}_{(i,j),r} < b_{U}$,  $0 < b_{vL}<\hat{\sigma}^2_{i,r} (\boldsymbol{\mathrm{x}}_j) < b_{vU}$, $\forall j=1,\dots,m$, $i=1,\dots,k$. In the subsequent proof, we will use the notations $b_{L},b_{U},b_{vL}$, and $b_{vU}$ without repeated explanation.
\end{remark}

\begin{lemma}\label{lefi}
	For any context $\boldsymbol{\mathrm{x}}_j$ and any two treatments $i_1, i_2 \ne i^*(\boldsymbol{\mathrm{x}}_j)$, $\mathop{ \lim \inf }_{r \to \infty} \frac{\hat{\alpha}_{(i_1,j),r}}{\hat{\alpha}_{(i_2,j),r}} > 0$ almost surely.
\end{lemma}
\begin{proof}	
	We prove this lemma by contradiction when $r$ is large enough. Suppose that there exist $i_1, i_2 > 1$ and
	$
	\mathop{ \lim \inf }_{r \to \infty} \hat{\alpha}_{(i_1,j),r}/\hat{\alpha}_{(i_2,j),r} = 0.
	$
	Let $c = \frac{ b_{L} b_{vL} } {(b_{U} - b_{L}) b_{vU} + b_{U} b_{vU} +1 }$. Then, we could find a large enough iteration $r_1$ such that
	\begin{align}\label{defc}
		&\frac{\hat{\alpha}_{(i_1,j),r_1}}{\hat{\alpha}_{(i_2,j),r_1}} <  c < \frac{ b_{L} \hat{\sigma}^2_{i_1,r_1} (\boldsymbol{\mathrm{x}}_j) }{(b_{U} - b_{L}) \hat{\sigma}_{i^{*} (\x_{j}),r_1} (\boldsymbol{\mathrm{x}}_j) \hat{\sigma}_{i_2,r_1} (\boldsymbol{\mathrm{x}}_j) + b_{U} \hat{\sigma}^2_{i_2,r_1} (\boldsymbol{\mathrm{x}}_j) +1 },
	\end{align}
	and we will sample treatment $i_2$ under $\boldsymbol{\mathrm{x}}_j$ to make $\frac{\hat{\alpha}_{(i_1,j),r_1+1}}{\hat{\alpha}_{(i_2,j),r_1+1}} < \frac{\hat{\alpha}_{(i_1,j),r_1}}{\hat{\alpha}_{(i_2,j),r_1}}$. However, at iteration $r_1$,
	\begin{align*}
		\hat{\mathcal{V}}_{(i_1,j),r_1} - \hat{\mathcal{V}}_{(i_2,j),r_1}  &= \frac{\hat{\delta}_{(i_1,j),r_1}}{\big( \mathcal{S}_{(i^{*} (\x_{j}),j),r_1} + \mathcal{S}_{(i_1,j),r_1} \big) n^{(r_1)} } - \frac{\hat{\delta}_{(i_2,j),r_1}}{\big( \mathcal{S}_{(i^{*} (\x_{j}),j),r_1} + \mathcal{S}_{(i_2,j),r_1} \big) n^{(r_1)} }  \\
		&<  \frac{b_{U}}{\big( \mathcal{S}_{(i^{*} (\x_{j}),j),r_1} + \mathcal{S}_{(i_1,j),r_1} \big) n^{(r_1)} } - \frac{b_{L}}{\big( \mathcal{S}_{(i^{*} (\x_{j}),j),r_1} + \mathcal{S}_{(i_2,j),r_1} \big) n^{(r_1)} }  \\
		&= \frac{ (b_{U}-b_{L}) \frac{\hat{\sigma}^2_{i^{*} (\x_{j}),r_1} (\boldsymbol{\mathrm{x}}_j)}{\hat{\alpha}_{(i^{*} (\x_{j}),j),r_1}} + b_{U} \frac{\hat{\sigma}^2_{i_2,r_1} (\boldsymbol{\mathrm{x}}_j)}{\hat{\alpha}_{(i_2,j),r_1}} - b_{L} \frac{\hat{\sigma}^2_{i_1,r_1} (\boldsymbol{\mathrm{x}}_j)}{\hat{\alpha}_{(i_1,j),r_1} } } {\big( \mathcal{S}_{(i^{*} (\x_{j}),j),r_1} + \mathcal{S}_{(i_1,j),r_1} \big) \big( \mathcal{S}_{(i^{*} (\x_{j}),j),r_1} + \mathcal{S}_{(i_2,j),r_1} \big)(n^{(r_1)})^2}.
	\end{align*}
	For the numerator, we have
	\begin{eqnarray}
		&&(b_{U}-b_{L}) \frac{\hat{\sigma}^2_{i^{*} (\x_{j}),r_1} (\boldsymbol{\mathrm{x}}_j)}{\hat{\alpha}_{(i^{*} (\x_{j}),j),r_1}} + b_{U} \frac{\hat{\sigma}^2_{i_2,r_1} (\boldsymbol{\mathrm{x}}_j)}{\hat{\alpha}_{(i_2,j),r_1}} - b_{L}\frac{\hat{\sigma}^2_{i_1,r_1} (\boldsymbol{\mathrm{x}}_j)}{\hat{\alpha}_{(i_1,j),r_1} }  \nonumber \\
		&\le& (b_{U}-b_{L}) \frac{\hat{\sigma}_{i^{*} (\x_{j}),r_1} (\boldsymbol{\mathrm{x}}_j)\hat{\sigma}_{i_2,r_1} (\boldsymbol{\mathrm{x}}_j) }{ \hat{\alpha}_{(i_2,j),r_1} } + b_{U} \frac{\hat{\sigma}^2_{i_2,r_1} (\boldsymbol{\mathrm{x}}_j)}{\hat{\alpha}_{(i_2,j),r_1}} - b_{L} \frac{\hat{\sigma}^2_{i_1,r_1} (\boldsymbol{\mathrm{x}}_j)}{\hat{\alpha}_{(i_1,j),r_1} }   \label{eqb41} \\
		&<& (b_{U}-b_{L}) \frac{\hat{\sigma}_{i^{*} (\x_{j}),r_1} (\boldsymbol{\mathrm{x}}_j)\hat{\sigma}_{i_2,r_1} (\boldsymbol{\mathrm{x}}_j) }{ \hat{\alpha}_{(i_2,j),r_1} } + b_{U} \frac{\hat{\sigma}^2_{i_2,r_1} (\boldsymbol{\mathrm{x}}_j)}{\hat{\alpha}_{(i_2,j),r_1}} - b_{L} \frac{\hat{\sigma}^2_{i_1,r_1} (\boldsymbol{\mathrm{x}}_j)}{c \hat{\alpha}_{(i_2,j),r_1} }  < 0. \label{eqb43}
	\end{eqnarray}
	(\ref{eqb41}) holds because a non-best treatment is sampled by the definition of $r_1$ which suggests
	$$ \hat{\alpha}_{(i^{*} (\x_{j}),j),r_1}^2/\hat{\sigma}^2_{i^{*} (\x_{j}),r_1}(\boldsymbol{\mathrm{x}}_j) \ge \sum_{i \ne i^{*} (\x_{j})} \hat{\alpha}_{(i,j),r_1}^2/\hat{\sigma}_{i,r_1}^2(\boldsymbol{\mathrm{x}}_j) \ge \hat{\alpha}_{(i_2,j),r_1}^2/\hat{\sigma}_{i_2,r_1}^2(\boldsymbol{\mathrm{x}}_j).$$
	(\ref{eqb43}) holds because of (\ref{defc}) and the definition of $c$. Thus, $\hat{\mathcal{V}}_{(i_1,j),r_1} - \hat{\mathcal{V}}_{(i_2,j),r_1} < 0$ and treatment $i_2$ under $\x_{j}$ will not be sampled by our algorithm, contradicting $r_1$'s definition.
\end{proof}

\begin{remark}
	By symmetry, we have $\mathop{ \lim \sup }_{r \to \infty} \frac{\hat{\alpha}_{(i_1,j),r}}{\hat{\alpha}_{(i_2,j),r}} < \infty$ almost surely, for $i_1,i_2 \ne i^*(\x_j)$.
\end{remark}

\begin{lemma}\label{lefi2}
	For any context $\boldsymbol{\mathrm{x}}_j$, $0 < \mathop{ \lim \inf }_{r \to \infty} \hat{\alpha}_{(i^*(\boldsymbol{\mathrm{x}}_j),j),r}/\hat{\alpha}_{(i,j),r} < \mathop{ \lim \sup }_{r \to \infty} \hat{\alpha}_{(i^*(\boldsymbol{\mathrm{x}}_j),j),r}/\hat{\alpha}_{(i,j),r}$ $< \infty$ almost surely, for $i \ne i^*(\boldsymbol{\mathrm{x}}_j)$.
\end{lemma}
\begin{proof}
	For a fixed sample path, we consider a large enough $r$. It is sufficient to prove
	$\mathop{ \lim \sup }_{r \to \infty}$ $\hat{\alpha}_{(i^{*} (\x_{j}),j),r}/\hat{\alpha}_{(i,j),r} < \infty$, $i \ne i^{*} (\x_{j})$.
	
	Suppose there exists a non-best treatment $i_0 \ne i^{*} (\x_{j})$ under context $\boldsymbol{\mathrm{x}}_j$ such that
	$\mathop{ \lim \sup }_{r \to \infty}$ $\hat{\alpha}_{(i^{*} (\x_{j}),j),r}/\hat{\alpha}_{(i_0,j),r} = \infty.$ By Lemma \ref{lefi}, we have
	$ \hat{\alpha}_{(i,j),r} < c_{i} \hat{\alpha}_{(i_0,j),r}$, $i \ne i^{*} (\x_{j})$, where $c_i$'s are positive constants and $r=1,2,\dots$.
	Then, we could find an iteration $r_1$ where $(i^{r_1},j^{r_1}) = (i^{*} (\x_{j}),j)$ and
	\begin{equation}\label{eqalp}
		\left( \hat{\alpha}_{(i^{*} (\x_{j}),j),r_1}/\hat{\alpha}_{(i_0,j),r_1} \right)^2 > \frac{ b_{vU} }{ b_{vL} } \sum\nolimits_{i\ne i^{*} (\x_{j})} c_{i}^2 + 1.
	\end{equation}
	However, at iteration $r_1$,
	\begin{align*}
		\frac{ \big( \hat{\alpha}_{(i^{*} (\x_{j}),j),r_1} / \hat{\sigma}_{i^{*} (\x_{j}),r_1}(\boldsymbol{\mathrm{x}}_j)  \big)^2 }{ \sum_{i \ne i^{*} (\x_{j})} \big( \hat{\alpha}_{(i,j),r_1} / \hat{\sigma}_{i,r_1}(\boldsymbol{\mathrm{x}}_j) \big)^2 }
		&> \frac{  \hat{\alpha}_{(i^{*} (\x_{j}),j),r_1}^2 / b_{vU} }{ \sum_{i \ne i^{*} (\x_{j})} \hat{\alpha}_{(i,j),r_1}^2 / b_{vL} } > \frac{  \hat{\alpha}_{(i^{*} (\x_{j}),j),r_1}^2 / b_{vU} }{ \hat{\alpha}^2_{(i_0,j),r_1} \sum_{i \ne i^{*} (\x_{j})} c_{i}^2 / b_{vL} }  > 1.
	\end{align*}
	The last inequality holds because of (\ref{eqalp}). So, based on CR\&S Algorithm 1, $(i^{r_1},j^{r_1}) \ne (i^{*} (\x_{j}),j)$ because $\hat{\Ucal}_{(j),r_1}^b/\hat{\Ucal}_{(j),r_1}^{non}>1$, which leads to contradiction. Similarly, it is easy to show that $\mathop{ \lim \inf }_{r \to \infty} $ $ \hat{\alpha}_{(i^*(\boldsymbol{\mathrm{x}}_j),j),r}/\hat{\alpha}_{(i,j),r} > 0.$
\end{proof}

\begin{lemma}\label{lesix}
	For $j_1,j_2=1,\cdots,m$, $i_1,i_2=1,\cdots,k$, $i_1 \ne i^*(\boldsymbol{\mathrm{x}}_{j_1})$ and $i_2 \ne i^*(\boldsymbol{\mathrm{x}}_{j_2})$, $\mathop{\lim \sup}_{r \to \infty}$ $ \hat{\alpha}_{(i_1,j_1),r}/\hat{\alpha}_{(i_2,j_2),r} < \infty $ almost surely.
\end{lemma}
\begin{proof}
	Suppose $i_1,i_2,j_1$ and $j_2$ satisfy
	$
	\mathop{\lim \sup}_{r \to \infty} \hat{\alpha}_{(i_2,j_2),r}/\hat{\alpha}_{(i_1,j_1),r} = \infty.
	$
	By Lemma \ref{lefi2}, there exist $H_1$ and $H_2$ such that $\hat{\alpha}_{(i^{*} (\x_{j_1}),j_1),r} < H_1 \hat{\alpha}_{(i_1,j_1),r}$ and $\hat{\alpha}_{(i^{*} (\x_{j_2}),j_2),r} > H_2 \hat{\alpha}_{(i_2,j_2),r}$ for all iterations $r$.
	
	Since $
	\mathop{\lim \sup}_{r \to \infty} \hat{\alpha}_{(i_2,j_2),r}/\hat{\alpha}_{(i_1,j_1),r} = \infty
	$, we could find an iteration $r_1$ such that
	\begin{equation}\label{deqcond1}
		\frac{\hat{\alpha}_{(i_2,j_2),r_1}}{\hat{\alpha}_{(i_1,j_1),r_1}} > \frac{b_{U} \big( b_{vU}/H_2 + b_{vU} \big) +1 }{b_{L} \big( b_{vL}/H_1 + b_{vL} \big)} \triangleq c > \frac{b_{U} \big( \hat{\sigma}^2_{i^{*} (\x_{j_2}),r_1}(\boldsymbol{\mathrm{x}}_{j_2})/H_2 + \hat{\sigma}^2_{i_2,r_1} (\boldsymbol{\mathrm{x}}_{j_2}) \big) +1 }{b_{L} \big( \hat{\sigma}^2_{i^{*} (\x_{j_1}),r_1}(\boldsymbol{\mathrm{x}}_{j_1})/H_1 + \hat{\sigma}^2_{i_1,r_1} (\boldsymbol{\mathrm{x}}_{j_1}) \big)} ,
	\end{equation}
	and $(i^{r_1},j^{r_1}) = (i_2,j_2)$.
	Then, we have
	\begin{align*}
		&\hat{\mathcal{V}}_{(i_2,j_2),r_1} - \hat{\mathcal{V}}_{(i_1,j_1),r_1} = \frac{\hat{\delta}_{(i_2,j_2),r_1}}{\big( \mathcal{S}_{(i^{*} (\x_{j_2}),j_2),r_1} + \mathcal{S}_{(i_2,j_2),r_1} \big)n^{(r_1)} }  -  \frac{\hat{\delta}_{(i_1,j_1),r_1}}{\big( \mathcal{S}_{(i^{*} (\x_{j_1}),j_1),r_1} + \mathcal{S}_{(i_1,j_1),r_1} \big) n^{(r_1)} } \\
		&>  \frac{b_{L}}{\big( \mathcal{S}_{(i^{*} (\x_{j_2}),j_2),r_1} + \mathcal{S}_{(i_2,j_2),r_1} \big) n^{(r_1)} } - \frac{b_{U}}{\big( \mathcal{S}_{(i^{*} (\x_{j_1}),j_1),r_1} + \mathcal{S}_{(i_1,j_1),r_1} \big)n^{(r_1)} }  \\
		&> \frac{b_{L}}{\hat{\sigma}^2_{i^{*} (\x_{j_2}),r_1}(\boldsymbol{\mathrm{x}}_{j_2})/H_2 + \hat{\sigma}^2_{i_2,r_1} (\boldsymbol{\mathrm{x}}_{j_2})} \hat{\alpha}_{(i_2,j_2),r_1} - \frac{b_{U}}{\hat{\sigma}^2_{i^{*} (\x_{j_1}),r_1}(\boldsymbol{\mathrm{x}}_{j_1})/H_1 + \hat{\sigma}^2_{i_1,r_1} (\boldsymbol{\mathrm{x}}_{j_1})} \hat{\alpha}_{(i_1,j_1),r_1}  \\
		&= \frac{b_{L} \big( \hat{\sigma}^2_{i^{*} (\x_{j_1}),r_1}(\boldsymbol{\mathrm{x}}_{j_1})/H_1 + \hat{\sigma}^2_{i_1,r_1} (\boldsymbol{\mathrm{x}}_{j_1}) \big) \hat{\alpha}_{(i_2,j_2),r_1} - b_{U} \big( \hat{\sigma}^2_{i^{*} (\x_{j_2}),r_1}(\boldsymbol{\mathrm{x}}_{j_2})/H_2 + \hat{\sigma}^2_{i_2,r_1} (\boldsymbol{\mathrm{x}}_{j_2}) \big) \hat{\alpha}_{(i_1,j_1),r_1}  }{\big( \hat{\sigma}^2_{i^{*} (\x_{j_1}),r_1}(\boldsymbol{\mathrm{x}}_{j_1})/H_1 + \hat{\sigma}^2_{i_1,r_1} (\boldsymbol{\mathrm{x}}_{j_1}) \big) \big( \hat{\sigma}^2_{i^{*} (\x_{j_2}),r_1}(\boldsymbol{\mathrm{x}}_{j_2})/H_2 + \hat{\sigma}^2_{i_2,r_1} (\boldsymbol{\mathrm{x}}_{j_2}) \big)} > 0.
	\end{align*}
	The last inequality holds because of (\ref{deqcond1}). So, $\hat{\mathcal{V}}_{(i_2,j_2),r_1} - \hat{\mathcal{V}}_{(i_1,j_1),r_1} > 0$. We cannot sample treatment $i_2$ of context $\boldsymbol{\mathrm{x}}_{j_2}$ at iteration $r_1$ in CR\&S algorithm 1, contradicting $r_1$'s definition.
\end{proof}

\begin{lemma}\label{lem::order1}
	For any two contexts $\boldsymbol{\mathrm{x}}_{j_1}$ and $\boldsymbol{\mathrm{x}}_{j_2}$ and any two treatments $i_1$ and $i_2$, $\mathop{ \lim \inf }_{r \to \infty} \frac{\hat{\alpha}_{(i_1,j_1),r}}{\hat{\alpha}_{(i_2,j_2),r}}$ $> 0$ almost surely.
\end{lemma}

This lemma is straightforward from Lemmas \ref{lefi}, \ref{lefi2} and \ref{lesix}. Lemma \ref{lem::order1} implies that $ b_{1}\hat{\alpha}_{(i_2,j_2),r} \le \hat{\alpha}_{(i_1,j_1),r} \le b_{2}\hat{\alpha}_{(i_2,j_2),r}$ almost surely, where $b_{1}$ and $b_{2}$ are positive constants, $i_1,i_2 = 1,\cdots,k$ and $j_1,j_2 = 1, \cdots,m$. Then Theorem 4 is obvious.

\begin{remark}
	Starting from Proposition \ref{propos1} below, we will make frequent use of the Landau notation. Suppose $a,b,c:\mathbb{N} \to [0,\infty)$ where $\mathbb{N}$ is the set of natural numbers. Then
	\begin{align*}
		&a(r) = O(b(r)) \iff  \limsup_{r\to \infty} a(r)/b(r) < \infty,  \\
		&a(r) = \Theta(b(r))  \iff \limsup_{r\to \infty} a(r)/b(r) < \infty \text{ and } \liminf_{r \to \infty} a(r)/b(r) > 0.
	\end{align*}
	Note that $
	a(r) = \Theta(b(r)) \Longrightarrow a(r) = O(b(r))
	$,
	and that $
	a(r) = \Theta(b(r)), \ b(r) = \Theta(c(r)) \Longrightarrow a(r) = \Theta(c(r)).
	$ Then, by Lemma \ref{lem::order1}, we have $\hat{n}_{(i,j),r} = \Theta(n^{(r)})$, $i=1,\dots,k$, $j=1,\dots,m$.
\end{remark}

\begin{proposition}\label{propos1}
	For $\hat{\alpha}_{i,j}$ generated by the CR$\&$S Algorithm 1, $i=1,2,...,k$ and $j=1,2,...,m$, we have
	$\Bigg| \frac{\hat{\alpha}_{i^*(\x_j),j}^2}{\sigma_{i^*(\x_j)}^2(\x_j)}-
	\sum_{i=1,i\neq i^*(\x_j)}^k \frac{\hat{\alpha}_{i,j}^2} {\sigma_{i}^2(\x_j)}\Bigg| \leq O\left( \sqrt{ \frac{  \log \log n^{(r)} }{n^{(r)}}}\right),
	$
	almost surely, for any $j=1,2,\cdots,m$.
\end{proposition}
\begin{proof}
	
	For a fixed sample path and a given context $\x_j$, the iterations where any treatment of $\x_j$ is sampled form into a subsequence of $\{1,2,\dots\}$. We denote this subsequence by $\{r^{(j)}_l,l=1,2,\dots\}$.
	
	Suppose $(i^{r^{(j)}_s},j^{r^{(j)}_s}) = (i_1,j)$ at iteration $r^{(j)}_s$ where $s \ge 1$ is large enough and $i_1 \ne i^*(\x_j)$, $(i^{r^{(j)}_{s+h}},j^{r^{(j)}_{s+h}}) = (i^*(\x_j),j)$ at the successive iterations $r^{(j)}_{s+h}$ where $1 \le h \le t-1$, and $(i^{r^{(j)}_{s+t}},j^{r^{(j)}_{s+t}}) = (i_2,j)$ and $i_2 \ne i^*(\x_j)$ at iteration $r^{(j)}_{s+t}$.
	Since any non-best treatment of context $\x_j$ would not be sampled from iteration $r^{(j)}_{s}+1$ to $r^{(j)}_{s+t}-1$, we have $\hat{n}_{ (i,j), r} \le  \hat{n}_{ (i,j), r^{(j)}_{s}} +1$ for $i \ne i^*(\x_j)$ and $r^{(j)}_s+1 \le r \le r^{(j)}_{s+t}-1$. Since $\hat{n}_{ (i,j), r} = \Theta(\hat{n}_{ (i,j), r^{(j)}_{s}})$ for $i \ne i^*(\x_j)$, $\hat{n}_{ (i,j), r} = \Theta(n^{(r)})$ and $\hat{n}_{ (i,j), r^{(j)}_{s}} = \Theta(n^{(r^{(j)}_{s})})$ by Lemma \ref{lem::order1},  we have $\hat{n}_{ (i,j), r} = \Theta( n^{(r^{(j)}_{s})} )$ for $i=1,\dots,k$ and $r^{(j)}_s +1 \le r \le r^{(j)}_{s+t}-1$. By the law of the iterated logarithm,
	$
	\left| \frac{\hat{\sigma}^2_{i,r}(\boldsymbol{\mathrm{x}}_{j})}{\sigma^2_{i}(\boldsymbol{\mathrm{x}}_{j})} - 1 \right| = O\left(\sqrt{\frac{\log \log \hat{n}_{ (i,j), r}}{\hat{n}_{ (i,j), r}}}\right) = O\left(\sqrt{\frac{\log \log n^{(r^{(j)}_{s})} }{ n^{(r^{(j)}_{s})} }}\right), \ r^{(j)}_s+1 \le r \le r^{(j)}_{s+t}-1,
	$
	which indicates $\left| \hat{\sigma}^2_{i,r}(\boldsymbol{\mathrm{x}}_{j}) / \hat{\sigma}^2_{i,r^{(j)}_{s}}(\boldsymbol{\mathrm{x}}_{j}) - 1 \right| = O\left(\sqrt{\frac{\log \log n^{(r^{(j)}_{s})} }{ n^{(r^{(j)}_{s})} }}\right)$, $r^{(j)}_s+1 \le r \le r^{(j)}_{s+t}-1$.
	
	Then for $r^{(j)}_s+1 \le r \le r^{(j)}_{s+t}-1$,
	\begin{align*}
		&-\frac{\hat{n}^2_{ (i^*(\x_j),j), r}}{  \hat{\sigma}^2_{i^*(\x_j),r}(\boldsymbol{\mathrm{x}}_{j}) } + \sum_{i \ne i^{*} (\x_{j})} \frac{ \hat{n}^2_{ (i,j), r} }{ \hat{\sigma}^2_{i,r}(\boldsymbol{\mathrm{x}}_{j}) } \\
		\le & -\frac{\hat{n}^2_{ (i^*(\x_j),j), r^{(j)}_s}}{  \hat{\sigma}^2_{i^*(\x_j),r^{(j)}_s}(\boldsymbol{\mathrm{x}}_{j}) \left( 1 + O\left(\sqrt{\frac{\log \log n^{(r^{(j)}_{s})} }{ n^{(r^{(j)}_{s})} }}\right) \right) } + \sum_{i \ne i^{*} (\x_{j})} \frac{ \hat{n}^2_{ (i,j), r^{(j)}_s} }{ \hat{\sigma}^2_{i,r^{(j)}_s}(\boldsymbol{\mathrm{x}}_{j}) \left( 1 - O\left(\sqrt{\frac{\log \log n^{(r^{(j)}_{s})} }{ n^{(r^{(j)}_{s})} }}\right) \right) }  \\
		& + \frac{ 2\hat{n}_{ (i_1,j), r^{(j)}_s} + 1 }{ \hat{\sigma}^2_{i_1,r^{(j)}_{s}}(\boldsymbol{\mathrm{x}}_{j}) \left( 1 - O\left(\sqrt{\frac{\log \log n^{(r^{(j)}_{s})} }{ n^{(r^{(j)}_{s})} }}\right) \right) }  \\
		=&  -\frac{\hat{n}^2_{ (i^*(\x_j),j), r^{(j)}_s}}{  \hat{\sigma}^2_{i^*(\x_j),r^{(j)}_s}(\boldsymbol{\mathrm{x}}_{j}) } \left( 1 - O\left(\sqrt{\frac{\log \log n^{(r^{(j)}_{s})} }{ n^{(r^{(j)}_{s})} }}\right) \right)  + \sum_{i \ne i^{*} (\x_{j})} \frac{ \hat{n}^2_{ (i,j), r^{(j)}_s} }{ \hat{\sigma}^2_{i,r^{(j)}_s}(\boldsymbol{\mathrm{x}}_{j})  } \left( 1 + O\left(\sqrt{\frac{\log \log n^{(r^{(j)}_{s})} }{ n^{(r^{(j)}_{s})} }}\right) \right) \\
		&+ \frac{ 2\hat{n}_{ (i_1,j), r^{(j)}_s} + 1 }{ \hat{\sigma}^2_{i_1,r^{(j)}_{s}}(\boldsymbol{\mathrm{x}}_{j}) } \left( 1 + O\left(\sqrt{\frac{\log \log n^{(r^{(j)}_{s})} }{ n^{(r^{(j)}_{s})} }}\right) \right) \\
		=& -\frac{\hat{n}^2_{ (i^*(\x_j),j), r^{(j)}_s}}{  \hat{\sigma}^2_{i^*(\x_j),r^{(j)}_s}(\boldsymbol{\mathrm{x}}_{j}) }  + \sum_{i \ne i^{*} (\x_{j})} \frac{ \hat{n}^2_{ (i,j), r^{(j)}_s} }{ \hat{\sigma}^2_{i,r^{(j)}_s}(\boldsymbol{\mathrm{x}}_{j})  } + O\left( n^{(r^{(j)}_{s})} \sqrt{ n^{(r^{(j)}_{s})} \log \log n^{(r^{(j)}_{s})} }\right) \\
		\le & O\left( n^{(r^{(j)}_{s})} \sqrt{ n^{(r^{(j)}_{s})} \log \log n^{(r^{(j)}_{s})} } \right),
	\end{align*}
	where the last inequality holds because $(i^{r^{(j)}_s},j^{r^{(j)}_s}) = (i_1,j)$ and $i_1 \ne i^*(\x_j)$ so that $$-\hat{n}^2_{ (i^*(\x_j),j), r^{(j)}_s} /  \hat{\sigma}^2_{i^*(\x_j),r^{(j)}_s}(\boldsymbol{\mathrm{x}}_{j})  + \sum_{i \ne i^{*} (\x_{j})}  \hat{n}^2_{ (i,j), r^{(j)}_s} / \hat{\sigma}^2_{i,r^{(j)}_s}(\boldsymbol{\mathrm{x}}_{j})   \le 0.$$
	Since $\hat{n}_{ (i,j), r} = \Theta( n^{(r)} )$ by Lemma \ref{lem::order1} for all $i = 1,\dots,k$ and at the same time, $\hat{n}_{ (i,j), r} =\Theta( n^{(r^{(j)}_{s})} )$ for $r^{(j)}_s+1 \le r \le r^{(j)}_{s+t}-1$, we have $n^{(r^{(j)}_{s})} = \Theta(n^{(r)})$ for $r^{(j)}_s +1\le r \le r^{(j)}_{s+t}-1$. So
	\begin{align*}
		0<&-\frac{\hat{\alpha}^2_{ (i^*(\x_j),j), r}}{  \hat{\sigma}^2_{i^*(\x_j),r}(\boldsymbol{\mathrm{x}}_{j}) } + \sum_{i \ne i^{*} (\x_{j})} \frac{ \hat{\alpha}^2_{ (i,j), r} }{ \hat{\sigma}^2_{i,r}(\boldsymbol{\mathrm{x}}_{j}) }
		= \frac{1}{(n^{(r)})^2} \left( -\frac{\hat{n}^2_{ (i^*(\x_j),j), r}}{  \hat{\sigma}^2_{i^*(\x_j),r}(\boldsymbol{\mathrm{x}}_{j}) } + \sum_{i \ne i^{*} (\x_{j})} \frac{ \hat{n}^2_{ (i,j), r} }{ \hat{\sigma}^2_{i,r}(\boldsymbol{\mathrm{x}}_{j}) }  \right)  \\
		\le& \frac{1}{(n^{(r)})^2} O\left( n^{(r^{(j)}_{s})} \sqrt{ n^{(r^{(j)}_{s})} \log \log n^{(r^{(j)}_{s})} } \right) = O\left( \sqrt{ \frac{ \log \log n^{(r)} }{n^{(r)}}} \right).
	\end{align*}
	
	Similarly, suppose $(i^{r^{(j)}_s},j^{r^{(j)}_s}) = (i^*(\x_j),j)$ at iteration $r^{(j)}_s$ where $s \ge 1$ is large enough, $(i^{r^{(j)}_{s+h}},j^{r^{(j)}_{s+h}}) = (i_{(h)},j)$ at the successive iterations $r^{(j)}_{s+h}$ where $1 \le h \le t-1$ and $i_{(h)} \ne i^*(\x_j)$, and $(i^{r^{(j)}_{s+t}},j^{r^{(j)}_{s+t}}) = (i^*(\x_j),j)$ at iteration $r^{(j)}_{s+t}$. We would have
	$
	0>-\frac{\hat{\alpha}^2_{ (i^*(\x_j),j), r}}{  \hat{\sigma}^2_{i^*(\x_j),r}(\boldsymbol{\mathrm{x}}_{j}) } + \sum_{i \ne i^{*} (\x_{j})} \frac{ \hat{\alpha}^2_{ (i,j), r} }{ \hat{\sigma}^2_{i,r}(\boldsymbol{\mathrm{x}}_{j}) } \ge -O\left( \sqrt{ \frac{ \log \log n^{(r)} }{n^{(r)}}} \right)
	$
	for $r^{(j)}_s+1 \le r \le r^{(j)}_{s+t}-1$. Therefore,
	\begin{align*}
		&\left| -\frac{\hat{\alpha}^2_{ (i^*(\x_j),j), r}}{  \sigma^2_{i^*(\x_j)}(\boldsymbol{\mathrm{x}}_{j}) } + \sum_{i \ne i^{*} (\x_{j})} \frac{ \hat{\alpha}^2_{ (i,j), r} }{ \sigma^2_{i}(\boldsymbol{\mathrm{x}}_{j}) } \right|  \\
		\le & \left| -\frac{\hat{\alpha}^2_{ (i^*(\x_j),j), r}}{  \sigma^2_{i^*(\x_j)}(\boldsymbol{\mathrm{x}}_{j}) } + \frac{\hat{\alpha}^2_{ (i^*(\x_j),j), r}}{  \hat{\sigma}^2_{i^*(\x_j),r}(\boldsymbol{\mathrm{x}}_{j}) } \right| + \left| -\frac{\hat{\alpha}^2_{ (i^*(\x_j),j), r}}{  \hat{\sigma}^2_{i^*(\x_j),r}(\boldsymbol{\mathrm{x}}_{j}) } + \sum_{i \ne i^{*} (\x_{j})} \frac{ \hat{\alpha}^2_{ (i,j), r} }{ \hat{\sigma}^2_{i,r}(\boldsymbol{\mathrm{x}}_{j}) } \right| \\
		&+ \left| -\sum_{i \ne i^{*} (\x_{j})} \frac{ \hat{\alpha}^2_{ (i,j), r} }{ \hat{\sigma}^2_{i,r}(\boldsymbol{\mathrm{x}}_{j}) } + \sum_{i \ne i^{*} (\x_{j})} \frac{ \hat{\alpha}^2_{ (i,j), r} }{ \sigma^2_{i}(\x_{j})} \right|  \\
		= &O\left(\sqrt{\frac{ \log\log \hat{n}_{ (i^*(\x_j),j), r} }{ \hat{n}_{ (i^*(\x_j),j), r} }}\right) + O\left( \sqrt{ \frac{ \log \log n^{(r)} }{n^{(r)}}} \right) + \sum_{i\ne i^{*} (\x_{j})} O\left(\sqrt{\frac{ \log\log \hat{n}_{ (i,j), r} }{ \hat{n}_{ (i,j), r} }}\right)  \\
		=& O\left( \sqrt{ \frac{ \log \log n^{(r)} }{n^{(r)}}} \right),
	\end{align*}
	where the last equality holds because $\hat{n}_{ (i,j), r} = \Theta(n^{(r)})$ for all $i$ and $\x_j$ by Lemma \ref{lem::order1}.
\end{proof}

\begin{remark}
	The proof relies on the necessary condition for the CR$\&$S Algorithm 1 to sample non-best treatments under context $\x_j$:
	$
	\frac{\hat{n}_{(\hat{i}_r^*(\boldsymbol{\mathrm{x}}_j),j),r}^2}{\hat{\sigma}^2_{\hat{i}_r^*(\boldsymbol{\mathrm{x}}_j),r}(\boldsymbol{\mathrm{x}}_j)} \ge \sum_{i \ne \hat{i}_r^*(\boldsymbol{\mathrm{x}}_j)} \frac{\hat{n}_{(i,j),r}^2}{\hat{\sigma}_{i,r}^2(\boldsymbol{\mathrm{x}}_j)}.
	$
	We have frequently used the following tricks (can be proved by the Taylor expansion)
	$
	\left( 1 + O\left(\sqrt{\frac{\log \log n^{(r^{(j)}_{s})} }{ n^{(r^{(j)}_{s})} }}\right) \right)^{-1} =  1 - O\left(\sqrt{\frac{\log \log n^{(r^{(j)}_{s})} }{ n^{(r^{(j)}_{s})} }}\right), $ and $
	\left( 1 - O\left(\sqrt{\frac{\log \log n^{(r^{(j)}_{s})} }{ n^{(r^{(j)}_{s})} }}\right) \right)^{-1} =  1 + O\left(\sqrt{\frac{\log \log n^{(r^{(j)}_{s})} }{ n^{(r^{(j)}_{s})} }}\right).
	$
	During the proof here and below, keep in mind that
	$
	\hat{n}_{(i,j),r_1} \le \hat{n}_{(i,j),r} \le \hat{n}_{(i,j),r_2}, \text{ for } r_1 \le r \le r_2, \ i=1,\dots,k, \  j=1,\dots,m.
	$
\end{remark}

\begin{lemma}\label{lem::supple}
	For any context $\x_j$, a non-best treatment $i_1$ of $\x_j$, and any positive constant $c_2$, suppose
	{\small
		\begin{align*}
			&(i_1,j) = \mathop{\arg\min}_{j\in\{1,...,m\},i\in\{1,...,k\}\setminus\{i^*(\x_j)\}}\hat{\Vcal}_{(i,j),r} ,  \text{ and }
			s = \mathop{\arg\min}_l \left\{ l > 0: (i_1,j) = \mathop{\arg\min}_{j\in\{1,...,m\},i\in\{1,...,k\}\setminus\{i^*(\x_j)\}}\hat{\Vcal}_{(i,j),r+l} \right\}.
	\end{align*}}
	If $\hat{n}_{ (i^*(\x_j),j), r+s} - \hat{n}_{ (i^*(\x_j),j), r} > c_1 \sqrt{n^{(r)} \log \log n^{(r)}}$ where $c_1$ is a large enough value that may depend on $c_2$ but is independent of $r$, then there exists a non-best treatment $i_2$ of $\x_j$ such that
	$
	\frac{\hat{n}_{ (i_2,j), r+s}}{\hat{n}_{ (i^*(\x_j),j), r+s}} > \frac{\hat{n}_{ (i_2,j), r}}{\hat{n}_{ (i^*(\x_j),j), r}} \left( 1+c_2 \sqrt{ \frac{\log \log n^{(r)}}{n^{(r)}} }  \right)
	$
	almost surely.
\end{lemma}
\begin{proof}
	For a fixed sample path, suppose there exists $c_2>0$ such that
	\begin{align}\label{ineq::nonbd}
		\frac{\hat{n}_{ (i,j), r+s}}{\hat{n}_{ (i^*(\x_j),j), r+s}} \le \frac{\hat{n}_{ (i,j), r}}{\hat{n}_{ (i^*(\x_j),j), r}} \left( 1+c_2 \sqrt{ \frac{\log \log n^{(r)}}{n^{(r)}} }  \right)
	\end{align}
	for all $i \ne i^*(\x_j)$, no matter how large $c_1$ is. Let $c_1$ be a large positive constant whose value is to be decided and $s_1 = \mathop{\arg\max}_l \left\{ l < s: (i^{(r+l)},j^{(r+l)}) = (i^*(\x_j),j) \right\}$. Then $\hat{n}_{ (i^*(\x_j),j), r+s} = \hat{n}_{ (i^*(\x_j),j), r+s_1}+1$ because $r+s_1$ is the last iteration before $r+s$ where we sample $(i^*(\x_j),j)$. We will show that $(i^{(r+s_1)},j^{(r+s_1)}) = (i^*(\x_j),j)$ is impossible when $c_1$ is large enough and thus leads to contradiction. Before that, we need to bound some quantities.
	
	Let $r$ be large enough. By the definition of $s$, $\hat{n}_{ (i_1,j), r+s} \le \hat{n}_{ (i_1,j), r}+1$ because the only possible iteration between $r$ and $r+s-1$ where we sample $(i_1,j)$ is iteration $r$. By Lemma \ref{lem::order1}, $\hat{n}_{ (i_1,j), r+s} = \Theta(n^{(r+s)})$ and $\hat{n}_{ (i_1,j), r} = \Theta(n^{(r)})$, thus $n^{(r+s)} = \Theta(n^{(r)})$. Since $n^{(r)} \le n^{(r+s_1)} \le n^{(r+s)}$, we have $\hat{n}_{ (i^*(\x_j),j), r+s_1} = \Theta( n^{(r+s_1)} ) = \Theta( n^{(r)} )$. Let $\hat{n}_{ (i^*(\x_j),j), r+s_1} \le c_3 n^{(r)}$. By \eqref{ineq::nonbd},
	\begin{align}\label{ineq::contrad1}
		&\frac{ \hat{n}^2_{ (i,j), r+s} }{  \hat{n}^2_{ (i^*(\x_j),j), r+s} }
		< \frac{ \hat{n}^2_{ (i,j), r} }{  \hat{n}^2_{ (i^*(\x_j),j), r} }  \left( 1+ (2c_2+c_2^2c_4) \sqrt{ \frac{\log \log n^{(r)}}{n^{(r)}} }  \right),  \nonumber  \\
		&\frac{ \hat{n}^2_{ (i,j), r+s_1} }{  \hat{n}^2_{ (i^*(\x_j),j), r+s_1} } < \frac{ \hat{n}^2_{ (i,j), r+s} }{  (\hat{n}_{ (i^*(\x_j),j), r+s} - 1 )^2 } = \frac{ \hat{n}^2_{ (i,j), r+s} }{  \hat{n}_{ (i^*(\x_j),j), r+s}^2 } \frac{ \hat{n}^2_{ (i^*(\x_j),j), r+s} }{  (\hat{n}_{ (i^*(\x_j),j), r+s} - 1 )^2 }  \nonumber  \\
		&\phantom{\frac{ \hat{n}^2_{ (i,j), r+s_1} }{  \hat{n}^2_{ (i^*(\x_j),j), r+s_1} }}  <\frac{ \hat{n}^2_{ (i,j), r} }{  \hat{n}^2_{ (i^*(\x_j),j), r} }  \left( 1+ (2c_2+c_2^2c_4) \sqrt{ \frac{\log \log n^{(r)}}{n^{(r)}} }  \right) \left( 1+ c_5 \sqrt{ \frac{\log \log n^{(r)}}{n^{(r)}} }  \right)  \nonumber  \\
		&\phantom{\frac{ \hat{n}^2_{ (i,j), r+s_1} }{  \hat{n}^2_{ (i^*(\x_j),j), r+s_1} }} < \frac{ \hat{n}^2_{ (i,j), r} }{  \hat{n}^2_{ (i^*(\x_j),j), r} } \left( 1+ h(c_2,c_4,c_5) \sqrt{ \frac{\log \log n^{(r)}}{n^{(r)}} }  \right),
	\end{align}
	where $c_4$ is a small positive constant such that $c_4 \sqrt{ \log \log n^{(r)} / n^{(r)} } > \log \log n^{(r)} / n^{(r)}$ and $c_5$ is also a small positive constant. The $h(c_2,c_4,c_5)$ is a polynomial of $c_2,c_4,c_5$.
	
	At the same time, let $c_6 \sqrt{ n^{(r)} \log \log n^{(r)}  }  > 1$. Since $\hat{n}_{ (i^*(\x_j),j), r+s} - \hat{n}_{ (i^*(\x_j),j), r} > c_1 \sqrt{n^{(r)} \log \log n^{(r)}}$ and $\left(  \hat{n}_{ (i^*(\x_j),j), r+s_1} + \hat{n}_{ (i^*(\x_j),j), r}  \right) / \hat{n}_{ (i^*(\x_j),j), r+s_1} > 1 $,  we can further bound $\frac{\hat{n}^2_{ (i_1,j), r+s_1}}{ \hat{n}^2_{ (i^*(\x_j),j), r+s_1} }$ as follows.
	\begin{align}\label{ineq::bestbd}
		&\frac{\hat{n}^2_{ (i^*(\x_j),j), r} }{ \hat{n}^2_{ (i^*(\x_j),j), r+s_1} } = 1 - \frac{\hat{n}_{ (i^*(\x_j),j), r+s_1} - \hat{n}_{ (i^*(\x_j),j), r} }{ \hat{n}_{ (i^*(\x_j),j), r+s_1} } \frac{\hat{n}_{ (i^*(\x_j),j), r+s_1} + \hat{n}_{ (i^*(\x_j),j), r} }{ \hat{n}_{ (i^*(\x_j),j), r+s_1} } < 1- \frac{c_1-c_6}{c_3}\sqrt{ \frac{\log \log n^{(r)}}{n^{(r)}}},  \nonumber  \\
		&\frac{\hat{n}^2_{ (i_1,j), r+s_1}}{ \hat{n}^2_{ (i^*(\x_j),j), r+s_1} } \le \frac{(\hat{n}_{ (i_1,j), r}+1)^2}{ \hat{n}^2_{ (i^*(\x_j),j), r+s_1} } = \frac{\hat{n}_{ (i_1,j), r}^2}{ \hat{n}^2_{ (i^*(\x_j),j), r} }  \frac{\hat{n}^2_{ (i^*(\x_j),j), r}}{ \hat{n}^2_{ (i^*(\x_j),j), r+s_1} } \frac{(\hat{n}_{ (i_1,j), r}+1)^2}{ \hat{n}_{ (i_1,j), r}^2 } \nonumber \\
		&\phantom{\frac{\hat{n}^2_{ (i_1,j), r+s_1}}{ \hat{n}^2_{ (i^*(\x_j),j), r+s_1} }= \frac{(\hat{n}_{ (i_1,j), r}+1)^2}{ \hat{n}^2_{ (i^*(\x_j),j), r+s_1} }} < \frac{\hat{n}_{ (i_1,j), r}^2}{ \hat{n}^2_{ (i^*(\x_j),j), r} }  \left( 1 - \left( \frac{c_1-c_6}{c_3} -c_7\right) \sqrt{ \frac{\log \log n^{(r)}}{n^{(r)}}} \right) \nonumber   \\
		&  \phantom{\frac{\hat{n}^2_{ (i_1,j), r+s_1}}{ \hat{n}^2_{ (i^*(\x_j),j), r+s_1} }= \frac{(\hat{n}_{ (i_1,j), r}+1)^2}{ \hat{n}^2_{ (i^*(\x_j),j), r+s_1} }} \triangleq \frac{\hat{n}_{ (i_1,j), r}^2}{ \hat{n}^2_{ (i^*(\x_j),j), r} }  \left( 1 -  c_8 \sqrt{ \frac{\log \log n^{(r)}}{n^{(r)}}} \right),
	\end{align}
	where $c_7$ is a small positive constant independent of $c_1$ such that $c_7 \sqrt{ \log \log n^{(r)} / n^{(r)} } > (\hat{n}_{ (i_1,j), r}+1)^2/\hat{n}^2_{ (i_1,j), r} - 1$. By the law of the iterated logarithm, there exists $c_9>0$ such that
	\begin{align}\label{ineq::varbd}
		\left| \frac{ 1 }{ \hat{\sigma}^2_{i,r+s_1}(\boldsymbol{\mathrm{x}}_{j}) } / \frac{ 1 }{ \hat{\sigma}^2_{i,r}(\boldsymbol{\mathrm{x}}_{j}) } - 1\right| < c_9 \sqrt{ \frac{\log \log n^{(r)}}{n^{(r)}}}, \ \text{for all } i=1,\dots,k,
	\end{align}
	where we use an uniform bound to avoid the parameter redundancy. Next, we show the contradiction. At iteration $r$, since a non-best treatment of $\x_j$ is sampled,
	\begin{align}\label{ineq::iteratr}
		\frac{\hat{n}^2_{ (i^*(\x_j),j), r}}{  \hat{\sigma}^2_{i^*(\x_j),r}(\boldsymbol{\mathrm{x}}_{j}) } - \sum_{i \ne i^{*} (\x_{j})} \frac{ \hat{n}^2_{ (i,j), r} }{ \hat{\sigma}^2_{i,r}(\boldsymbol{\mathrm{x}}_{j}) } > 0.
	\end{align}
	At iteration $r+s_1$,
	\begin{align}
		&\frac{1}{ \hat{n}^2_{ (i^*(\x_j),j), r+s_1} }\left( \frac{\hat{n}^2_{ (i^*(\x_j),j), r+s_1}}{  \hat{\sigma}^2_{i^*(\x_j),r+s_1}(\boldsymbol{\mathrm{x}}_{j}) } - \sum_{i \ne i^{*} (\x_{j})} \frac{ \hat{n}^2_{ (i,j), r+s_1} }{ \hat{\sigma}^2_{i,r+s_1}(\boldsymbol{\mathrm{x}}_{j}) }  \right)  \nonumber  \\
		=&  \frac{1}{  \hat{\sigma}^2_{i^*(\x_j),r+s_1}(\boldsymbol{\mathrm{x}}_{j}) } - \sum_{\begin{subarray}{c}
				i \ne i^{*} (\x_{j}) \\
				i \ne i_1
		\end{subarray}} \frac{ 1 }{ \hat{\sigma}^2_{i,r+s_1}(\boldsymbol{\mathrm{x}}_{j}) } \frac{\hat{n}^2_{ (i,j), r+s_1}}{ \hat{n}^2_{ (i^*(\x_j),j), r+s_1} }  -  \frac{ 1 }{ \hat{\sigma}^2_{i_1,r+s_1}(\boldsymbol{\mathrm{x}}_{j}) } \frac{\hat{n}^2_{ (i_1,j), r+s_1}}{ \hat{n}^2_{ (i^*(\x_j),j), r+s_1} }  \nonumber \\
		>& \frac{1}{  \hat{\sigma}^2_{i^*(\x_j),r}(\boldsymbol{\mathrm{x}}_{j}) } \left(1 -c_9 \sqrt{ \frac{\log \log n^{(r)}}{n^{(r)}}} \right)  \nonumber  \\
		&- \sum_{\begin{subarray}{c}
				i \ne i^{*} (\x_{j}) \\
				i \ne i_1
		\end{subarray}} \frac{ 1 }{ \hat{\sigma}^2_{i,r}(\boldsymbol{\mathrm{x}}_{j}) }  \left(1 + c_9 \sqrt{ \frac{\log \log n^{(r)}}{n^{(r)}}} \right) \frac{\hat{n}^2_{ (i,j), r}}{ \hat{n}^2_{ (i^*(\x_j),j), r} } \left( 1+ h(c_2,c_4,c_5) \sqrt{ \frac{\log \log n^{(r)}}{n^{(r)}} }  \right)  \nonumber  \\
		&-  \frac{ 1 }{ \hat{\sigma}^2_{i_1,r}(\boldsymbol{\mathrm{x}}_{j}) } \left(1 + c_9 \sqrt{ \frac{\log \log n^{(r)}}{n^{(r)}}} \right)  \frac{\hat{n}^2_{ (i_1,j), r}}{ \hat{n}^2_{ (i^*(\x_j),j), r} } \left( 1 -  c_8 \sqrt{ \frac{\log \log n^{(r)}}{n^{(r)}}}  \right) \label{ineq::rs1} \\
		>& \frac{1}{  \hat{\sigma}^2_{i^*(\x_j),r}(\boldsymbol{\mathrm{x}}_{j}) } - \sum_{
			i \ne i^{*} (\x_{j}) } \frac{ 1 }{ \hat{\sigma}^2_{i,r}(\boldsymbol{\mathrm{x}}_{j}) }  \frac{\hat{n}^2_{ (i,j), r}}{ \hat{n}^2_{ (i^*(\x_j),j), r} } - \frac{c_9}{  \hat{\sigma}^2_{i^*(\x_j),r}(\boldsymbol{\mathrm{x}}_{j}) }  \sqrt{ \frac{\log \log n^{(r)}}{n^{(r)}}}  \nonumber  \\
		&- \sum_{\begin{subarray}{c}
				i \ne i^{*} (\x_{j}) \\
				i \ne i_1
		\end{subarray}} \frac{ 1 }{ \hat{\sigma}^2_{i,r}(\boldsymbol{\mathrm{x}}_{j}) }  \frac{\hat{n}^2_{ (i,j), r}}{ \hat{n}^2_{ (i^*(\x_j),j), r} } \left( c_9 + h(c_2,c_4,c_5) + c_{10} \right) \sqrt{ \frac{\log \log n^{(r)}}{n^{(r)}}}  \nonumber  \\
		&+ \frac{ 1 }{ \hat{\sigma}^2_{i_1,r}(\boldsymbol{\mathrm{x}}_{j}) }  \frac{\hat{n}^2_{ (i_1,j), r}}{ \hat{n}^2_{ (i^*(\x_j),j), r} } \left( c_8 - c_9 \right) \sqrt{ \frac{\log \log n^{(r)}}{n^{(r)}}}  \nonumber  \\
		>& \Bigg(  - \frac{c_9}{  \hat{\sigma}^2_{i^*(\x_j),r}(\boldsymbol{\mathrm{x}}_{j}) } -  \sum_{\begin{subarray}{c}
				i \ne i^{*} (\x_{j}) \\
				i \ne i_1
		\end{subarray}} \frac{ 1 }{ \hat{\sigma}^2_{i,r}(\boldsymbol{\mathrm{x}}_{j}) }  \frac{\hat{n}^2_{ (i,j), r}}{ \hat{n}^2_{ (i^*(\x_j),j), r} } \left( c_9 + h(c_2,c_4,c_5) + c_{10} \right)\nonumber \\
		&+\frac{ 1 }{ \hat{\sigma}^2_{i_1,r}(\boldsymbol{\mathrm{x}}_{j}) }  \frac{\hat{n}^2_{ (i_1,j), r}}{ \hat{n}^2_{ (i^*(\x_j),j), r} } \left( c_8 - c_9 \right) \Bigg) \sqrt{ \frac{\log \log n^{(r)}}{n^{(r)}}}, \label{ineq::rs2}
	\end{align}
	where \eqref{ineq::rs1} holds because of \eqref{ineq::contrad1}, \eqref{ineq::bestbd}, and \eqref{ineq::varbd}, $c_{10}$ is a small positive constant, and \eqref{ineq::rs2} holds because of \eqref{ineq::iteratr}.
	Then $c_2$, $c_3$, $c_4$, $c_5$, $c_6$, $c_7$,  $c_9$, $c_{10}$ are all positive constants independent of $c_1$ and $r$, $c_8 = \frac{c_1 - c_6}{c_3} - c_7$ is an increasing function of $c_1$, and $\frac{\hat{n}^2_{ (i,j), r}}{ \hat{n}^2_{ (i^*(\x_j),j), r}} = \Theta(1)$ for all treatment $i$ by Lemma \ref{lem::order1}. Therefore, no matter what the value of $r$ is, we can set $c_1$ large enough such that the lower bound \eqref{ineq::rs2} is strictly greater than 0, which means $(i^{(r+s_1)},j^{(r+s_1)}) = (i^*(\x_j),j)$ is impossible.
\end{proof}

\begin{remark}
	In Lemma \ref{lem::supple}, by definition, treatment $i_1$ under context $\x_j$ can receive at most one replication between iteration $r$ and iteration $r+s-1$. We aim to show that if treatment $i^*(\x_j)$ under context $\x_j$ receives too many replications (more than $c_1\sqrt{n^{(r)} \log \log n^{(r)}}$), there should be a non-best treatment under context $\x_j$ which also receives many replications.
	
	The proofs of Lemmas \ref{lem::supple} and \ref{lem::bdopttreat} frequently use the following tricks (the values of constant $c$ and $c'$ could change among different items in this remark).
	\begin{itemize}
		\item $\sqrt{ \log \log n^{(r)} / n^{(r)} }$ converges to zero much more slowly than $\log \log n^{(r)} / n^{(r)}$, so we can find a constant $c>0$ such that $c \sqrt{ \log \log n^{(r)} / n^{(r)} } > \log \log n^{(r)} / n^{(r)}$.
		\item $\left( 1+c \sqrt{ \frac{\log \log n^{(r)}}{n^{(r)}} }  \right) \left( 1 \pm c' \sqrt{ \frac{\log \log n^{(r)}}{n^{(r)}} }  \right) = 1+(c \pm c') \sqrt{ \frac{\log \log n^{(r)}}{n^{(r)}} } \pm cc' \frac{\log \log n^{(r)}}{n^{(r)}} $. By the first trick in this remark, $\frac{\log \log n^{(r)}}{n^{(r)}}$ can be further (lower or upper) bounded by $\sqrt{ \log \log n^{(r)} / n^{(r)} }$. This trick is used for many times, e.g., \eqref{ineq::contrad1}, the inequality below \eqref{ineq::rs1}, and \eqref{ineq::i1s2}.
		\item If $\frac{ \hat{n}_{ (i,j), r+s} }{ \hat{n}_{ (i,j), r} } > 1+c \sqrt{ \frac{\log \log n^{(r)}}{n^{(r)}} }  $, then $\frac{ \hat{n}_{ (i,j), r} }{ \hat{n}_{ (i,j), r+s} } <  1-\frac{c}{2} \sqrt{ \frac{\log \log n^{(r)}}{n^{(r)}} }  $ when $r$ is large enough. This trick is used in \eqref{ineq::ss}.
		\item $\frac{ \hat{n}^2_{ (i,j), r} }{  (\hat{n}_{ (i,j), r} - 1 )^2 } =  \left( 1 + \frac{ 1 }{  \hat{n}_{ (i,j), r} - 1  } \right)^2 \le \left( 1 + \frac{ 1 }{  n^{(r)}  } \right)^2 =  1 + \frac{ 2 }{  n^{(r)} } +  \frac{ 1 }{  (n^{(r)})^2 }  \le  1+ c \sqrt{ \frac{\log \log n^{(r)}}{n^{(r)}} }  $. This is because $\sqrt{ \log \log n^{(r)} / n^{(r)} }$ converges to zero much more slowly than $\frac{ 2 }{  n^{(r)} } +  \frac{ 1 }{  (n^{(r)})^2 }$. This trick is used for many times, e.g., the inequality above \eqref{ineq::bestbd}.
		\item $\left( 1+c \sqrt{ \frac{\log \log n^{(r)}}{n^{(r)}} }  \right) \left( 1-c' \frac{1}{n^{(r)}}   \right) = \left( 1+(1-c'')c \sqrt{ \frac{\log \log n^{(r)}}{n^{(r)}} }  \right)$ where $0 < c'< c'' < 1$ is small. This trick is used in \eqref{ineq::s2} and \eqref{ineq::ss2}.
	\end{itemize}
\end{remark}

\begin{lemma}\label{lem::bdopttreat}
	For any context $\x_j$ and a non-best treatment $i_1$ of $\x_j$, suppose
	{\small
		\begin{align*}
			(i_1,j) = \mathop{\arg\min}_{j\in\{1,...,m\},i\in\{1,...,k\}\setminus\{i^*(\x_j)\}}\hat{\Vcal}_{(i,j),r} \text{ and }
			s = \mathop{\arg\min}_l \left\{ l > 0: (i_1,j) = \mathop{\arg\min}_{j\in\{1,...,m\},i\in\{1,...,k\}\setminus\{i^*(\x_j)\}}\hat{\Vcal}_{(i,j),r+l} \right\}.
	\end{align*}}
	There exists $c_1 > 0$ independent of $r$ such that $\hat{n}_{ (i^*(\x_j),j), r+s} - \hat{n}_{ (i^*(\x_j),j), r} < c_1 \sqrt{n^{(r)} \log \log n^{(r)}}$
	almost surely.
\end{lemma}
\begin{proof}
	We prove this lemma by contradiction. Suppose for any $c_1$, there exists a large enough $r$ such that $\hat{n}_{ (i^*(\x_j),j), r+s} - \hat{n}_{ (i^*(\x_j),j), r} > c_1 \sqrt{n^{(r)} \log \log n^{(r)}}$. By Lemma \ref{lem::supple}, for any large constant $c_2$, we can find the $r$ and $c_1$ satisfying $\hat{n}_{ (i^*(\x_j),j), r+s} - \hat{n}_{ (i^*(\x_j),j), r} > c_1 \sqrt{n^{(r)} \log \log n^{(r)}}$ such that there exists a non-best treatment $i_2$ of $\x_j$ satisfying
	\begin{align}\label{ineq::s}
		\frac{\hat{n}_{ (i_2,j), r+s}}{\hat{n}_{ (i^*(\x_j),j), r+s}} > \frac{\hat{n}_{ (i_2,j), r}}{\hat{n}_{ (i^*(\x_j),j), r}} \left( 1+c_2 \sqrt{ \frac{\log \log n^{(r)}}{n^{(r)}} }  \right).
	\end{align}
	Then
	$
	\hat{n}_{ (i_2,j), r+s} >\hat{n}_{ (i_2,j), r}  \left( 1+c_2 \sqrt{ \frac{\log \log n^{(r)}}{n^{(r)}} }  \right)
	$
	because $ \hat{n}_{ (i^*(\x_j),j), r+s} / \hat{n}_{ (i^*(\x_j),j), r}  > 1$,
	so
	\begin{align}\label{ineq::ss}
		\frac{ \hat{n}_{ (i_2,j), r} }{ \hat{n}_{ (i_2,j), r+s} } < \left( 1-\frac{c_2}{2} \sqrt{ \frac{\log \log n^{(r)}}{n^{(r)}} }  \right)  \triangleq \left( 1-c_3 \sqrt{ \frac{\log \log n^{(r)}}{n^{(r)}} }  \right).
	\end{align}
	
	Let $s_2 = \mathop{\arg\max}_l \left\{ l < s: (i^{(r+l)},j^{(r+l)}) = (i_2,j) \right\}$. Then $\hat{n}_{ (i_2,j), r+s} = \hat{n}_{ (i_2,j), r+s_2}+1$. We will show that $(i^{(r+s_2)},j^{(r+s_2)}) = (i_2,j)$ is impossible when $c_2$ and $c_1$ are large enough (their values are to be decided) and thus reach a contradiction.
	
	Since $\hat{n}_{ (i^*(\x_j),j), r+s} = O(\hat{n}_{ (i^*(\x_j),j), r})$, we have
	\begin{align}
		&\frac{\hat{n}_{ (i_2,j), r+s_2}}{\hat{n}_{ (i^*(\x_j),j), r+s_2}} > \frac{\hat{n}_{ (i_2,j), r+s}-1}{\hat{n}_{ (i^*(\x_j),j), r+s}}  > \frac{\hat{n}_{ (i_2,j), r}}{\hat{n}_{ (i^*(\x_j),j), r}} \left( 1+(1-c_4)c_2 \sqrt{ \frac{\log \log n^{(r)}}{n^{(r)}} }  \right), \label{ineq::s2}  \\
		&\frac{ \hat{n}_{ (i_2,j), r} }{ \hat{n}_{ (i_2,j), r+s_2} } = \frac{ \hat{n}_{ (i_2,j), r} }{ \hat{n}_{ (i_2,j), r+s} -1} < \left( 1-(1-c_4)c_3 \sqrt{ \frac{\log \log n^{(r)}}{n^{(r)}} }  \right),  \label{ineq::ss2}
	\end{align}
	where $c_4 < 1$ is a small positive constant independent of $c_2$, \eqref{ineq::s2} holds due to \eqref{ineq::s}, and \eqref{ineq::ss2} holds due to \eqref{ineq::ss}. Again, by the law of the iterated logarithm, for $0 \le l \le s$, $i,i' \ne i^*(\x_j)$,
	\begin{align*}
		&\left| \frac{\hat{\delta}_{(i,j),r+l}}{ \hat{\delta}_{(i,j),r} } - 1\right| = \left| \left( \frac{ \bar{Y}_{i,r+l}(\x_j) - \bar{Y}_{i^*(\x_j),r+l } (\x_j)}{\bar{Y}_{i,r}(\x_j) - \bar{Y}_{i^*(\x_j),r} (\x_j) } \right)^2 - 1\right| < c_5 \sqrt{ \frac{\log \log n^{(r)}}{n^{(r)}} },  \\
		& \left| \frac{\hat{\delta}_{(i,j),r+l} - \hat{\delta}_{(i',j),r+l} }{ \hat{\delta}_{(i,j),r} - \hat{\delta}_{(i',j),r} } - 1\right|  < c_5 \sqrt{ \frac{\log \log n^{(r)}}{n^{(r)}} }, \ \ \left| \frac{ \hat{\sigma}^2_{i'',r+l}(\boldsymbol{\mathrm{x}}_{j}) }{\hat{\sigma}^2_{i'',r}(\boldsymbol{\mathrm{x}}_{j})} - 1\right| < c_5 \sqrt{ \frac{\log \log n^{(r)}}{n^{(r)}} },
	\end{align*}
	where we use an uniform bound to avoid the parameter redundancy. Since $(i^{(r)},j^{(r)}) = (i_1,j)$, $\hat{\tau}_{(i_1,j),r} < \hat{\tau}_{(i_2,j),r}$, which means
	\begin{align}\label{ineq::vr}
		\hat{\delta}_{(i_1,j),r} \left( \frac{ \hat{\sigma}^2_{i^*(\x_j),r}(\boldsymbol{\mathrm{x}}_{j}) }{\hat{n}_{ (i^*(\x_j),j), r}} + \frac{ \hat{\sigma}^2_{i_2,r}(\boldsymbol{\mathrm{x}}_{j}) }{\hat{n}_{ (i_2,j), r}} \right) < \hat{\delta}_{(i_2,j),r} \left( \frac{ \hat{\sigma}^2_{i^*(\x_j),r}(\boldsymbol{\mathrm{x}}_{j}) }{\hat{n}_{ (i^*(\x_j),j), r}} + \frac{ \hat{\sigma}^2_{i_1,r}(\boldsymbol{\mathrm{x}}_{j}) }{\hat{n}_{ (i_1,j), r}} \right).
	\end{align}
	At the iteration $r+s_2$,
	\begin{align*}
		\Lambda \triangleq&\hat{\delta}_{(i_1,j),r+s_2} \left( \frac{ \hat{\sigma}^2_{i^*(\x_j),r+s_2}(\boldsymbol{\mathrm{x}}_{j}) }{\hat{n}_{ (i^*(\x_j),j), r+s_2}} + \frac{ \hat{\sigma}^2_{i_2,r+s_2}(\boldsymbol{\mathrm{x}}_{j}) }{\hat{n}_{ (i_2,j), r+s_2}} \right) - \hat{\delta}_{(i_2,j),r+s_2} \left( \frac{ \hat{\sigma}^2_{i^*(\x_j),r+s_2}(\boldsymbol{\mathrm{x}}_{j}) }{\hat{n}_{ (i^*(\x_j),j), r+s_2}} + \frac{ \hat{\sigma}^2_{i_1,r+s_2}(\boldsymbol{\mathrm{x}}_{j}) }{\hat{n}_{ (i_1,j), r+s_2}} \right).
	\end{align*}
	We discuss the value of $\Lambda$ by two disjoint and collectively exhaustive cases.
	
	\textbf{CASE 1 $\boldsymbol{(y_{i_1}(\x_j) \ge y_{i_2}(\x_j))}$:} $\hat{\delta}_{(i_1,j),r} - \hat{\delta}_{(i_2,j),r} \ge 0$ for $r$ large enough. Then,
	\begin{align}
		\Lambda <&  \hat{\delta}_{(i_1,j),r+s_2}  \frac{ \hat{\sigma}^2_{i_2,r+s_2}(\boldsymbol{\mathrm{x}}_{j}) }{\hat{n}_{ (i_2,j), r+s_2}}  - \hat{\delta}_{(i_2,j),r+s_2} \frac{ \hat{\sigma}^2_{i_1,r+s_2}(\boldsymbol{\mathrm{x}}_{j}) }{\hat{n}_{ (i_1,j), r}+1} + \left(\hat{\delta}_{(i_1,j),r+s_2} - \hat{\delta}_{(i_2,j),r+s_2} \right)  \frac{ \hat{\sigma}^2_{i^*(\x_j),r+s_2}(\boldsymbol{\mathrm{x}}_{j}) }{\hat{n}_{ (i^*(\x_j),j), r}}  \nonumber  \\
		<& \hat{\delta}_{(i_1,j),r}  \frac{ \hat{\sigma}^2_{i_2,r}(\boldsymbol{\mathrm{x}}_{j}) }{\hat{n}_{ (i_2,j), r}} \left( 1+(2c_5+c_6) \sqrt{ \frac{\log \log n^{(r)}}{n^{(r)}} } \right) \left( 1- (1-c_4)c_3 \sqrt{ \frac{\log \log n^{(r)}}{n^{(r)}} } \right)   \label{ineq::applyss2}\\
		& - \hat{\delta}_{(i_2,j),r} \frac{ \hat{\sigma}^2_{i_1,r}(\boldsymbol{\mathrm{x}}_{j}) }{\hat{n}_{ (i_1,j), r}} \left( 1-(2c_5+c_7) \sqrt{ \frac{\log \log n^{(r)}}{n^{(r)}} } \right) \nonumber\\
		&+ \left(\hat{\delta}_{(i_1,j),r} - \hat{\delta}_{(i_2,j),r} \right)  \frac{ \hat{\sigma}^2_{i^*(\x_j),r}(\boldsymbol{\mathrm{x}}_{j}) }{\hat{n}_{ (i^*(\x_j),j), r}} \left( 1+(2c_5+c_6) \sqrt{ \frac{\log \log n^{(r)}}{n^{(r)}} } \right)   \nonumber  \\
		<&\hat{\delta}_{(i_1,j),r}  \frac{ \hat{\sigma}^2_{i_2,r}(\boldsymbol{\mathrm{x}}_{j}) }{\hat{n}_{ (i_2,j), r}} - \hat{\delta}_{(i_2,j),r} \frac{ \hat{\sigma}^2_{i_1,r}(\boldsymbol{\mathrm{x}}_{j}) }{\hat{n}_{ (i_1,j), r}} + \left(\hat{\delta}_{(i_1,j),r} - \hat{\delta}_{(i_2,j),r} \right)  \frac{ \hat{\sigma}^2_{i^*(\x_j),r}(\boldsymbol{\mathrm{x}}_{j}) }{\hat{n}_{ (i^*(\x_j),j), r}}  \nonumber  \\
		& + \left(c^2_{10} \frac{2c_5+c_6}{c_8} + c^2_{10} \frac{2c_5+c_7}{c_8} + c^2_{10} \frac{2c_5+c_6}{c_8} - c_9^2 (1-c_4)c_3 \right) \frac{1}{n^{(r)}}\sqrt{ \frac{\log \log n^{(r)}}{n^{(r)}} } \label{ineq::bdtriv}  \\
		<& \left(c^2_{10} \frac{2c_5+c_6}{c_8} + c^2_{10} \frac{2c_5+c_7}{c_8} + c^2_{10} \frac{2c_5+c_6}{c_8} - c_9^2 (1-c_4)c_3 \right) \frac{1}{n^{(r)}}\sqrt{ \frac{\log \log n^{(r)}}{n^{(r)}} }, \label{ineq::bdr}
	\end{align}
	where $c_6$ and $c_7$ are small positive constants and \eqref{ineq::applyss2} holds because of \eqref{ineq::ss2}. According to Lemma \ref{lem::order1}, we set $\hat{n}_{ (i,j), r} > c_8 n^{(r)}$ for all $i=1,\dots,k$. $c_{9}>0$ is an uniform lower bound of $\hat{\delta}_{(i_1,j),r}$ and  $ \hat{\sigma}^2_{i_2,r}(\boldsymbol{\mathrm{x}}_{j})$ in \eqref{ineq::bdtriv}.  $c_{10}$ is an uniform upper bound of $\hat{\delta}_{(i,j),r}$, $ \hat{\sigma}^2_{i,r}(\boldsymbol{\mathrm{x}}_{j})$, $i=1,\dots,k$, and    $\hat{\delta}_{(i_1,j),r} - \hat{\delta}_{(i_2,j),r} $ in \eqref{ineq::bdtriv}. Inequality \eqref{ineq::bdr} holds because of \eqref{ineq::vr}. Since $c_4$, $c_5$, $c_6$, $c_7$, $c_8$, $c_9$, $c_{10}$ are positive constants independent of $c_2$, the upper bound \eqref{ineq::bdr} is strictly smaller than 0 when $c_3=c_2/2$ is large enough, suggesting $\hat{\tau}_{(i_1,j),r+s_2} < \hat{\tau}_{(i_2,j),r+s_2}$. Therefore $(i^{(r+s_2)},j^{(r+s_2)}) = (i_2,j)$ is impossible when $c_1$ is large enough, contradicting the definition of $s_2$.
	
	\textbf{CASE 2 $\boldsymbol{(y_{i_1}(\x_j) < y_{i_2}(\x_j))}$:} $\hat{\delta}_{(i_1,j),r} - \hat{\delta}_{(i_2,j),r} < 0$ for $r$ large enough. We need to further bound some quantities. First,
	$
	\frac{ \hat{n}_{ (i^*(\x_j),j), r+s_2} }{\hat{n}_{ (i_1,j), r+s_2}} > \frac{ \hat{n}_{ (i^*(\x_j),j), r} }{\hat{n}_{ (i_1,j), r}+1} > \frac{ \hat{n}_{ (i^*(\x_j),j), r} }{\hat{n}_{ (i_1,j), r}} \left( 1-c_{11} \sqrt{ \frac{\log \log n^{(r)}}{n^{(r)}} } \right),
	$
	where $c_{11}$ is a positive constant such that $c_{11} \sqrt{ \frac{\log \log n^{(r)}}{n^{(r)}} } > \frac{1}{ \hat{n}_{ (i_1,j), r}+1 } $. Then, combining \eqref{ineq::s2},
	\begin{align}
		&\frac{ \hat{n}_{ (i_2,j), r+s_2} }{\hat{n}_{ (i_1,j), r+s_2}} = \frac{ \hat{n}_{ (i_2,j), r+s_2} }{\hat{n}_{ (i^*(\x_j),j), r+s_2}} \frac{ \hat{n}_{ (i^*(\x_j),j), r+s_2} }{\hat{n}_{ (i_1,j), r+s_2}}  \nonumber  \\
		&\phantom{nnn}> \frac{\hat{n}_{ (i_2,j), r}}{\hat{n}_{ (i^*(\x_j),j), r}} \left( 1+(1-c_4)c_2 \sqrt{ \frac{\log \log n^{(r)}}{n^{(r)}} }  \right)  \frac{ \hat{n}_{ (i^*(\x_j),j), r} }{\hat{n}_{ (i_1,j), r}} \left( 1-c_{11} \sqrt{ \frac{\log \log n^{(r)}}{n^{(r)}} } \right)  \nonumber  \\
		&\phantom{nnn}> \frac{\hat{n}_{ (i_2,j), r}}{\hat{n}_{ (i_1,j), r}} \left( 1+(1-c_{12})c_2  \sqrt{ \frac{\log \log n^{(r)}}{n^{(r)}} } \right), \label{ineq::i1s2}
	\end{align}
	where $c_{4} < c_{12} < 1 $ is a small positive constant independent of $c_2$. Thus,
	\begin{align*}
		&\hat{n}_{ (i_2,j), r+s_2} \Lambda < \hat{\delta}_{(i_1,j),r+s_2}  \hat{\sigma}^2_{i_2,r+s_2}(\boldsymbol{\mathrm{x}}_{j})   - \hat{\delta}_{(i_2,j),r+s_2} \hat{\sigma}^2_{i_1,r+s_2}(\boldsymbol{\mathrm{x}}_{j}) \frac{ \hat{n}_{ (i_2,j), r+s_2} }{\hat{n}_{ (i_1,j), r+s_2}}  \nonumber  \\
		&\phantom{\hat{n}_{ (i_2,j), r+s_2} \Lambda <} + \left(\hat{\delta}_{(i_1,j),r+s_2} - \hat{\delta}_{(i_2,j),r+s_2} \right)  \hat{\sigma}^2_{i^*(\x_j),r+s_2}(\boldsymbol{\mathrm{x}}_{j}) \frac{ \hat{n}_{ (i_2,j), r+s_2} }{\hat{n}_{ (i^*(\x_j),j), r+s_2}}  \nonumber \\
		<& \hat{\delta}_{(i_1,j),r}  \hat{\sigma}^2_{i_2,r}(\boldsymbol{\mathrm{x}}_{j}) \left( 1+(2c_5+c_6) \sqrt{ \frac{\log \log n^{(r)}}{n^{(r)}} } \right) \nonumber \\
		&- \hat{\delta}_{(i_2,j),r} \hat{\sigma}^2_{i_1,r}(\boldsymbol{\mathrm{x}}_{j}) \frac{\hat{n}_{ (i_2,j), r}}{\hat{n}_{ (i_1,j), r}} \left( 1-2c_5 \sqrt{ \frac{\log \log n^{(r)}}{n^{(r)}} } \right) \left( 1+(1-c_{12})c_2 \sqrt{ \frac{\log \log n^{(r)}}{n^{(r)}} } \right)  \nonumber  \\
		&+ \left(\hat{\delta}_{(i_1,j),r} - \hat{\delta}_{(i_2,j),r} \right)  \hat{\sigma}^2_{i^*(\x_j),r}(\boldsymbol{\mathrm{x}}_{j}) \frac{\hat{n}_{ (i_2,j), r}}{\hat{n}_{ (i^*(\x_j),j), r}} \left( 1-2c_5 \sqrt{ \frac{\log \log n^{(r)}}{n^{(r)}} } \right) \left( 1+(1-c_4)c_2 \sqrt{ \frac{\log \log n^{(r)}}{n^{(r)}} }  \right)  \nonumber  \\
		<& \hat{\delta}_{(i_1,j),r}  \hat{\sigma}^2_{i_2,r}(\boldsymbol{\mathrm{x}}_{j}) - \hat{\delta}_{(i_2,j),r} \hat{\sigma}^2_{i_1,r}(\boldsymbol{\mathrm{x}}_{j}) \frac{\hat{n}_{ (i_2,j), r}}{\hat{n}_{ (i_1,j), r}} + \left(\hat{\delta}_{(i_1,j),r} - \hat{\delta}_{(i_2,j),r} \right)  \hat{\sigma}^2_{i^*(\x_j),r}(\boldsymbol{\mathrm{x}}_{j}) \frac{\hat{n}_{ (i_2,j), r}}{\hat{n}_{ (i^*(\x_j),j), r}}  \nonumber  \\
		&+ \left( c^2_{10}(2c_5+c_6) -c_{13}  (1-c_{14})c_2  - c_{15} (1-c_{16})c_2  \right)  \sqrt{ \frac{\log \log n^{(r)}}{n^{(r)}} }  \nonumber  \\
		<& \left( c^2_{10}(2c_5+c_6) -c_{13}  (1-c_{14})c_2  - c_{15} (1-c_{16})c_2  \right)  \sqrt{ \frac{\log \log n^{(r)}}{n^{(r)}} },
	\end{align*}
	where $c_{13}$, $c_{12}<c_{14}<1$, $c_{15}$, and $c_{4} < c_{16} <1$ are positive constants independent of $c_2$. We can set $c_2$ large enough such that the bound of $\hat{n}_{ (i_2,j), r+s_2} \Lambda$  above is strictly smaller than 0. Then we again conclude that $(i^{(r+s_2)},j^{(r+s_2)}) = (i_2,j)$ is impossible when $c_2$ is large enough, contradicting the definition of $s_2$.
\end{proof}

\begin{proposition}\label{propos2}
	For $\hat{\alpha}_{i,j}$ generated by CR$\&$S Algorithm 1, $i=1,2,...,k$ and $j=1,2,...,m$, we have
	\begin{align*}
		&\Bigg| \frac{(y_{i} (\x_{j})-y_{i^*(\x_{j})} (\x_{j}))^2} {\sigma_{i^*(\x_{j})}^2(\x_{j})/\hat{\alpha}_{i^*(\x_{j}),j}+\sigma_{i}^2(\x_{j})/\hat{\alpha}_{i,j}}-
		\frac{(y_{i'} (\x_{j'})-y_{i^*(\x_{j'})} (\x_{j'}))^2} {\sigma_{i^*(\x_{j'})}^2(\x_{j'})/\hat{\alpha}_{i^*(\x_{j'}),j'}+\sigma_{i'}^2(\x_{j'})/\hat{\alpha}_{i',j'}} \Bigg|
		\leq O\left(\sqrt{ \frac{ \log \log n^{(r)}}{n^{(r)}} }\right),
	\end{align*}
	almost surely, for any $j,j' = 1,2,\cdots,m$, $i,i' = 1,2,\cdots,k$, $i \ne i^*(\boldsymbol{\mathrm{x}}_{j})$ and $i' \ne i^*(\boldsymbol{\mathrm{x}}_{j'})$.
\end{proposition}
\begin{proof}
	Fix a sample path and let $\{ r_l^{(i,j)}, l=1,2,\dots \}$ be the subsequence of all iterations that
	$(i,j) = \mathop{\arg\min}_{j\in\{1,2,...,m\},i\in\{1,...,k\}\setminus\{i^*(\x_j)\}}\hat{\Vcal}_{(i,j),r_l^{(i,j)}}.$
	Let $r_1 = r_{l_1}^{(i,j)}$ where $l_1 \ge 1$ is any large enough integer and $r_2 = r_{l_1+1}^{(i,j)}$. By Lemma \ref{lem::bdopttreat}, we have $\hat{n}_{ (i^*(\x_j),j), r_2} - \hat{n}_{ (i^*(\x_j),j), r_1} = O (\sqrt{n^{(r_1)} \log \log n^{(r_1)}})$. Therefore, $\hat{n}_{ (i^*(\x_j),j), r_2} = \hat{n}_{ (i^*(\x_j),j), r_1} (1+O(\sqrt{ \log \log n^{(r_1)} / n^{(r_1)} }))$.  Since $\hat{n}_{ (i,j), r_2} \le \hat{n}_{ (i,j), r_1} + 1$, $\hat{n}_{ (i,j), r_1} = \Theta(n^{(r_1)}) $, and $\hat{n}_{ (i,j), r_2} = \Theta(n^{(r_2)}) $, we have $n^{(r)} = O(n^{(r_1)})$ for all $r_1 \le r \le r_2$. By the law of the iterated logarithm, for $i=1,\dots,k$ and $j=1,\dots,m$,
	$
	\left| \hat{\delta}_{(i,j),r}/\hat{\delta}_{(i,j),r_1} - 1 \right| \le O\left(\sqrt{\frac{\log \log n^{(r_1)}}{n^{(r_1)}}}\right), \ \left| \hat{\sigma}^2_{i,r}(\boldsymbol{\mathrm{x}}_{j}) / \hat{\sigma}^2_{i,r_1}(\boldsymbol{\mathrm{x}}_{j}) - 1 \right| \le  O\left(\sqrt{\frac{\log \log n^{(r_1)}}{n^{(r_1)}}}\right).
	$
	Let $(i',j') \ne (i,j)$. For $r_1 \le r \le r_2$,
	\begin{align}
		&\hat{\tau}_{(i,j),r}-\hat{\tau}_{(i',j'),r} = \frac{ \hat{\delta}_{(i,j),r} }{\frac{ \hat{\sigma}^2_{i^*(\x_{j}),r}(\boldsymbol{\mathrm{x}}_{j}) }{\hat{n}_{ (i^*(\x_{j}),j), r}} + \frac{ \hat{\sigma}^2_{i,r}(\boldsymbol{\mathrm{x}}_{j}) }{\hat{n}_{ (i,j), r}} } - \frac{\hat{\delta}_{(i',j'),r}}{ \frac{ \hat{\sigma}^2_{i^*(\x_{j'}),r}(\boldsymbol{\mathrm{x}}_{j'}) }{\hat{n}_{ (i^*(\x_{j'}),j'), r}} + \frac{ \hat{\sigma}^2_{i',r}(\boldsymbol{\mathrm{x}}_{j'}) }{\hat{n}_{ (i',j'), r}} }  \nonumber  \\
		&\le \hat{\delta}_{(i,j),r} \left( \frac{ \hat{\sigma}^2_{i^*(\x_{j}),r}(\boldsymbol{\mathrm{x}}_{j}) }{\hat{n}_{ (i^*(\x_{j}),j), r_2}} + \frac{ \hat{\sigma}^2_{i,r}(\boldsymbol{\mathrm{x}}_{j}) }{\hat{n}_{ (i,j), r_2}} \right)^{-1} - \hat{\delta}_{(i',j'),r} \left( \frac{ \hat{\sigma}^2_{i^*(\x_{j'}),r}(\boldsymbol{\mathrm{x}}_{j'}) }{\hat{n}_{ (i^*(\x_{j'}),j'), r_1}} + \frac{ \hat{\sigma}^2_{i',r}(\boldsymbol{\mathrm{x}}_{j'}) }{\hat{n}_{ (i',j'), r_1}} \right)^{-1}  \nonumber  \\
		&\le \frac{ \hat{\delta}_{(i,j),r_1} \left( 1+O\left(\sqrt{\frac{\log \log n^{(r_1)}}{n^{(r_1)}}}\right) \right) }{  \frac{ \hat{\sigma}^2_{i^*(\x_{j}),r_1}(\boldsymbol{\mathrm{x}}_{j}) }{\hat{n}_{ (i^*(\x_{j}),j), r_1}} \frac{ 1- O\left(\sqrt{\frac{\log \log n^{(r_1)}}{n^{(r_1)}}}\right) }{ 1+ O\left(\sqrt{\frac{\log \log n^{(r_1)}}{n^{(r_1)}}}\right) } + \frac{ \hat{\sigma}^2_{i,r_1}(\boldsymbol{\mathrm{x}}_{j}) }{\hat{n}_{ (i,j), r_1}} \frac{ 1- O\left(\sqrt{\frac{\log \log n^{(r_1)}}{n^{(r_1)}}}\right) }{ 1+ O\left(\frac{1}{n^{(r_1)}}\right) } }  \nonumber  \\
		&\phantom{=} - \frac{ \hat{\delta}_{(i',j'),r_1} \left( 1- O\left(\sqrt{\frac{\log \log n^{(r_1)}}{n^{(r_1)}}}\right) \right) }{  \frac{ \hat{\sigma}^2_{i^*(\x_{j'}),r_1}(\boldsymbol{\mathrm{x}}_{j'}) }{\hat{n}_{ (i^*(\x_{j'}),j'), r_1}} \left( 1+ O\left(\sqrt{\frac{\log \log n^{(r_1)}}{n^{(r_1)}}}\right) \right) + \frac{ \hat{\sigma}^2_{i',r_1}(\boldsymbol{\mathrm{x}}_{j'}) }{\hat{n}_{ (i',j'), r_1}} \left( 1+ O\left(\sqrt{\frac{\log \log n^{(r_1)}}{n^{(r_1)}}}\right) \right) }  \nonumber  \\
		&\le \frac{ \hat{\delta}_{(i,j),r_1} \left( 1+O\left(\sqrt{\frac{\log \log n^{(r_1)}}{n^{(r_1)}}}\right) \right) }{  \frac{ \hat{\sigma}^2_{i^*(\x_{j}),r_1}(\boldsymbol{\mathrm{x}}_{j}) }{\hat{n}_{ (i^*(\x_{j}),j), r_1}}  \left( 1- O\left(\sqrt{\frac{\log \log n^{(r_1)}}{n^{(r_1)}}}\right) \right)  + \frac{ \hat{\sigma}^2_{i,r_1}(\boldsymbol{\mathrm{x}}_{j}) }{\hat{n}_{ (i,j), r_1}}  \left( 1- O\left(\sqrt{\frac{\log \log n^{(r_1)}}{n^{(r_1)}}}\right) \right)  }  \nonumber  \\
		&\phantom{=}  - \frac{ \hat{\delta}_{(i',j'),r_1} \left( 1- O\left(\sqrt{\frac{\log \log n^{(r_1)}}{n^{(r_1)}}}\right) \right) }{  \frac{ \hat{\sigma}^2_{i^*(\x_{j'}),r_1}(\boldsymbol{\mathrm{x}}_{j'}) }{\hat{n}_{ (i^*(\x_{j'}),j'), r_1}} \left( 1+ O\left(\sqrt{\frac{\log \log n^{(r_1)}}{n^{(r_1)}}}\right) \right) + \frac{ \hat{\sigma}^2_{i',r_1}(\boldsymbol{\mathrm{x}}_{j'}) }{\hat{n}_{ (i',j'), r_1}} \left( 1+ O\left(\sqrt{\frac{\log \log n^{(r_1)}}{n^{(r_1)}}}\right) \right) }  \nonumber \\
		&\le \frac{ \hat{\delta}_{(i,j),r_1}  }{  \frac{ \hat{\sigma}^2_{i^*(\x_{j}),r_1}(\boldsymbol{\mathrm{x}}_{j}) }{\hat{n}_{ (i^*(\x_{j}),j), r_1}} + \frac{ \hat{\sigma}^2_{i,r_1}(\boldsymbol{\mathrm{x}}_{j}) }{\hat{n}_{ (i,j), r_1}} }  \left( 1+O\left(\sqrt{\frac{\log \log n^{(r_1)}}{n^{(r_1)}}}\right) \right) \nonumber  \\
		&\phantom{=} - \frac{ \hat{\delta}_{(i',j'),r_1} }{  \frac{ \hat{\sigma}^2_{i^*(\x_{j'}),r_1}(\boldsymbol{\mathrm{x}}_{j'}) }{\hat{n}_{ (i^*(\x_{j'}),j'), r_1}} + \frac{ \hat{\sigma}^2_{i',r_1}(\boldsymbol{\mathrm{x}}_{j'}) }{\hat{n}_{ (i',j'), r_1}}  } \left( 1- O\left(\sqrt{\frac{\log \log n^{(r_1)}}{n^{(r_1)}}}\right) \right) \nonumber \\
		&= \frac{ \hat{\delta}_{(i,j),r_1}  }{  \frac{ \hat{\sigma}^2_{i^*(\x_{j}),r_1}(\boldsymbol{\mathrm{x}}_{j}) }{\hat{n}_{ (i^*(\x_{j}),j), r_1}} + \frac{ \hat{\sigma}^2_{i,r_1}(\boldsymbol{\mathrm{x}}_{j}) }{\hat{n}_{ (i,j), r_1}} } - \frac{ \hat{\delta}_{(i',j'),r_1} }{  \frac{ \hat{\sigma}^2_{i^*(\x_{j'}),r_1}(\boldsymbol{\mathrm{x}}_{j'}) }{\hat{n}_{ (i^*(\x_{j'}),j'), r_1}} + \frac{ \hat{\sigma}^2_{i',r_1}(\boldsymbol{\mathrm{x}}_{j'}) }{\hat{n}_{ (i',j'), r_1}}  } + O\left(\sqrt{n^{(r_1)} \log \log n^{(r_1)}}\right)  \nonumber \\
		&< O\left(\sqrt{n^{(r_1)} \log \log n^{(r_1)}}\right), \nonumber
	\end{align}
	where the last inequality holds because of the definition of $r_1$. Then,
	$
	\mathcal{\hat{V}}_{(i,j),r} - \mathcal{\hat{V}}_{(i',j'),r} = \frac{1}{n^{(r)}} ( \hat{\tau}_{(i,j),r}$ $-\hat{\tau}_{(i',j'),r} ) \le  O\left(\sqrt{\frac{\log \log n^{(r)}}{n^{(r)}}}\right)
	$
	because $n^{(r_1)} = \Theta(n^{(r)})$. By symmetry, we would have $\mathcal{\hat{V}}_{(i',j'),r} - \mathcal{\hat{V}}_{(i,j),r} \le  O\left(\sqrt{\frac{\log \log n^{(r)}}{n^{(r)}}}\right)$. Therefore, applying the law of the iterated logarithm,
	\begin{align*}
		&\Bigg| \frac{(y_{i} (\x_{j})-y_{i^*(\x_{j})} (\x_{j}))^2} {\sigma_{i^*(\x_{j})}^2(\x_{j})/\hat{\alpha}_{(i^*(\x_{j}),j),r}+\sigma_{i}^2(\x_{j})/\hat{\alpha}_{(i,j),r}}-
		\frac{(y_{i'} (\x_{j'})-y_{i^*(\x_{j'})} (\x_{j'}))^2} {\sigma_{i^*(\x_{j'})}^2(\x_{j'})/\hat{\alpha}_{(i^*(\x_{j'}),j'),r}+\sigma_{i'}^2(\x_{j'})/\hat{\alpha}_{(i',j'),r}} \Bigg|  \\
		\le& \left| \frac{(y_{i} (\x_{j})-y_{i^*(\x_{j})} (\x_{j}))^2} {\sigma_{i^*(\x_{j})}^2(\x_{j})/\hat{\alpha}_{(i^*(\x_{j}),j),r}+\sigma_{i}^2(\x_{j})/\hat{\alpha}_{(i,j),r}} - \mathcal{\hat{V}}_{(i,j),r} \right| + \left| \mathcal{\hat{V}}_{(i,j),r} -  \mathcal{\hat{V}}_{(i',j'),r} \right|  \\
		& + \left| \mathcal{\hat{V}}_{(i',j'),r} - \frac{(y_{i'} (\x_{j'})-y_{i^*(\x_{j'})} (\x_{j'}))^2} {\sigma_{i^*(\x_{j'})}^2(\x_{j'})/\hat{\alpha}_{(i^*(\x_{j'}),j'),r}+\sigma_{i'}^2(\x_{j'})/\hat{\alpha}_{(i',j'),r}} \right|
		\le O\left(\sqrt{\frac{\log \log n^{(r)}}{n^{(r)}}}\right).
	\end{align*}
\end{proof}

\subsection{Proof of Lemma 1}

Since $\bar{Y}^L_{i_1}(\x_h^\circ) = y_{i_1}(\x_h^\circ) + \frac{1}{n_{i_1,h}} \sum_{l=1}^{n_{i_1,h}} \epsilon_{i_1l} (\x_h^\circ) = \bbf(\x_h^\circ)^\top \bbeta_{i_1} + \frac{1}{n_{i_1,h}} \sum_{l=1}^{n_{i_1,h}} \epsilon_{i_1l} (\x_h^\circ)$,
we have
$
\widehat{\bbeta}_{i_1} = \left( \F^\top \F \right)^{-1} \F^\top \left( \F \bbeta_{i_1} + \overline{\bepsilon}_{i_1} \right) = \bbeta_{i_1} + \sum_{h=1}^{p} \left( \F^\top \F \right)^{-1} \bbf(\x_h^\circ) \overline{\epsilon}_{i_1}(\x_h^\circ),$ $i_1=i,i^*(\x).
$
Then the cumulant generating function of $\widehat{\bbeta}_i - \widehat{\bbeta}_{i^*(\x)}$ is
\begin{align*}
	&\Psi_n^{i,i^*(\x)}(\btheta) = \log \EE\left( e^{\btheta^\top (\widehat{\bbeta}_i - \widehat{\bbeta}_{i^*(\x)})} \right)  = \btheta^\top (\bbeta_i-\bbeta_{i^*(\x)}) + \log \EE\left( e^{\sum_{h=1}^{p} \btheta^\top \left( \F^\top \F \right)^{-1} \F^\top (\overline{\epsilon}_i(\x_h^\circ) - \overline{\epsilon}_{i^*(\x)}(\x_h^\circ) )} \right).
\end{align*}
Note that $\overline{\epsilon}_{i_1}(\x_h^\circ)$ is normally distributed as $\Ncal(0,\sigma_{i_1}^2(\x_h^\circ)/n_{i_1,h})$, $i_1=i,i^*(\x)$. Thus,
\begin{align}\label{eq:linmoment}
	\log \EE\left( e^{\pm \btheta^\top \left( \F^\top \F \right)^{-1}  \bbf(\x_h^\circ) \overline{\epsilon}_{i_1}(\x_h^\circ)} \right) =  \frac{\sigma_{i_1}^2(\x_h^\circ) \left( \btheta^\top \left( \F^\top \F \right)^{-1}  \bbf(\x_h^\circ) \right)^2}{2n_{i_1,h}}, \quad i_1=i,i^*(\x).
\end{align}	
Plugging \eqref{eq:linmoment} into the cumulant generating function $\Psi_n^{i,i^*(\x)}(\btheta)$, we have
$
\Psi_n^{i,i^*(\x)}(\btheta) = \btheta^\top (\bbeta_i-\bbeta_{i^*(\x)}) + \frac{1}{2n}   \btheta^\top \left( \F^\top \F \right)^{-1}  \F^\top \left( \bSigmap_{\epsilon,i} + \bSigmap_{\epsilon,i^*(\x)} \right) \F \left( \F^\top \F \right)^{-1}  \btheta.
$
Then, define	$\Psi^{i,i^*(\x)}(\btheta) \triangleq \lim\limits_{n \to \infty} \frac{1}{n} \Psi_n^{i,i^*(\x)}(n \btheta) $.
The Fenchel-Legendre transform of $\Psi^{i,i^*(\x)}(\btheta)$ is
\begin{equation*}
	\begin{aligned}
		&I^{i,i^*(\x)}(\boldsymbol{u}) = \sup_{\btheta} \left\{ \btheta^\top \boldsymbol{u} - \Psi^{i,i^*(\x)}(\btheta) \right\} \\
		&= \sup_{\btheta} \left\{ \btheta^\top \left(\boldsymbol{u} - (\bbeta_i-\bbeta_{i^*(\x)})\right) - \frac{1}{2}   \btheta^\top \left( \F^\top \F \right)^{-1}  \F^\top \left( \bSigmap_{\epsilon,i} + \bSigmap_{\epsilon,i^*(\x)} \right) \F \left( \F^\top \F \right)^{-1}  \btheta \right\} \\
		&= \frac{1}{2}\left(\boldsymbol{u} - (\bbeta_i-\bbeta_{i^*(\x)})\right)^\top \left( \left( \F^\top \F \right)^{-1}  \F^\top \left( \bSigmap_{\epsilon,i} + \bSigmap_{\epsilon,i^*(\x)} \right) \F \left( \F^\top \F \right)^{-1} \right)^{-1} \left(\boldsymbol{u} - (\bbeta_i-\bbeta_{i^*(\x)})\right)
	\end{aligned}
\end{equation*}
Let $A = \{ \boldsymbol{u}: \bbf(\x)^\top \boldsymbol{u} \le 0 \}$. By G\"artner-Ellis Theorem,
$
\lim\limits_{n \to \infty} \frac{1}{n} \log \PP \left( \bbf(\x)^\top (\hat{\bbeta}_i - \hat{\bbeta}_{i^*(\x)}) \le 0 \right)
=\lim\limits_{n \to \infty} \frac{1}{n} \log \PP \left( ( \hat{\bbeta}_i - \hat{\bbeta}_{i^*(\x)}) \in A  \right)
= - \inf\limits_{\boldsymbol{u} \in \{ \boldsymbol{u}: \bbf(\x)^\top \boldsymbol{u} \le 0 \}}  I^{i,i^*(\x)}(\boldsymbol{u}).
$
Since
$
\min\limits_{\boldsymbol{u}} I^{i,i^*(\x)}(\boldsymbol{u})  ~
s.t.~\bbf(\x)^\top \boldsymbol{u} \le 0
$
is a convex programming problem, we can find the optimal solution by checking the optimality condition. This lemma is proved.

% (10) is eq:reprob
\subsection{Proof of Theorem 4}

Model (8) of the main paper is equivalent to
\begin{equation}\label{eq:reprob2}
	\begin{aligned}
		\min ~ & z  \\
		s.t. ~ & \frac{1}{\Gcal_{i^*(\x_j),i,j}^{L}(\alphab)} \le z, \  i=1,2,...,k \text{ and }i\neq i^*(\x_j), j=1,2,...,m,  \\
		& \sum_{i=1}^k \sum_{h=1}^p \alpha_{i,h}=1, \  \alpha_{i,h}\geq0, ~ i=1,2,...,k, \ h=1,2,...,p.
	\end{aligned}
\end{equation}
We will analyze \eqref{eq:reprob2} instead of (8) of the main paper because $\frac{1}{\Gcal_{i^*(\x_j),i,j}^{L}(\alphab)}$ is more tractable.
The Lagrangian of \eqref{eq:reprob2} is
\begin{align*}
	L(\alphab,z,\blam,\theta) =& z + \sum_{j=1}^m \sum_{i=1, i\ne i^*(\x_j)}^k \lambda_{i,j} \left( \frac{1}{\Gcal_{i^*(\x_j),i,j}^{L}(\alphab)} - z \right) + \theta \left( \sum_{i=1}^k \sum_{h=1}^p \alpha_{i,h}-1 \right).
\end{align*}
Note that
\begin{align*}
	&\sum_{j=1}^m \sum_{i=1, i\ne i^*(\x_j)}^k \lambda_{i,j} \frac{2 \bbf(\x_j)^\top \left( \F^\top \F \right)^{-1}  \F^\top \left(\bSigmap_{\epsilon,i}+\bSigmap_{\epsilon,i^*(\x_j)}\right) \F \left( \F^\top \F \right)^{-1} \bbf(\x_j) }{\left[ \bbf(\x_j)^\top (\bbeta_i - \bbeta_{i^*(\x_j)}) \right]^2 }   \\
	=&  \sum_{j=1}^m \sum_{i=1, i\ne i^*(\x_j)}^k \frac{2\lambda_{i,j}}{ \left[ \bbf(\x_j)^\top (\bbeta_i - \bbeta_{i^*(\x_j)}) \right]^2 } \sum_{h=1}^{p} \left( \bbf(\x_{j})^\top \left( \F^\top \F \right)^{-1} \bbf(\x_{h}^\circ) \right)^2 \left( \frac{\sigma^2_i(\x_{h}^\circ)}{\alpha_{i,h}} + \frac{\sigma^2_{i^*(\x_j)}(\x_{h}^\circ)}{\alpha_{i^*(\x_j),h}} \right)  \\
	=& \sum_{i'=1}^k \Bigg( \sum_{j \in \Ccal_{i'}} \sum_{i=1,i\ne i'}^{k} \frac{2\lambda_{i,j}}{ \left[ \bbf(\x_j)^\top (\bbeta_i - \bbeta_{i'}) \right]^2 } \sum_{h=1}^{p} \left( \bbf(\x_{j})^\top \left( \F^\top \F \right)^{-1} \bbf(\x_{h}^\circ) \right)^2  \frac{\sigma^2_{i'}(\x_{h}^\circ)}{\alpha_{i',h}} \\
	& \phantom{\sum\Bigg(} + \sum_{j \notin \Ccal_{i'}} \frac{2\lambda_{i',j}}{ \left[ \bbf(\x_j)^\top (\bbeta_{i'} - \bbeta_{i^*(\x_j)}) \right]^2 } \sum_{h=1}^{p} \left( \bbf(\x_{j})^\top \left( \F^\top \F \right)^{-1} \bbf(\x_{h}^\circ) \right)^2  \frac{\sigma^2_{i'}(\x_{h}^\circ)}{\alpha_{i',h}} \Bigg)  \\
	=& \sum_{i'=1}^k \sum_{h=1}^{p} \frac{\sigma^2_{i'}(\x_{h}^\circ)}{\alpha_{i',h}} \Bigg( \sum_{j \in \Ccal_{i'}} \sum_{i=1,i\ne i'}^{k} \frac{2\lambda_{i,j}}{ \left[ \bbf(\x_j)^\top (\bbeta_i - \bbeta_{i'}) \right]^2 }  \left( \bbf(\x_{j})^\top \left( \F^\top \F \right)^{-1} \bbf(\x_{h}^\circ) \right)^2    \\
	&\phantom{\sum \sum \frac{1}{\alpha_{i',j'}}\Bigg(} + \sum_{j \notin \Ccal_{i'}} \frac{2\lambda_{i',j}}{ \left[ \bbf(\x_j)^\top (\bbeta_{i'} - \bbeta_{i^*(\x_j)}) \right]^2 } \left( \bbf(\x_{j})^\top \left( \F^\top \F \right)^{-1} \bbf(\x_{h}^\circ) \right)^2   \Bigg).
\end{align*}
Then the dual function is
\begin{align*}
	g(\blam,\theta) =& \inf_{\alphab,z} L(\alphab,z,\blam,\theta)  \\
	=& \inf_{\alphab,z} \left( 1-\sum_{j=1}^m \sum_{i=1, i\ne i^*(\x_j)}^k\lambda_{i,j} \right) z + \theta \left( \sum_{i=1}^k \sum_{h=1}^p \alpha_{i,h}-1 \right) \\
	&+ \sum_{j=1}^m \sum_{i=1, i\ne i^*(\x_j)}^k \lambda_{i,j}  \frac{2 \bbf(\x_j)^\top \left( \F^\top \F \right)^{-1}  \F^\top \left(\bSigmap_{\epsilon,i}+\bSigmap_{\epsilon,i^*(\x_j)}\right) \F \left( \F^\top \F \right)^{-1} \bbf(\x_j) }{\left[ \bbf(\x_j)^\top (\bbeta_i - \bbeta_{i^*(\x_j)}) \right]^2 }  \\
	=& \inf_{\alphab,z} \left( 1-\sum_{j=1}^m \sum_{i=1, i\ne i^*(\x_j)}^k\lambda_{i,j} \right) z + \sum_{i=1}^k \sum_{h=1}^p \left( \frac{\chi_{i,h}(\blam, \boldsymbol{y}, \boldsymbol{\sigma}^2)}{\alpha_{i,h}} + \theta \alpha_{i,h} \right) - \theta.
\end{align*}
When $\sum_{j=1}^m \sum_{i=1, i\ne i^*(\x_j)}^k\lambda_{i,j} \ne 1 $ or $\theta \le 0$, $g(\blam,\theta) = -\infty$. And when $\sum_{j=1}^m \sum_{i=1, i\ne i^*(\x_j)}^k\lambda_{i,j} = 1 $ and $\theta >0$, $\alpha_{i,h} = \sqrt{\frac{\chi_{i,h}(\blam, \boldsymbol{y}, \boldsymbol{\sigma}^2)}{\theta}}$ corresponds to the minimum point. Therefore,
\begin{align*}
	g(\blam,\theta)= \begin{cases}
		\sqrt{\theta}\sum_{i=1}^k \sum_{h=1}^p 2\sqrt{\chi_{i,h}(\blam, \boldsymbol{y}, \boldsymbol{\sigma}^2)} - \theta, & \sum_{j=1}^m \sum_{i=1, i\ne i^*(\x_j)}^k\lambda_{i,j}=1 \text{ and } \theta > 0, \\
		-\infty, & \text{otherwise.}
	\end{cases}
\end{align*}
The Lagrange dual problem of \eqref{eq:reprob2} is
\begin{equation*}
	\begin{aligned}
		\max_{\lambda,\theta} ~ & \sqrt{\theta}\sum_{i=1}^k \sum_{h=1}^p 2\sqrt{\chi_{i,h}(\blam, \boldsymbol{y}, \boldsymbol{\sigma}^2)} - \theta  \\
		s.t. ~ & \sum_{j=1}^m \sum_{i=1, \ i\ne i^*(\x_j)}^k\lambda_{i,j}=1, \lambda_{i,j} \ge 0, ~i=1,\dots,k \text{ and } i \ne i^*(\x_j), \ j = 1,\dots,m; \theta > 0.
	\end{aligned}
\end{equation*}
which can be simplified by taking $\theta = \left( \sum_{i=1}^k \sum_{h=1}^p \sqrt{\chi_{i,h}(\blam, \boldsymbol{y}, \boldsymbol{\sigma}^2)} \right)^2$ as
\begin{equation*}
	\begin{aligned}
		\max_{\lambda} ~ & \left( \sum_{i=1}^k \sum_{h=1}^p \sqrt{\chi_{i,h}(\blam, \boldsymbol{y}, \boldsymbol{\sigma}^2)} \right)^2  \\
		s.t. ~ & \sum_{j=1}^m \sum_{i=1, i\ne i^*(\x_j)}^k\lambda_{i,j}=1, \ \lambda_{i,j} \ge 0, ~i=1,\dots,k \text{ and } i \ne i^*(\x_j), \ j = 1,\dots,m.
	\end{aligned}
\end{equation*}
Since \eqref{eq:reprob2} is a convex program and satisfies the Slater's condition  \citep{boyd2004}, strong duality holds so that the optimal solution $(\alphab,z)$ to \eqref{eq:reprob2} (and also problem (8) of the main paper) is also the optimal point of   $\min_{\alphab',z} L(\alphab',z,\blam,\theta)$, where $(\blam,\theta)$ is the optimal for Lagrange dual problem (10) of the main paper.

Note that $\alpha_{i,h} = \sqrt{\frac{\chi_{i,h}(\blam, \boldsymbol{y}, \boldsymbol{\sigma}^2)}{\theta}}$ and $\theta = \left( \sum_{i=1}^k \sum_{h=1}^p \sqrt{\chi_{i,h}(\blam, \boldsymbol{y}, \boldsymbol{\sigma}^2)} \right)^2$ at the optimal solution to problem  (10). Thus, the optimal solution to \eqref{eq:reprob2} (and also problem (8) of the main paper) satisfies $\alpha_{i,h} = \sqrt{\frac{\chi_{i,h}(\blam, \boldsymbol{y}, \boldsymbol{\sigma}^2)}{\theta}} = \frac{\sqrt{\chi_{i,h}(\blam, \boldsymbol{y}, \boldsymbol{\sigma}^2)}}{\sum_{i=1}^k \sum_{h=1}^p \sqrt{\chi_{i,h}(\blam, \boldsymbol{y}, \boldsymbol{\sigma}^2)}}$ where $\blam$ is the solution to Lagrange dual problem (10) of the main paper.	

\subsection{Proof of Theorem 5}

We adapt the proofs of Theorem 5 in \cite{zhou2021} and Proposition 6.1 in \cite{lin2009decomposition} to show Theorem 5 here. Note that \cite{zhou2021} considers the classic R\&S instead of CR\&S and \cite{lin2009decomposition} solves a deterministic convex program. Thus, our proof is not a simple application of the techniques in the references. The proof of Theorem 5 requires Lemmas \ref{lem:ref} and \ref{lem:lincons}. Lemma \ref{lem:ref} provides more insights into CR\&S Algorithm 2 and Lemma \ref{lem:lincons} shows the consistency of it. Theorem 5 is then proved by showing that $a(\hat{\blam}^{(r)},\boldsymbol{y}, \boldsymbol{\sigma}^2)$ decreases with $r$ when $r$ is large enough.

For any $\blam$, define a feasible direction of $\blam$ as $\dfrak \in \mathbb{R}^{(k-1)m}$. Note that each element of $\dfrak$ corresponds to an element of $\blam$. The $\dfrak$ should satisfy $\mathbf{1}^\top \dfrak = 0$ and the element of $\dfrak$ that corresponds to $\lambda_{i,j}$ is strictly positive if $\lambda_{i,j}=0$. Let $\Dcal(\blam)$ denote the set of all the feasible directions. To show that $\blam$ is a stationary point, we should prove that $\nabla a(\blam,\boldsymbol{y}, \boldsymbol{\sigma}^2)^\top \dfrak \ge 0$ for any $\dfrak \in \Dcal(\blam)$. However,  the next lemma showed that we just need to prove $\nabla a(\blam,\boldsymbol{y}, \boldsymbol{\sigma}^2)^\top \dfrak \ge 0$ for a small subset of $\Dcal(\blam)$.
\begin{lemma}[\cite{zhou2021,lin2009decomposition}]\label{lem:ref}
	For any $\blam$, suppose $\lambda_{i,j} > 0$ and define $\Dcal^{(i,j)}(\blam) =  \{e_{i',j'} - e_{i,j}: i' \ne i \text{ or } j' \ne j\} \bigcup \{e_{i,j}-e_{i',j'}: i' \ne i \text{ or } j' \ne j, \lambda_{i',j'} > 0\}$. Then $\blam$ is a stationary point of problem (10) of the main paper if and only if $\nabla a(\blam,\boldsymbol{y}, \boldsymbol{\sigma}^2)$ is well-defined and $\nabla a(\blam,\boldsymbol{y}, \boldsymbol{\sigma}^2)^\top \dfrak \ge 0$, $\forall \dfrak \in \Dcal^{(i,j)}(\blam)$.
\end{lemma}

Based on Lemma \ref{lem:ref}, instead of considering all feasible directions for the gradient descent, we can simply find one of the $(i,j)$'s such that $\lambda_{i,j} > 0$ and consider the feasible directions in $\Dcal^{(i,j)}(\blam)$ only. In CR\&S Algorithm 2, Step 2 finds the $(i^{r*},j^{r*})$ satisfying $\hat{\lambda}_{i^{r*},j^{r*}} > \eta$ instead of $\hat{\lambda}_{i^{r*},j^{r*}} > 0$ to prevent $\hat{\lambda}_{i^{r*},j^{r*}}$ from being too small. Thus, $\eta$ is used for numerical stability.
Given the chosen $\hat{\lambda}_{i^{r*},j^{r*}}$,  Step 3 of CR\&S Algorithm 2 selects a descent direction from $\Dcal^{(i^{r*},j^{r*})}(\hat{\blam})$. As a result, the rationale of Steps 2 and 3 follows from Lemma \ref{lem:ref}.

To prove Theorem 5, we first show the consistency of CR\&S Algorithm 2. We need to show that $\widehat{\bbeta}_i \to \bbeta_i$ with probability one for all $i=1,\dots,k$. Fix a sample path $\omega$. On the sample path $\omega$, let $A$ denote the set of context-treatment pairs $(i,\x_h^\circ)$ such that $n_{i,h} \to \infty$ in CR\&S Algorithm 2.

By the strong law of large numbers,  $\bar{Y}_{i}(\x_h^\circ) \to y_i(\x_h^\circ)$ and $\hat{\sigma}_{i}^2(\x_h^\circ) \to \sigma_{i}^2(\x_h^\circ)$ as $r \to \infty$ if $(i,\x_h^\circ) \in A$. Meanwhile, there exists a large enough $r_0$ such that $(i,\x_h^\circ)$ does not receive any samples for all iteration $r > r_0$ and $(i,\x_h^\circ) \notin A$.

Let $\bar{y}_i(\x_h^\circ)$ and $\bar{\sigma}_{i}^2(\x_h^\circ)$ denote the limiting value of $\bar{Y}_{i}(\x_h^\circ) $ and $\hat{\sigma}_{i}^2(\x_h^\circ)$, $i=1,\dots,k$, $h=1,\dots,p$. That is, $\bar{Y}_{i}(\x_h^\circ) \to \bar{y}_i(\x_h^\circ)$ and $\hat{\sigma}_{i}^2(\x_h^\circ) \to \bar{\sigma}_{i}^2(\x_h^\circ)$ as $r \to \infty$. (Obviously, $\bar{y}_i(\x_h^\circ) = y_i(\x_h^\circ)$ and $\bar{\sigma}_{i}^2(\x_h^\circ) = \sigma_{i}^2(\x_h^\circ)$ if $(i,\x_h^\circ) \in A$.) Let $\bar{\sigma}^2_{\max}$ and $\bar{\sigma}^2_{\min}$ denote the upper and lower bounds of $\bar{\sigma}_{i}^2(\x_h^\circ)$ for all $i$ and $h$.

For simplicity, we define some notations. Let $\bar{\boldsymbol{y}}_i = \left(\bar{y}_i(\x_1^\circ),\ldots, \bar{y}_i(\x_p^\circ)\right)^\top$, $\bar{\boldsymbol{\sigma}}^2_i = \left( \bar{\sigma}_{i}^2(\x_1^\circ), \dots, \bar{\sigma}_{i}^2(\x_p^\circ) \right)^\top$, $\bar{\boldsymbol{y}} = \left(\bar{\boldsymbol{y}}_1^\top,\ldots,\bar{\boldsymbol{y}}_k^\top \right)^\top$, and $\bar{\boldsymbol{\sigma}}^2 = \left( \bar{\boldsymbol{\sigma}}^{2,\top}_1, \dots, \bar{\boldsymbol{\sigma}}^{2,\top}_k \right)^\top$. Note that $\widehat{\bbeta}_i= \left( \F^\top \F \right)^{-1} \F^\top \overline{\Y}_i$ converges to $\bar{\bbeta}_i = \left( \F^\top \F \right)^{-1} \F^\top \bar{\boldsymbol{y}}_i$ as $r \to \infty$. Let $\bar{i}^*(\x_j) = \arg \min_{i=1,\dots,k} \bbf(\x_j)^\top \bar{\bbeta}_i $ denote the estimated best treatment as $r \to \infty$, $j=1,\dots,m$. Let  $\bar{\Ccal}_i = \{j: \bar{i}^*(\x_j) = i\}$ and $\hat{\Ccal}_i = \{j: \hat{i}^*(\x_j) = i\}$.

Define $\psi_{i_1,j_1,i_2,h}(\boldsymbol{y}) = \frac{2 \left( \bbf(\x_{j_1})^\top \left( \F_\circ^\top \F_\circ \right)^{-1} \bbf(\x_{h}^\circ) \right)^2}{ \left[ \bbf(\x_{j_1})^\top \left(\bbeta_{i_1} - \bbeta_{i_2} \right) \right]^2 } $ if $i_2=i^*(\x_{j_1})$ and $\psi_{i_1,j_1,i_2,h}(\boldsymbol{y})$ $= 0$ if $i_2 \ne  i^*(\x_{j_1})$. Then, by (11) of the main paper, $\chi_{i,h}(\blam, \boldsymbol{y}, \boldsymbol{\sigma}^2)$ can be rewritten as
\begin{align*}
	\chi_{i,h}(\blam, \boldsymbol{y}, \boldsymbol{\sigma}^2) =& \sigma^2_{i}(\x_{h}^\circ) \left(  \sum\limits_{j' \in \Ccal_i} \sum\limits_{i'=1,i'\ne i}^{k} \lambda_{i',j'}  \psi_{i',j',i,h}(\boldsymbol{y}) + \sum\limits_{j' \notin \Ccal_i} \lambda_{i,j'} \psi_{i,j',i^*(\x_{j'}),h}(\boldsymbol{y}) \right).
\end{align*}
Similarly, we define $\psi_{i_1,j_1,i_2,h}(\bar{\boldsymbol{y}}) = \frac{2 \left( \bbf(\x_{j_1})^\top \left( \F^\top \F \right)^{-1} \bbf(\x_{h}^\circ) \right)^2}{ \left[ \bbf(\x_{j_1})^\top \left( \bar{\bbeta}_{i_1} - \bar{\bbeta}_{i_2} \right) \right]^2 } = \frac{2 \left( \bbf(\x_{j_1})^\top \left( \F^\top \F \right)^{-1} \bbf(\x_{h}^\circ) \right)^2}{ \left[ \bbf(\x_{j_1})^\top \left( \F^\top \F \right)^{-1} \F^\top \left( \bar{\boldsymbol{y}}_{i_1} - \bar{\boldsymbol{y}}_{i_2} \right) \right]^2 }$ if $i_2=\bar{i}^*(\x_{j_1})$ and $\psi_{i_1,j_1,i_2,h}(\bar{\boldsymbol{y}})= 0$ if $i_2 \ne  \bar{i}^*(\x_{j_1})$. Define $\psi_{i_1,j_1,i_2,h}(\overline{\Y})$ similarly. Let $\psi_{\text{max}}$ denote the upper bound of $\psi_{i_1,j_1,i_2,h}(\bar{\boldsymbol{y}})$ for all $i_1$, $j_1$, $i_2$ and $h$.

Let $\Ccal_{i,h}  = \{(i_1,j_1):\text{ either } i_1 \text{ or } i^*(\x_{j_1}) \text{ is } i \text{ and } \psi_{i_1,j_1,i^*(\x_{j_1}),h} (\boldsymbol{y}) \ne 0\}$ which is the set of indices $(i_1,j_1)$ such that $\psi_{i_1,j_1, i^*(\x_{j_1}),h}(\boldsymbol{y}) \ne 0$ and is used for calculating $\chi_{i,h}(\blam, \boldsymbol{y}, \boldsymbol{\sigma}^2)$. Let $\Ccal = \cup_{i=1}^k \cup_{h=1}^p \Ccal_{i,h}$.
Let $\bar{\Ccal}_{i,h}$ be the estimate of $\Ccal_{i,h}$ when plugging in $\bar{\boldsymbol{y}}$  and let $\hat{\Ccal}_{i,h}$ be the estimate of $\Ccal_{i,h}$ when plugging in $\overline{\Y}$. Let $\bar{\Ccal} = \cup_{i=1}^k \cup_{h=1}^p \bar{\Ccal}_{i,h}$ and $\hat{\Ccal} = \cup_{i=1}^k \cup_{h=1}^p \hat{\Ccal}_{i,h}$. $\bar{\Ccal}$ indicates the index of the strictly positive $\psi_{i_1,j_1,i_2,h}(\bar{\boldsymbol{y}})$'s  (or equivalently, $\psi_{i_1,j_1,\bar{i}^*(\x_{j_1}),h}(\bar{\boldsymbol{y}})$'s, by noting that $\psi_{i_1,j_1,i_2,h}(\bar{\boldsymbol{y}})=0$ if $i_2 \ne \bar{i}^*(\x_{j_1})$). Let $\psi_{\text{min}}$  denote the minimum of  $\psi_{i_1,j_1,\bar{i}^*(\x_{j_1}),h}(\bar{\boldsymbol{y}})$ where $(i_1,j_1) \in \bar{\Ccal}$.
For  $(i_1,j_1) \in \bar{\Ccal}$,
we can show by continuity that $\left| \psi_{i_1,j_1,\bar{i}^*(\x_{j_1}),h}(\overline{\Y}) - \psi_{i_1,j_1,\bar{i}^*(\x_{j_1}),h}(\bar{\boldsymbol{y}}) \right| \le C_{a0} \lVert \overline{\Y} - \bar{\boldsymbol{y}} \rVert $.

%Furthermore, by continuity again, for any $\blam \in \Dcal$,
%\begin{align}
%	&\left| \sqrt{\chi_{i,h}(\blam, \overline{\Y}, \hat{\boldsymbol{\sigma}}^2)} - \sqrt{\chi_{i,h}(\blam, \bar{\boldsymbol{y}}, \bar{\boldsymbol{\sigma}}^2)} \right|
%	\le C_{\chi}  \lVert \overline{\Y} - \bar{\boldsymbol{y}} \rVert.  \label{ineq:chiconv}
%\end{align}
%
%
%Since $a(\blam; \overline{\Y}, \hat{\boldsymbol{\sigma}}^2) =  -\sum_{i=1}^k \sum_{h=1}^p \sqrt{ \chi_{i,h}(\blam, \overline{\Y}, \hat{\boldsymbol{\sigma}}^2)  }$ and $\sqrt{ \chi_{i,h}(\blam, \overline{\Y}, \hat{\boldsymbol{\sigma}}^2)  }$ converges uniformly to $\sqrt{ \chi_{i,h}(\blam, \bar{y}, \bar{\boldsymbol{\sigma}}^2)  }$ as $r \to \infty$ by \eqref{ineq:chiconv}, we can bound  $a(\blam; \overline{\Y}, \hat{\boldsymbol{\sigma}}^2)$ uniformly by
%\begin{align}\label{ineq:abd}
%	C_{a1} \triangleq -kp \sqrt{a_{\text{max}} + \varepsilon_0} \le a(\blam; \overline{\Y}, \hat{\boldsymbol{\sigma}}^2) \le -kp \sqrt{a_{\text{min}} - \varepsilon_0} \triangleq C_{a2}, \ \forall \blam \in \Dcal.
%\end{align}

Define the derivative of $\sqrt{ \chi_{i,h}(\blam, \boldsymbol{y}, \boldsymbol{\sigma}^2)  }$ with respect to $\lambda_{i_1,j_1}$ as $\varsigma_{i_1,j_1,i,h}(\blam,\boldsymbol{y}, \boldsymbol{\sigma}^2) = \frac{\partial \sqrt{ \chi_{i,h}(\blam, \boldsymbol{y}, \boldsymbol{\sigma}^2)  }}{ \partial  \lambda_{i_1,j_1}}$. Then,  $\varsigma_{i_1,j_1,i,h}(\blam,\boldsymbol{y}, \boldsymbol{\sigma}^2) = \sigma_{i}^2(\x_h^\circ) \frac{\psi_{i_1,j_1,i^*(\x_{j_1}),h}(\boldsymbol{y})}{2\sqrt{ \chi_{i,h}(\blam, \boldsymbol{y}, \boldsymbol{\sigma}^2)  }}$ if $(i_1,j_1) \in \Ccal_{i,h}$;  $\varsigma_{i_1,j_1,i,h}(\blam,$ $\boldsymbol{y}, \boldsymbol{\sigma}^2) =0$ if $(i_1,j_1) \notin \Ccal_{i,h}$.
Let $\Ccal_{(i,j)} = \{(i',h'): \varsigma_{i,j,i',h'}(\blam,\boldsymbol{y}, \boldsymbol{\sigma}^2) \ne 0\}$. $\bar{\Ccal}_{(i,j)}$ and $\hat{\Ccal}_{(i,j)}$ are obtained when $(\bar{\boldsymbol{y}}, \bar{\boldsymbol{\sigma}}^2)$ and $(\overline{\Y}, \hat{\boldsymbol{\sigma}}^2)$ are plugged in respectively. Note that $\Ccal_{(i,j)}$ is defined for the  treatment-context pair $(i,\x_j)$, while $\Ccal_{i,h}$ is defined for the treatment-context pair $(i,\x_h^\circ)$.

In the following, we always assume that $r>r_0$ is large enough so that for a $\varepsilon_0$ small enough, all estimates will not deviate from their limits by more than $\varepsilon_0$. For example, when $r>r_0$, $\left| \bar{Y}_i(\x_{h}^\circ) - \bar{y}_i(\x_{h}^\circ) \right| \le \varepsilon_0$, $\left| \hat{\sigma}_i(\x_{h}^\circ) - \bar{\sigma}_i(\x_{h}^\circ) \right| \le \varepsilon_0$, and $\left|\psi_{i_1,j_1,\hat{i}^*(\x_{j_1}),h}(\overline{\Y}) - \psi_{i_1,j_1,\bar{i}^*(\x_{j_1}),h}(\bar{\boldsymbol{y}})\right| \le \epsilon_0$. Since  $\hat{i}^*(\x_j) = \bar{i}^*(\x_j)$, $j=1,\dots,m$, when $r>r_0$, we can show that $ \hat{\Ccal}_{i,h} = \bar{\Ccal}_{i,h}$ and $\bar{\Ccal}_{(i,j)}=\hat{\Ccal}_{(i,j)}$ for $r>r_0$. Again, we append subscript/superscript $r$ to notations in CR\&S Algorithm 2 to indicate the iteration number in the subsequent proof.

%Since $\frac{\partial a(\blam; \overline{\Y}, \hat{\boldsymbol{\sigma}}^2)}{\partial \lambda_{i_1,j_1} } =  -\sum_{i=1}^k \sum_{h=1}^p \varsigma_{i_1,j_1,i,h}(\blam, \overline{\Y}, \hat{\boldsymbol{\sigma}}^2)  = -\sum_{i,h:(i,h) \in \hat{\Ccal}_{(i_1,j_1)}}  \varsigma_{i_1,j_1,i,h}(\blam, \overline{\Y}, \hat{\boldsymbol{\sigma}}^2)$, we have
%\begin{align}\label{eq:adiff}
%	\left| \frac{\partial a(\blam; \overline{\Y}, \hat{\boldsymbol{\sigma}}^2)}{\partial \lambda_{i_1,j_1} } \Big|_{\blam = \hat{\blam}} - \frac{\partial a(\blam; \bar{\boldsymbol{y}}, \bar{\boldsymbol{\sigma}}^2)}{\partial \lambda_{i_1,j_1} } \Big|_{\blam = \hat{\blam}} \right| \le kpC_{d} \left( \lVert \overline{\Y} - \bar{\boldsymbol{y}} \rVert + \lVert  \hat{\boldsymbol{\sigma}}^2 - \bar{\boldsymbol{\sigma}}^2 \rVert \right).
%\end{align}

\begin{lemma}\label{lem:lincons}
	Under CR\&S Algorithm 2, $\hat{\alpha}_{i,h} = \Theta(1)$ almost surely for all $i=1,\dots,k$, $h=1,\dots,p$.
\end{lemma}

\begin{proof}
	It is sufficient to show
	\begin{align}\label{ineq:lowbdalgo2}
		\mathop{\lim \inf}_{r \to \infty} \chi_{i,h}(\hat{\blam}^{(r)}, \overline{\Y}^{(r)}, \hat{\boldsymbol{\sigma}}^{2,(r)})
		=   \mathop{\lim \inf}_{r \to \infty} \hat{\sigma}_{i,r}^{2}(\x_h^\circ) \sum\limits_{(i_1,j_1) \in \bar{\Ccal}_{i,h}}  \hat{\lambda}_{(i_1,j_1),r}  \psi_{i_1,j_1,\bar{i}^*(\x_{j_1}),h}(\overline{\Y}^{(r)})  > 0.
	\end{align}
	The following inequalities will be used in this proof. First, a general upper bound is
	{\small \begin{align}
			&\chi_{i,h}(\hat{\blam}^{(r)}, \overline{\Y}^{(r)}, \hat{\boldsymbol{\sigma}}^{2,(r)})  = \hat{\sigma}_{i,r}^{2}(\x_{h}^\circ)\sum\limits_{(i_1,j_1) \in \bar{\Ccal}_{i,h}}  \hat{\lambda}_{(i_1,j_1),r}  \psi_{i_1,j_1,\bar{i}^*(\x_{j_1}),h}(\overline{\Y}^{(r)})  \nonumber \\
			\le&    k(m-1) (\psi_{\text{max}} + \varepsilon_0) (\bar{\sigma}^2_{\max} + \varepsilon_0)  \max_{(i_1,j_1) \in \bar{\Ccal}_{i,h}} \hat{\lambda}_{(i_1,j_1),r}
			\le  k(m-1) (\psi_{\text{max}} + \varepsilon_0) (\bar{\sigma}^2_{\max} + \varepsilon_0) \triangleq C_1 . \label{ineq:cons1}
	\end{align}}
	The second inequality is, for any $(i_1,j_1) \in \bar{\Ccal}_{i_0,h_0}$, $i_0=1,\dots,k$, $h_0=1,\dots,p$,
	{\small \begin{align}
			&\left[ \nabla a(\hat{\blam}^{(r)},\overline{\Y}^{(r)}, \hat{\boldsymbol{\sigma}}^{2,(r)}) \right]_{(i_1,j_1)}
			= -\sum_{i=1}^k \sum_{h=1}^p \varsigma_{i_1,j_1,i,h} (\hat{\blam}^{(r)},\overline{\Y}^{(r)}, \hat{\boldsymbol{\sigma}}^{2,(r)})
			\le - \frac{ \hat{\sigma}_{i,r}^{2}(\x_{h}^\circ) \psi_{i_1, j_1,\bar{i}^*(\x_{j_1}),h_0}(\overline{\Y}^{(r)})}{2\sqrt{ \chi_{i_0,h_0}(\hat{\blam}^{(r)},\overline{\Y}^{(r)}, \hat{\boldsymbol{\sigma}}^{2,(r)})  }}  \nonumber \\
			\le& - \frac{ (\bar{\sigma}^2_{\text{min} } - \varepsilon_0) (\psi_{\text{min}}-\varepsilon_0)}{2 \sqrt{k(m-1) (\psi_{\text{max}} + \varepsilon_0) (\bar{\sigma}^2_{\max} + \varepsilon_0)  \max_{(i,j) \in \bar{\Ccal}_{i_0,h_0}} \hat{\lambda}_{(i,j),r}}}
			\triangleq -\frac{C_2}{\sqrt{\max_{(i,j) \in \bar{\Ccal}_{i_0,h_0}} \hat{\lambda}_{(i,j),r}}}. \label{ineq:cons11}
	\end{align}}
	Third,  for any $(i_1,j_1)$, $i_1=1,\dots,k$, $j_1=1,\dots,m$,
	{\small \begin{align}
			0\ge&\left[ \nabla a(\hat{\blam}^{(r)},\overline{\Y}^{(r)}, \hat{\boldsymbol{\sigma}}^{2,(r)}) \right]_{(i_1, j_1)}
			= -\sum_{i=1}^k \sum_{h=1}^p \varsigma_{i_1, j_1, i, h } (\hat{\blam}^{(r)},\overline{\Y}^{(r)}, \hat{\boldsymbol{\sigma}}^{2,(r)}) = -\sum_{(i,h) \in \bar{\Ccal}_{(i_1, j_1)}} \varsigma_{i_1, j_1, i, h } (\hat{\blam}^{(r)},\overline{\Y}^{(r)}, \hat{\boldsymbol{\sigma}}^{2,(r)})  \nonumber \\
			\ge& -\sum_{(i,h) \in \bar{\Ccal}_{(i_1, j_1)}}  \frac{ (\bar{\sigma}^2_{\text{max}} + \varepsilon_0) \psi_{i_1, j_1,\bar{i}^*(\x_{j_1}),h}(\overline{\Y}^{(r)})}{2\sqrt{ (\bar{\sigma}^2_{\text{min}} - \varepsilon_0) \hat{\lambda}_{(i_1, j_1),r} \psi_{i_1, j_1,\bar{i}^*(\x_{j_1}),h}(\overline{\Y}^{(r)}) }}
			= -\frac{\bar{\sigma}^2_{\text{max}} + \varepsilon_0}{2} \sum_{(i,h) \in \bar{\Ccal}_{(i_1, j_1)}} \sqrt{ \frac{\psi_{i_1, j_1,\bar{i}^*(\x_{j_1}),h}(\overline{\Y}^{(r)})}{ (\bar{\sigma}^2_{\text{min}} - \varepsilon_0) \hat{\lambda}_{(i_1, j_1),r} } } \nonumber \\
			\ge& -\frac{kp (\bar{\sigma}^2_{\text{max}} + \varepsilon_0) }{2} \sqrt{\frac{\psi_{\text{max}}}{\bar{\sigma}^2_{\text{min}} - \varepsilon_0}} \frac{1}{\sqrt{\hat{\lambda}_{(i_1, j_1),r}}} \triangleq  -\frac{C_3}{\sqrt{\hat{\lambda}_{(i_1, j_1),r}}}. \label{ineq:cons2}
	\end{align}}
	
	Suppose there exists a pair $(i_0,h_0)$ with $\mathop{\lim \inf}_{r \to \infty} \chi_{i_0,h_0}(\hat{\blam}^{(r)}, \overline{\Y}^{(r)}, \hat{\boldsymbol{\sigma}}^{2,(r)}) = 0$. Since $\chi_{i_0,h_0}(\hat{\blam}^{(r)},$ $ \overline{\Y}^{(r)}, \hat{\boldsymbol{\sigma}}^{2,(r)}) = \hat{\sigma}_{i_0,r}^{2}(\x_{h_0}^\circ) \sum\limits_{(i_1,j_1) \in \bar{\Ccal}_{i_0,h_0}}  \hat{\lambda}_{(i_1,j_1),r}  \psi_{i_1,j_1,\bar{i}^*(\x_{j_1}),h_0}(\overline{\Y}^{(r)})$ and  $\psi_{i_1,j_1,\bar{i}^*(\x_{j_1}),h_0}(\overline{\Y}^{(r)})$ is lower bounded by $\psi_{\text{min}} -\varepsilon_0 > 0$, we have $ \mathop{\lim \inf}_{r \to \infty} \sup_{(i_1,j_1) \in \bar{\Ccal}_{i_0,h_0}} \hat{\lambda}_{(i_1,j_1),r} \to 0$.
	
	Let $C_0 = \min \{ 2/3,(C_2/C_3)^2  \} \eta$ and
	{\small
		\begin{align}\label{ineq:defeps}
			\varepsilon \le  \min\left\{ \frac{C_0}{3}, C_2^2\left( \frac{C_3(1+\sfrak_2)}{\sqrt{\eta}} + \sfrak_2 \frac{C_3}{\sqrt{C_0}}  \right)^{-2}, C_2^2 \left( C_3 \sqrt{\frac{3}{2  \eta}} + \frac{3\sfrak_2kp\sqrt{C_1}}{2\sfrak_1 \eta} \right)^{-2}, \frac{1}{k(m-1)} \right\}.
	\end{align}}
	Since $ \mathop{\lim \inf}_{r \to \infty} \sup_{(i_1,j_1) \in \bar{\Ccal}_{i_0,h_0}} \hat{\lambda}_{(i_1,j_1),r} \to 0$ by assumption, there exists $r_t$ such that $\hat{\lambda}_{(i_1,j_1),r_t} < \varepsilon$ for all $(i_1,j_1) \in \bar{\Ccal}_{i_0,h_0}$ while $\hat{\lambda}_{(i_0',j_0'),r_t-1} \ge \varepsilon$ for some $(i_0',j_0') \in \bar{\Ccal}_{i_0,h_0}$ because the initial value $\hat{\lambda}_{(i_0',j_0'),0} = \frac{1}{k(m-1)} \ge \varepsilon$.  Since  $\hat{\lambda}_{(i_1,j_1),r_t} < \varepsilon$ for all $(i_1,j_1) \in \bar{\Ccal}_{i_0,h_0}$, the value of $\hat{\lambda}_{(i_1,j_1),r_t-1}$ must fall in one of the following three cases.
	\begin{enumerate}
		\item[(1)] There exists $(i'_0,j'_0) \in \bar{\Ccal}_{i_0,h_0}$ such that $ \varepsilon \le \hat{\lambda}_{(i'_0,j'_0),r_{t}-1} \le C_0$, while $\hat{\lambda}_{(i_1,j_1),r_{t}-1} \le \varepsilon$ for $(i_1,j_1) \ne (i'_0,j'_0)$ and $(i_1,j_1) \in \bar{\Ccal}_{i_0,h_0}$. %For all $(i_1,j_1) \in \bar{\Ccal}_{i_0,h_0}$, $\hat{\lambda}_{i_1,j_1}^{(r_{t}-1)} \le C_0$.
		\item[(2)] There exists $(i'_0,j'_0)\in \bar{\Ccal}_{i_0,h_0}$ such that $C_0 < \hat{\lambda}_{(i'_0,j'_0),r_{t}-1} < \eta$, while $\hat{\lambda}_{(i_1,j_1),r_{t}-1} \le \varepsilon$ for $(i_1,j_1) \ne (i'_0,j'_0)$ and $(i_1,j_1) \in \bar{\Ccal}_{i_0,h_0}$.
		\item[(3)] There exists $(i'_0,j'_0)\in \bar{\Ccal}_{i_0,h_0}$ such that $ \hat{\lambda}_{(i'_0,j'_0),r_{t}-1} \ge \eta$, while $\hat{\lambda}_{(i_1,j_1),r_{t}-1} \le \varepsilon$ for $(i_1,j_1) \ne (i'_0,j'_0)$ and $(i_1,j_1) \in \bar{\Ccal}_{i_0,h_0}$.
	\end{enumerate}
	
	In case (1), since $\hat{\lambda}_{(i_1,j_1),r_{t}-1}  \le C_0 \le \eta$ for all $(i_1,j_1)\in \bar{\Ccal}_{i_0,h_0}$, the selected pair $(i^{r_{t}-1*}, j^{r_{t}-1*})$ in Step 2 of CR\&S Algorithm 2 should satisfy   $(i^{r_{t}-1*}, j^{r_{t}-1*}) \notin \bar{\Ccal}_{i_0,h_0}$ because $\hat{\lambda}_{(i^{r_{t}-1*}, j^{r_{t}-1*}), r_t-1} \ge \eta$. By \eqref{ineq:cons2}, we have
	$
	\left[ \nabla a(\hat{\blam}^{(r_{t}-1)},\overline{\Y}^{(r_{t}-1)}, \hat{\boldsymbol{\sigma}}^{2,(r_{t}-1)}) \right]_{(i^{r_{t}-1*}, j^{r_{t}-1*})}
	\ge -\frac{C_3}{\sqrt{\hat{\lambda}_{(i^{r_{t}-1*}, j^{r_{t}-1*}),r_t-1}}} \ge -\frac{C_3}{\sqrt{\eta}}.
	$
	Meanwhile, for $(i_0',j_0') \in \bar{\Ccal}_{i_0,h_0}$, by \eqref{ineq:cons11},
	\begin{align*}
		\left[ \nabla a(\hat{\blam}^{(r_{t}-1)},\overline{\Y}^{(r_{t}-1)}, \hat{\boldsymbol{\sigma}}^{2,(r_{t}-1)}) \right]_{(i_0',j_0')}  \le& - \frac{C_2}{\sqrt{\max_{(i,j) \in \bar{\Ccal}_{i_0,h_0}} \hat{\lambda}_{(i_1,j_1),r_t-1}}}
		\le -\frac{C_2}{\sqrt{C_0}} \\
		\le& - \frac{C_3}{\sqrt{\eta}}  \le \left[ \nabla a(\hat{\blam}^{(r_t-1)},\overline{\Y}^{(r_t-1)}, \hat{\boldsymbol{\sigma}}^{2,(r_t-1)}) \right]_{(i^{r_{t}-1*}, j^{r_{t}-1*})}.
	\end{align*}
	Then, the direction $\dfrak^{(r)}$ chosen in Step 3 cannot be $e_{i^{r_{t}-1*}, j^{r_{t}-1*}} - e_{i'_0,j'_0}$ because $\nabla a(\hat{\blam}^{(r_t-1)},\overline{\Y}^{(r_t-1)},$ $ \hat{\boldsymbol{\sigma}}^{2,(r_t-1)})^\top (e_{i^{r_{t}-1*}, j^{r_{t}-1*}} - e_{i'_0,j'_0}) \ge 0$. Then $\hat{\lambda}_{(i'_0,j'_0),r_{t}} \ge \hat{\lambda}_{(i'_0,j'_0),r_{t}-1} > \varepsilon$, contradicting $r_t$'s definition.
	
	In case (2), since $\hat{\lambda}_{(i_1,j_1),r_{t}-1} \le \eta$ for all $(i_1,j_1) \in \bar{\Ccal}_{i_0,h_0}$, the selected pair $(i^{r_{t}-1*}, j^{r_{t}-1*})$ in Step 2 of CR\&S Algorithm 2 should satisfy $(i^{r_{t}-1*}, j^{r_{t}-1*}) \notin \bar{\Ccal}_{i_0,h_0}$  because $\hat{\lambda}_{(i^{r_{t}-1*}, j^{r_{t}-1*}),r_t-1} \ge \eta$. To have $\hat{\lambda}_{(i_0',j_0'),r_{t}} \le \varepsilon$, the direction $\dfrak^{(r+1)}$ chosen in Step 2 should be $e_{i^{r_{t}-1*}, j^{r_{t}-1*}} - e_{i'_0,j'_0}$. Moreover, the chosen stepsize $\sfrak^{(r_t)} $ should satisfy $\sfrak^{(r_t)} \ge C_0 - \varepsilon$ and the condition (14) of the main paper:
	{\small \begin{align}\label{eq:lincond8}
			\nabla a(\hat{\blam}^{(r_{t}-1)}+\sfrak^{(r_t)}\cdot \dfrak^{(r_t)},\overline{\Y}^{(r_{t}-1)}, \hat{\boldsymbol{\sigma}}^{2,(r_{t}-1)})^\top \dfrak^{(r_t)}  \le \sfrak_2 |\nabla a(\hat{\blam}^{(r_{t}-1)},\overline{\Y}^{(r_{t}-1)}, \hat{\boldsymbol{\sigma}}^{2,(r_{t}-1)})^\top \dfrak^{(r_t)}|.
	\end{align}}
	Let $\hat{\blam}^{(r_{t})} = \hat{\blam}^{(r_{t}-1)}+\sfrak^{(r_t)}\cdot \dfrak^{(r_t)}$. Since $\dfrak^{(r_t)}= e_{i^{r_{t}-1*}, j^{r_{t}-1*}} - e_{i'_0,j'_0}$, we have $\hat{\lambda}_{(i^{r_{t}-1*}, j^{r_{t}-1*}),r_t} \ge \hat{\lambda}_{(i^{r_{t}-1*}, j^{r_{t}-1*}),r_{t}-1} \ge \eta$. Again, by \eqref{ineq:cons2}, we have
	$
	\left[ \nabla a(\hat{\blam}^{(r_{t}-1)},\overline{\Y}^{(r_{t}-1)}, \hat{\boldsymbol{\sigma}}^{2,(r_{t}-1)}) \right]_{(i^{r_{t}-1*}, j^{r_{t}-1*})}  \ge -\frac{C_3}{\sqrt{\eta}}$ and  $ \left[ \nabla a(\hat{\blam}^{(r_t)},\overline{\Y}^{(r_{t}-1)}, \hat{\boldsymbol{\sigma}}^{2,(r_{t}-1)}) \right]_{(i^{r_{t}-1*}, j^{r_{t}-1*})}   \ge -\frac{C_3}{\sqrt{\eta}}.
	$
	Similarly, since $\hat{\lambda}_{(i'_0,j'_0),r_{t}-1} > C_0$,
	we have by \eqref{ineq:cons2} that $
	\left[ \nabla a(\hat{\blam}^{(r_{t}-1)},\overline{\Y}^{(r_{t}-1)}, \hat{\boldsymbol{\sigma}}^{2,(r_{t}-1)}) \right]_{(i'_0,j'_0)}   \ge -\frac{C_3}{\sqrt{C_0}}.
	$
	By \eqref{eq:lincond8},
	\begin{align}
		&\left[ \nabla a(\hat{\blam}^{(r_t)},\overline{\Y}^{(r_t-1)}, \hat{\boldsymbol{\sigma}}^{2,(r_t-1)}) \right]_{(i'_0,j'_0)}
		\ge \left[ \nabla a(\hat{\blam}^{(r_t)},\overline{\Y}^{(r_t-1)}, \hat{\boldsymbol{\sigma}}^{2,(r_t-1)}) \right]_{(i^{r_{t}-1*}, j^{r_{t}-1*})}  \nonumber \\
		&- \sfrak_2 \left| \left[ \nabla a(\hat{\blam}^{(r_{t}-1)},\overline{\Y}^{(r_{t}-1)}, \hat{\boldsymbol{\sigma}}^{2,(r_{t}-1)}) \right]_{(i^{r_{t}-1*}, j^{r_{t}-1*})} - \left[ \nabla a(\hat{\blam}^{(r_{t}-1)},\overline{\Y}^{(r_{t}-1)}, \hat{\boldsymbol{\sigma}}^{2,(r_{t}-1)}) \right]_{(i'_0,j'_0)} \right|  \nonumber \\
		\ge& -\frac{C_3}{\sqrt{\eta}} - \sfrak_2 \left( \frac{C_3}{\sqrt{\eta}} + \frac{C_3}{\sqrt{C_0}} \right)
		= - \frac{C_3(1+\sfrak_2)}{\sqrt{\eta}}  - \sfrak_2 \frac{C_3}{\sqrt{C_0}}. \label{eq:case2contra}
	\end{align}
	Since $\hat{\lambda}_{(i_1,j_1),r_t} \le \varepsilon$ for all $(i_1,j_1) \in \bar{\Ccal}_{i_0,h_0}$ by $r_t$'s definition, we have by \eqref{ineq:cons11} that
	\begin{align}
		\left[ \nabla a(\hat{\blam}^{(r_t)},\overline{\Y}^{(r_t-1)}, \hat{\boldsymbol{\sigma}}^{2,(r_t-1)}) \right]_{(i'_0,j'_0)}
		&\le - \frac{C_2}{ \sqrt{\max\limits_{(i_1,j_1) \in \bar{\Ccal}_{i_0,h_0}} \hat{\lambda}_{(i_1,j_1),r_t}} } \le  -\frac{C_2}{ \sqrt{\varepsilon} }.  \label{eq:case2contra2}
	\end{align}
	However, $-\frac{C_2}{ \sqrt{\varepsilon} } < - \frac{C_3(1+\sfrak_2)}{\sqrt{\eta}}  - \sfrak_2 \frac{C_3}{\sqrt{C_0}}$ by the definition of $\varepsilon$. \eqref{eq:case2contra} and \eqref{eq:case2contra2} are contradictory.
	
	In case (3), if the selected pair $(i^{r_{t}-1*}, j^{r_{t}-1*})$ in Step 2 of CR\&S Algorithm 2 satisfies $(i^{r_{t}-1*}, j^{r_{t}-1*})$ $\notin \bar{\Ccal}_{i_0,h_0}$,
	we can analyze similarly to case (2) and yield the contradiction. Now suppose the selected pair $(i^{r_{t}-1*}, j^{r_{t}-1*})$ in Step 2 of CR\&S Algorithm 2 satisfies $(i^{r_{t}-1*}, j^{r_{t}-1*}) = (i'_0,j'_0)$. Since $\hat{\lambda}_{(i'_0,j'_0),r_{t}-1} \ge \eta$, we have by \eqref{ineq:cons2} that
	$
	\left[ \nabla a(\hat{\blam}^{(r_{t}-1)},\overline{\Y}^{(r_{t}-1)}, \hat{\boldsymbol{\sigma}}^{2,(r_{t}-1)}) \right]_{(i'_0,j'_0)}   \ge -\frac{C_3}{\sqrt{\eta}}.
	$
	By $r_t$'s definition, $\hat{\lambda}_{(i'_0,j'_0),r_{t}} < \varepsilon$. Thus, $\dfrak^{(r_{t})} = e_{i^\dag,j^\dag} - e_{i'_0,j'_0}$ and $\sfrak^{(r_{t})} \ge \eta - \varepsilon$ because $\hat{\blam}^{(r_{t})} = \hat{\blam}^{(r_{t}-1)}+\sfrak^{(r_t)}\cdot \dfrak^{(r_t)}$ and $\hat{\lambda}_{(i'_0,j'_0),r_{t}-1} \ge \eta$. Moreover, $(i^\dag,j^\dag) \notin \bar{\Ccal}_{i_0,h_0}$ because $\hat{\lambda}_{(i^\dag,j^\dag),r_{t}} \ge \sfrak^{(r_{t})} \ge \eta - \varepsilon \ge \frac{2 \eta}{3} > \varepsilon$. By condition (13) of the main paper,
	\begin{align*}
		&a(\hat{\blam}^{(r_t)},\overline{\Y}^{(r_t-1)}, \hat{\boldsymbol{\sigma}}^{2,(r_t-1)}) - a(\hat{\blam}^{(r_t-1)},\overline{\Y}^{(r_t-1)}, \hat{\boldsymbol{\sigma}}^{2,(r_t-1)})
		\le  \sfrak_1 \sfrak^{(r_t)} \nabla a(\hat{\blam}^{(r_t-1)},\overline{\Y}^{(r_t-1)}, \hat{\boldsymbol{\sigma}}^{2,(r_t-1)})^\top \dfrak^{(r_t)}  \\
		=& \sfrak_1 \sfrak^{(r_t)} \left( \left[ \nabla a(\hat{\blam}^{(r_t-1)},\overline{\Y}^{(r_t-1)}, \hat{\boldsymbol{\sigma}}^{2,(r_t-1)}) \right]_{(i^\dag,j^\dag)} - \left[ \nabla a(\hat{\blam}^{(r_t-1)},\overline{\Y}^{(r_t-1)}, \hat{\boldsymbol{\sigma}}^{2,(r_t-1)}) \right]_{(i'_0,j'_0)} \right).
	\end{align*}
	Then,
	$
	\left[ \nabla a(\hat{\blam}^{(r_t-1)},\overline{\Y}^{(r_t-1)}, \hat{\boldsymbol{\sigma}}^{2,(r_t-1)}) \right]_{(i^\dag,j^\dag)} - \left[ \nabla a(\hat{\blam}^{(r_t-1)},\overline{\Y}^{(r_t-1)}, \hat{\boldsymbol{\sigma}}^{2,(r_t-1)}) \right]_{(i'_0,j'_0)} \\
	\ge \frac{1}{\sfrak_1 \sfrak^{(r_t)}} \left( a(\hat{\blam}^{(r_t)},\overline{\Y}^{(r_t-1)}, \hat{\boldsymbol{\sigma}}^{2,(r_t-1)}) - a(\hat{\blam}^{(r_t-1)},\overline{\Y}^{(r_t-1)}, \hat{\boldsymbol{\sigma}}^{2,(r_t-1)}) \right)
	\ge - \frac{kp\sqrt{C_1}}{\sfrak_1 \sfrak^{(r_t)}}
	\ge - \frac{3kp\sqrt{C_1}}{2\sfrak_1 \eta},
	$
	where the second inequality holds by \eqref{ineq:cons1} and the last inequality holds because $\sfrak^{(r_t)} \ge \frac{2 \eta}{3}$. Since $\hat{\lambda}_{(i^\dag,j^\dag),r_{t}} \ge \frac{2 \eta}{3}$, we have by \eqref{ineq:cons2} that
	$
	\left[ \nabla a(\hat{\blam}^{(r_t)},\overline{\Y}^{(r_t-1)}, \hat{\boldsymbol{\sigma}}^{2,(r_t-1)}) \right]_{(i^\dag,j^\dag)} \ge - C_3 \sqrt{\frac{3}{2  \eta}}.
	$
	By  condition (14) of the main paper, we have
	\begin{align}
		&\left[ \nabla a(\hat{\blam}^{(r_t)},\overline{\Y}^{(r_t-1)}, \hat{\boldsymbol{\sigma}}^{2,(r_t-1)}) \right]_{(i'_0,j'_0)}  \nonumber
		\ge \left[ \nabla a(\hat{\blam}^{(r_t)},\overline{\Y}^{(r_t-1)}, \hat{\boldsymbol{\sigma}}^{2,(r_t-1)}) \right]_{(i^\dag,j^\dag)}  \nonumber \\
		& - \sfrak_2 \left| \left[ \nabla a(\hat{\blam}^{(r_{t}-1)},\overline{\Y}^{(r_{t}-1)}, \hat{\boldsymbol{\sigma}}^{2,(r_{t}-1)}) \right]_{(i^\dag, j^\dag)} - \left[ \nabla a(\hat{\blam}^{(r_{t}-1)},\overline{\Y}^{(r_{t}-1)}, \hat{\boldsymbol{\sigma}}^{2,(r_{t}-1)}) \right]_{(i'_0,j'_0)} \right|  \nonumber \\
		\ge& - C_3 \sqrt{\frac{3}{2  \eta}} - \frac{3\sfrak_2kp\sqrt{C_1}}{2\sfrak_1 \eta}. \label{ineq:case3con1}
	\end{align}
	Similarly to \eqref{eq:case2contra2}, since $\lambda_{(i_1,j_1),r_t} \le \varepsilon$ for all $(i_1,j_1) \in \bar{\Ccal}_{i_0,h_0}$ by $r_t$'s definition, we have by \eqref{ineq:cons11} that
	\begin{align}\label{ineq:case3con2}
		\left[ \nabla a(\hat{\blam}^{(r_t)},\overline{\Y}^{(r_t-1)}, \hat{\boldsymbol{\sigma}}^{2,(r_t-1)}) \right]_{(i'_0,j'_0)} \le - \frac{C_2}{\sqrt{ \varepsilon  } }
		< - C_3 \sqrt{\frac{3}{2  \eta}} - \frac{3\sfrak_2kp\sqrt{C_1}}{2\sfrak_1 \eta}.
	\end{align}
	Again, \eqref{ineq:case3con1} and \eqref{ineq:case3con2} are contradictory. Thus, it is impossible to have $\hat{\lambda}_{(i'_0,j'_0),r_t} \le \varepsilon$. We have $\mathop{\lim \inf}_{r \to \infty} \chi_{i,h}(\hat{\blam}^{(r)}, \overline{\Y}^{(r)}, \hat{\boldsymbol{\sigma}}^{2,(r)}) > 0$ and thus $\mathop{\lim \inf}_{r \to \infty} \hat{\alpha}_{i,h}>0$, which is equivalently to $\hat{\alpha}_{i,h} = \Theta(1)$, almost surely.
\end{proof}

Next, we begin to show Theorem 5 of the main paper. Since $n_{i,h} \to \infty$ in CR\&S Algorithm 2 for all $i$ and $h$, we have $\widehat{\bbeta}_i \to \bbeta_i$,  $\bar{Y}_{i}(\x_h^\circ) \to y_i(\x_h^\circ)$ and $\hat{\sigma}_{i}^2(\x_h^\circ) \to \sigma_{i}^2(\x_h^\circ)$ as $r \to \infty$. In the following, we always assume $r \ge r_0$ is large enough such that all estimates will not deviate from their true values by more than $\varepsilon_0$. Note that given $(\boldsymbol{y}, \boldsymbol{\sigma}^2)$ (or $(\overline{\Y}^{(r-1)}, \hat{\boldsymbol{\sigma}}^{2,(r-1)})$), we can show that $a(\blam, \boldsymbol{y}, \boldsymbol{\sigma}^2)$ (or $a(\blam, \overline{\Y}^{(r-1)}, \hat{\boldsymbol{\sigma}}^{2,(r-1)})$) is convex because for $0 \le p_1 \le 1$, we have $\Big(\sqrt{\chi_{i,h}\big(p_1 \blam_1 + (1-p_1) \blam_2, \boldsymbol{y}, \boldsymbol{\sigma}^2\big)}\Big)^2 \ge \Big(p_1\sqrt{\chi_{i,h}( \blam_1, \boldsymbol{y}, \boldsymbol{\sigma}^2)} + (1-p_1) \sqrt{\chi_{i,h}( \blam_2, \boldsymbol{y}, \boldsymbol{\sigma}^2)}\Big)^2$. Moreover, if $\blam_1$ and $\blam_2$ are two different stationary points to (10) of the main paper, then
\begin{align}\label{eq:chiopt}
	\chi_{i,h}( \blam_1, \boldsymbol{y}, \boldsymbol{\sigma}^2) = \chi_{i,h}( \blam_2, \boldsymbol{y}, \boldsymbol{\sigma}^2)
\end{align}  
for $i=1,\dots,k$ and $h=1,\dots,p$. The reason is as follows. Suppose there exist $i_0$ and $h_0$ such that $\chi_{i_0,h_0}( \blam_1, \boldsymbol{y}, \boldsymbol{\sigma}^2) \ne \chi_{i_0,h_0}( \blam_2, \boldsymbol{y}, \boldsymbol{\sigma}^2)$. Then for $0 < p_1 < 1$,
{\small\begin{align*}
		&\Big(\sqrt{\chi_{i_0,h_0}\big(p_1 \blam_1 + (1-p_1) \blam_2, \boldsymbol{y}, \boldsymbol{\sigma}^2\big)}\Big)^2 - \Big(p_1\sqrt{\chi_{i_0,h_0}( \blam_1, \boldsymbol{y}, \boldsymbol{\sigma}^2)} + (1-p_1) \sqrt{\chi_{i_0,h_0}( \blam_2, \boldsymbol{y}, \boldsymbol{\sigma}^2)}\Big)^2  \\
		=& p_1 \chi_{i_0,h_0}\big( \blam_1 , \boldsymbol{y}, \boldsymbol{\sigma}^2\big) + (1-p_1) \chi_{i_0,h_0}\big(   \blam_2, \boldsymbol{y}, \boldsymbol{\sigma}^2\big) - p_1^2 \chi_{i_0,h_0}\big( \blam_1 , \boldsymbol{y}, \boldsymbol{\sigma}^2\big) - (1-p_1)^2 \chi_{i_0,h_0}\big(   \blam_2, \boldsymbol{y}, \boldsymbol{\sigma}^2\big) \\
		&- 2 p_1 (1-p_1) \sqrt{\chi_{i_0,h_0}( \blam_1, \boldsymbol{y}, \boldsymbol{\sigma}^2)}  \sqrt{\chi_{i_0,h_0}( \blam_2, \boldsymbol{y}, \boldsymbol{\sigma}^2)}\\
		=& p_1(1-p_1) \left( \sqrt{\chi_{i_0,h_0}\big( \blam_1 , \boldsymbol{y}, \boldsymbol{\sigma}^2\big)} - \sqrt{\chi_{i_0,h_0}( \blam_2, \boldsymbol{y}, \boldsymbol{\sigma}^2)} \right)^2 
		> 0,
\end{align*}}
which means 
\begin{align*}
	&a(p_1 \blam_1 + (1-p_1) \blam_2, \boldsymbol{y}, \boldsymbol{\sigma}^2) =  -\sum_{i=1}^k \sum_{h=1}^p \sqrt{ \chi_{i,h}(p_1 \blam_1 + (1-p_1) \blam_2, \boldsymbol{y}, \boldsymbol{\sigma}^2)  }  \\
	&< -p_1 \sum_{i=1}^k \sum_{h=1}^p \sqrt{ \chi_{i,h}(\blam_1 , \boldsymbol{y}, \boldsymbol{\sigma}^2)  } -(1-p_1) \sum_{i=1}^k \sum_{h=1}^p \sqrt{ \chi_{i,h}(\blam_2, \boldsymbol{y}, \boldsymbol{\sigma}^2)  }  \\
	&= p_1 a(\blam_1, \boldsymbol{y}, \boldsymbol{\sigma}^2) + (1-p_1) a( \blam_2, \boldsymbol{y}, \boldsymbol{\sigma}^2)  
	= a(\blam_1, \boldsymbol{y}, \boldsymbol{\sigma}^2)
\end{align*}
where the last equality holds because $\blam_1$ and $\blam_2$ are both stationary points to the convex program (10) of the main paper such that $a(\blam_1, \boldsymbol{y}, \boldsymbol{\sigma}^2) = a(\blam_2, \boldsymbol{y}, \boldsymbol{\sigma}^2)$. However, $a(p_1 \blam_1 + (1-p_1) \blam_2, \boldsymbol{y}, \boldsymbol{\sigma}^2) < a(\blam_1, \boldsymbol{y}, \boldsymbol{\sigma}^2)$ is contradictory to the assumption that $\blam_1$ is a stationary point to the convex program (10). 

Let $a^*$ denote the optimal value of (10) and  $\boldsymbol{\chi}^* \triangleq (\chi^*_{1,1},\dots,\chi^*_{1,p},\chi^*_{2,1},\dots,\chi^*_{2,p},\dots,\chi^*_{k,1},$ $\dots,\chi^*_{k,p})^\top$ denote the value of $\boldsymbol{\chi} (\blam_1, \boldsymbol{y}, \boldsymbol{\sigma}^2) \triangleq (\chi_{1,1}(\blam_1, \boldsymbol{y}, \boldsymbol{\sigma}^2),\dots,\chi_{1,p}(\blam_1, \boldsymbol{y}, \boldsymbol{\sigma}^2),\chi_{2,1}(\blam_1, \boldsymbol{y}, \boldsymbol{\sigma}^2),\dots,\chi_{2,p}(\blam_1, $ $\boldsymbol{y}, \boldsymbol{\sigma}^2),\dots,\chi_{k,1}(\blam_1, \boldsymbol{y}, \boldsymbol{\sigma}^2),\dots,\chi_{k,p}($ $ \blam_1, \boldsymbol{y}, \boldsymbol{\sigma}^2))^\top$ where $\blam_1$ is a stationary point to (10). Let $\Xi$ denote the value space of $\boldsymbol{\chi} (\blam, \boldsymbol{y}, \boldsymbol{\sigma}^2)$ where $\blam$ satisfies the constraints of (10). For any small enough $\varepsilon > 0$, let $\Xi_{\varepsilon} = \{ \boldsymbol{\chi} (\blam, \boldsymbol{y}, \boldsymbol{\sigma}^2) \in \Xi: \| \boldsymbol{\chi} (\blam, \boldsymbol{y}, \boldsymbol{\sigma}^2) - \boldsymbol{\chi}^* \|_{\infty} \ge \varepsilon   \}$. Note that by \eqref{eq:chiopt}, $a(\blam, \boldsymbol{y}, \boldsymbol{\sigma}^2)$ achieves the minimal value $a^*$ when $\blam$ satisfies $\boldsymbol{\chi} (\blam, \boldsymbol{y}, \boldsymbol{\sigma}^2) = \boldsymbol{\chi}^*$. If there exists a sequence of $\blam$ satisfying each  $\boldsymbol{\chi} (\blam, \boldsymbol{y}, \boldsymbol{\sigma}^2) \in \Xi_{\varepsilon}$ such that the corresponding sequence of $a(\blam, \boldsymbol{y}, \boldsymbol{\sigma}^2)$ converges to $a^*$, then there must exist a $\blam^\dag$ satisfying  $\boldsymbol{\chi} (\blam^\dag, \boldsymbol{y}, \boldsymbol{\sigma}^2) \in \Xi_{\varepsilon}$ and $a(\blam^\dag, \boldsymbol{y}, \boldsymbol{\sigma}^2) = a^*$ because $\Xi_{\varepsilon}$ is closed and bounded. This leads to contradiction because by \eqref{eq:chiopt}, $\boldsymbol{\chi} (\blam^\dag, \boldsymbol{y}, \boldsymbol{\sigma}^2) = \boldsymbol{\chi}^* $ should hold but $\boldsymbol{\chi}^* \notin \Xi_{\varepsilon}$. Thus, $a^*_{\varepsilon} > a^*$ where $a^*_{\varepsilon} = \min\{a(\blam, \boldsymbol{y}, \boldsymbol{\sigma}^2): \ \blam \text{ satisfies }  \boldsymbol{\chi} (\blam, \boldsymbol{y}, \boldsymbol{\sigma}^2) \in \Xi_{\varepsilon}\}$. We will show that for any $\varepsilon$ small enough, the $\boldsymbol{\chi} (\hat{\blam}^{(r)}, \boldsymbol{y}, \boldsymbol{\sigma}^2)$ of $\hat{\blam}^{(r)}$ of CR\&S Algorithm 2 will be in $\Xi \setminus \Xi_{\varepsilon} = \{ \boldsymbol{\chi} (\blam, \boldsymbol{y}, \boldsymbol{\sigma}^2) \in \Xi: \| \boldsymbol{\chi} (\blam, \boldsymbol{y}, \boldsymbol{\sigma}^2) - \boldsymbol{\chi}^* \|_{\infty} < \varepsilon   \}$ for $r$ large enough. To have the above result, we need two facts: a) $a(\hat{\blam}^{(r)},\boldsymbol{y}, \boldsymbol{\sigma}^2)$ decreases with $r$ when $r$ is large enough and b) there exists a subsequence of $\{\hat{\blam}^{(r)}, r=1,2,\dots\}$ that converges to a stationary point of problem (10) such that $a(\hat{\blam}^{(r)},\boldsymbol{y}, \boldsymbol{\sigma}^2) < a^*_{\varepsilon}$ for $r$ large enough.

By Step 4 of CR\&S Algorithm 2, $\hat{\blam}^{(r)}$ will be updated only when $W^{(r)}< \max\{ -\kappa_0, -(\frac{\log r}{r})^{1/4} \}$ and $\sfrak^{\max}(\dfrak^{(r)},\hat{\blam}^{(r-1)}) W^{(r)}< \max\{ -\kappa_0, -(\frac{\log r}{r})^{1/2} \}$, where we let $W^{(r)} = \nabla a(\hat{\blam}^{(r-1)}, \overline{\Y}^{(r-1)}, \hat{\boldsymbol{\sigma}}^{2,(r-1)})^\top \dfrak^{(r)}$ for notation simplicity. Suppose $\hat{\blam}^{(r)}$ is updated at iteration $r$.
By \eqref{ineq:lowbdalgo2} of Lemma \ref{lem:lincons}, we know that $\mathop{\lim \inf}_{r \to \infty} \chi_{i,h}(\hat{\blam}^{(r)},$ $ \overline{\Y}^{(r)}, \hat{\boldsymbol{\sigma}}^{2,(r)}) > 0$ for all $i=1,\dots,k$ and $h=1,\dots,p$. Let $C_{\chi} > 0$ denote the lower bound of $ \chi_{i,h}(\hat{\blam}^{(r)}, \overline{\Y}^{(r)}, \hat{\boldsymbol{\sigma}}^{2,(r)})$, $i=1,\dots,k$, $h=1,\dots,p$ for $r$ large enough. Then,
\begin{align*}
	&\left| [\nabla a(\hat{\blam}_1, \overline{\Y}^{(r-1)}, \hat{\boldsymbol{\sigma}}^{2,(r-1)})]_{i_1,j_1} - [\nabla a(\hat{\blam}_2, \overline{\Y}^{(r-1)}, \hat{\boldsymbol{\sigma}}^{2,(r-1)})]_{i_1,j_1} \right|  \\
	=& \left| -\sum_{(i,h) \in \Ccal_{(i_1, j_1)}} \varsigma_{i_1, j_1, i, h } (\hat{\blam}_1,\overline{\Y}^{(r)}, \hat{\boldsymbol{\sigma}}^{2,(r)}) + \sum_{(i,h) \in \Ccal_{(i_1, j_1)}} \varsigma_{i_1, j_1, i, h } (\hat{\blam}_2,\overline{\Y}^{(r)}, \hat{\boldsymbol{\sigma}}^{2,(r)}) \right|  \\
	\le& \sum_{(i,h) \in \Ccal_{(i_1, j_1)}} \left|  \frac{ \hat{\sigma}_{i}^2(\x_h^\circ) \psi_{i_1,j_1,i^*(\x_{j_1}),h}(\overline{\Y}^{(r)})}{2\sqrt{ \chi_{i,h}(\hat{\blam}_1, \overline{\Y}^{(r)}, \hat{\boldsymbol{\sigma}}^{2,(r)} )  }} - \frac{ \hat{\sigma}_{i}^2(\x_h^\circ) \psi_{i_1,j_1,i^*(\x_{j_1}),h}(\overline{\Y}^{(r)})}{2\sqrt{ \chi_{i,h}(\hat{\blam}_2, \overline{\Y}^{(r)}, \hat{\boldsymbol{\sigma}}^{2,(r)} )  }} \right|  \\
	\le& \sum_{(i,h) \in \Ccal_{(i_1, j_1)}} \frac{(\psi_{\max} + \varepsilon) (\bar{\sigma}^2_{\max} + \varepsilon) }{2}  \left|  \frac{ 1}{\sqrt{ \chi_{i,h}(\hat{\blam}_1, \overline{\Y}^{(r)}, \hat{\boldsymbol{\sigma}}^{2,(r)} )  }} - \frac{ 1}{\sqrt{ \chi_{i,h}(\hat{\blam}_2, \overline{\Y}^{(r)}, \hat{\boldsymbol{\sigma}}^{2,(r)} )  }} \right|  \\
	\le& \sum_{(i,h) \in \Ccal_{(i_1, j_1)}} \frac{(\psi_{\max} + \varepsilon) (\bar{\sigma}^2_{\max} + \varepsilon) }{4 C_{\chi}^{3/2}}  \left| \chi_{i,h}(\hat{\blam}_2, \overline{\Y}^{(r)}, \hat{\boldsymbol{\sigma}}^{2,(r)} )  - \chi_{i,h}(\hat{\blam}_1, \overline{\Y}^{(r)}, \hat{\boldsymbol{\sigma}}^{2,(r)} ) \right|    \\
	\le&  \sum_{(i,h) \in \Ccal_{(i_1, j_1)}} \frac{(\psi_{\max} + \varepsilon) (\bar{\sigma}^2_{\max} + \varepsilon) }{4 C_{\chi}^{3/2}} \sum\limits_{(i_1,j_1) \in \bar{\Ccal}_{i,h}} \left| \hat{\sigma}_{i,r}^{2}(\x_h^\circ)\psi_{i_1,j_1,\bar{i}^*(\x_{j_1}),h}(\overline{\Y}^{(r)}) \left(\hat{\lambda}_{(i_1,j_1),2} - \hat{\lambda}_{(i_1,j_1),1} \right)  \right|  \\
	\le& \sum_{(i,h) \in \Ccal_{(i_1, j_1)}} \frac{(\psi_{\max} + \varepsilon)^2 (\bar{\sigma}^2_{\max} + \varepsilon)^2 }{4 C_{\chi}^{3/2}} \| \hat{\blam}_1 - \hat{\blam}_2\|_1 \triangleq C_{\varsigma} \| \hat{\blam}_1 - \hat{\blam}_2\|_1.
\end{align*}
Thus, the second order derivative of $a(\hat{\blam}, \overline{\Y}^{(r-1)}, \hat{\boldsymbol{\sigma}}^{2,(r-1)})$ for $\hat{\lambda}_{i,j}$ is bounded. By Taylor's theorem, we have
$a(\hat{\blam}^{(r-1)} + \sfrak^{(r)} \dfrak^{(r)}, \overline{\Y}^{(r-1)}, \hat{\boldsymbol{\sigma}}^{2,(r-1)})
\le a(\hat{\blam}^{(r-1)} , \overline{\Y}^{(r-1)}, \hat{\boldsymbol{\sigma}}^{2,(r-1)}) + \sfrak^{(r)} W^{(r)} + \frac{ (\sfrak^{(r)})^2 C_{\varsigma} }{2} \|  \dfrak^{(r)} \|^2_2.
$
To satisfy the condition $a(\hat{\blam}^{(r-1)}+\sfrak^{(r)} \cdot \dfrak^{(r)},\overline{\Y}^{(r-1)}, \hat{\boldsymbol{\sigma}}^{2, (r-1)}) \le a(\hat{\blam}^{(r-1)},\overline{\Y}^{(r-1)}, \hat{\boldsymbol{\sigma}}^{2,(r-1)}) +\sfrak_1 \sfrak^{(r)} W^{(r)}$ of Algorithm 3 of the main paper, a sufficient condition is
\begin{align*}
	a(\hat{\blam}^{(r-1)} , \overline{\Y}^{(r-1)}, \hat{\boldsymbol{\sigma}}^{2,(r-1)}) + \sfrak^{(r)} W^{(r)} + \frac{ (\sfrak^{(r)})^2 C_{\varsigma} }{2} \|  \dfrak^{(r)} \|^2_2
	\le a(\hat{\blam}^{(r-1)},\overline{\Y}^{(r-1)}, \hat{\boldsymbol{\sigma}}^{2,(r-1)}) +\sfrak_1 \sfrak^{(r)} W^{(r)},
\end{align*}
which yields
$
(1-\sfrak_1)\sfrak^{(r)} W^{(r)} + \frac{ (\sfrak^{(r)})^2 C_{\varsigma} }{2} \|  \dfrak^{(r)} \|^2_2 \le 0.
$
Note that $\|  \dfrak^{(r)} \|^2_2 = 2$ because $\dfrak^{(r)}$ is a vector that has one element equal to one, one element equal to minus one, and other elements equal to zero. Thus, it is sufficient to have $\sfrak^{(r)} \le \frac{(\sfrak_1-1) W^{(r)} }{C_\varsigma}$.

Define $\sfrak_U^{(r)} = \frac{(\sfrak_1-1) W^{(r)} }{C_\varsigma}$ and $\sfrak_*^{(r)} = \arg\min_{\sfrak} a(\hat{\blam}^{(r-1)}+\sfrak \dfrak^{(r)},\overline{\Y}^{(r-1)}, \hat{\boldsymbol{\sigma}}^{2,(r-1)})$. Let $\sfrak_D^{(r)} = \max \{\sfrak : \nabla a(\hat{\blam}^{(r-1)}+\sfrak\cdot \dfrak^{(r)},\overline{\Y}^{(r-1)}, \hat{\boldsymbol{\sigma}}^{2,(r-1)})^\top \dfrak^{(r)}  \le \sfrak_2 |W^{(r)}|\}$. By the convexity of $a(\blam,\overline{\Y}^{(r-1)}, \hat{\boldsymbol{\sigma}}^{2,(r-1)})$, we have $\sfrak_*^{(r)} \le \sfrak_D^{(r)}$. Next, we discuss by cases.
\begin{enumerate}
	\item[(1)] If $\sfrak^{\max}(\dfrak^{(r)},\hat{\blam}^{(r-1)},\overline{\Y}^{(r-1)}, \hat{\boldsymbol{\sigma}}^{2,(r-1)}) \le \min\{ \sfrak_U^{(r)}, \sfrak_D^{(r)} \}$, we have $\sfrak^{(r)} = \sfrak^{\max}(\dfrak^{(r)},\hat{\blam}^{(r-1)})$ and by the definition of  $\sfrak_U^{(r)}$,
	$a(\hat{\blam}^{(r-1)},\overline{\Y}^{(r-1)}, \hat{\boldsymbol{\sigma}}^{2,(r-1)}) - a(\hat{\blam}^{(r)},\overline{\Y}^{(r-1)}, \hat{\boldsymbol{\sigma}}^{2,(r-1)})
	\ge -\sfrak_1 \sfrak^{\max}(\dfrak^{(r)},\hat{\blam}^{(r-1)}) W^{(r)}.
	$
	
	\item[(2)] If $\sfrak^{\max}(\dfrak^{(r)},\hat{\blam}^{(r-1)}) > \sfrak_D^{(r)} \ge \sfrak_U^{(r)}$, then $\sfrak^{(r)} \ge \tau \sfrak_U^{(r)}$ must hold. Thus,
	$
	a(\hat{\blam}^{(r-1)},\overline{\Y}^{(r-1)}, \hat{\boldsymbol{\sigma}}^{2,(r-1)}) - a(\hat{\blam}^{(r)},\overline{\Y}^{(r-1)}, \hat{\boldsymbol{\sigma}}^{2,(r-1)})
	\ge -\sfrak_1 \sfrak^{(r)}  W^{(r)}
	\ge \frac{ \tau \sfrak_1 (1-\sfrak_1) (W^{(r)})^2 }{C_\varsigma}.
	$
	
	\item[(3)] If $\sfrak^{\max}(\dfrak^{(r)},\hat{\blam}^{(r-1)}) > \sfrak_U^{(r)} \ge \sfrak_D^{(r)} $, then $\tau \sfrak_*^{(r)} \le \tau \sfrak_D^{(r)} \le \sfrak^{(r)} \le \sfrak_D^{(r)} \le \sfrak_U^{(r)}$ must hold. Note that $a(\blam,\overline{\Y}^{(r-1)}, \hat{\boldsymbol{\sigma}}^{2,(r-1)})$ is convex,
	\begin{align*}
		&a(\hat{\blam}^{(r)},\overline{\Y}^{(r-1)}, \hat{\boldsymbol{\sigma}}^{2,(r-1)}) = a(\hat{\blam}^{(r-1)}+\sfrak^{(r)} \dfrak^{(r)},\overline{\Y}^{(r-1)}, \hat{\boldsymbol{\sigma}}^{2,(r-1)})  \\
		\le& \max \left\{ a(\hat{\blam}^{(r-1)}+\tau \sfrak_*^{(r)} \dfrak^{(r)},\overline{\Y}^{(r-1)}, \hat{\boldsymbol{\sigma}}^{2,(r-1)}), a(\hat{\blam}^{(r-1)}+\sfrak_U^{(r)} \dfrak^{(r)},\overline{\Y}^{(r-1)}, \hat{\boldsymbol{\sigma}}^{2,(r-1)}) \right\}.
	\end{align*}
	By the definition of $\sfrak_U^{(r)}$,
	$ a(\hat{\blam}^{(r-1)},\overline{\Y}^{(r-1)}, \hat{\boldsymbol{\sigma}}^{2,(r-1)}) - a(\hat{\blam}^{(r-1)}+ \sfrak_U^{(r)} \dfrak^{(r)},\overline{\Y}^{(r-1)}, \hat{\boldsymbol{\sigma}}^{2,(r-1)})
	\ge -\sfrak_1 \sfrak_U^{(r)} W^{(r)}
	= \frac{\sfrak_1 (1-\sfrak_1) (W^{(r)})^2}{C_\varsigma}  .
	$
	Meanwhile, by the convexity,  $a(\hat{\blam}^{(r-1)}+\tau \sfrak_*^{(r)} \dfrak^{(r)},\overline{\Y}^{(r-1)}, \hat{\boldsymbol{\sigma}}^{2,(r-1)}) $ $\le (1-\tau) a(\hat{\blam}^{(r-1)},\overline{\Y}^{(r-1)}, \hat{\boldsymbol{\sigma}}^{2,(r-1)}) + \tau a(\hat{\blam}^{(r-1)}+ \sfrak_*^{(r)} \dfrak^{(r)},\overline{\Y}^{(r-1)}, \hat{\boldsymbol{\sigma}}^{2,(r-1)})$. Then,
	\begin{align*}
		&a(\hat{\blam}^{(r-1)},\overline{\Y}^{(r-1)}, \hat{\boldsymbol{\sigma}}^{2,(r-1)}) - a(\hat{\blam}^{(r-1)}+\tau \sfrak_*^{(r)} \dfrak^{(r)},\overline{\Y}^{(r-1)}, \hat{\boldsymbol{\sigma}}^{2,(r-1)})  \\
		\ge& \tau \left( a(\hat{\blam}^{(r-1)},\overline{\Y}^{(r-1)}, \hat{\boldsymbol{\sigma}}^{2,(r-1)}) - a(\hat{\blam}^{(r-1)}+ \sfrak_*^{(r)} \dfrak^{(r)},\overline{\Y}^{(r-1)}, \hat{\boldsymbol{\sigma}}^{2,(r-1)}) \right)  \\
		\ge& \tau \left( a(\hat{\blam}^{(r-1)},\overline{\Y}^{(r-1)}, \hat{\boldsymbol{\sigma}}^{2,(r-1)}) - a(\hat{\blam}^{(r-1)}+ \sfrak_U^{(r)} \dfrak^{(r)},\overline{\Y}^{(r-1)}, \hat{\boldsymbol{\sigma}}^{2,(r-1)}) \right)
		\ge \tau \frac{\sfrak_1 (1-\sfrak_1) (W^{(r)})^2}{C_\varsigma}
	\end{align*}
	Thus,
	$ a(\hat{\blam}^{(r-1)},\overline{\Y}^{(r-1)}, \hat{\boldsymbol{\sigma}}^{2,(r-1)}) - a(\hat{\blam}^{(r)},\overline{\Y}^{(r-1)}, \hat{\boldsymbol{\sigma}}^{2,(r-1)})
	\ge   \frac{\tau \sfrak_1 (1-\sfrak_1) (W^{(r)})^2 }{C_\varsigma}.
	$	
\end{enumerate}
Summarizing the results of the above three cases, we have
\begin{align}
	& a(\hat{\blam}^{(r-1)},\overline{\Y}^{(r-1)}, \hat{\boldsymbol{\sigma}}^{2,(r-1)}) - a(\hat{\blam}^{(r)},\overline{\Y}^{(r-1)}, \hat{\boldsymbol{\sigma}}^{2,(r-1)})  \nonumber \\
	\ge & \min \Bigg\{  \frac{\tau \sfrak_1 (1-\sfrak_1) (W^{(r)})^2 }{C_\varsigma},   -\sfrak_1 \sfrak^{\max}(\dfrak^{(r)},\hat{\blam}^{(r-1)}) W^{(r)}  \Bigg\}.  \label{ineq:lwbd_lin_conv}
\end{align}
Note that $W^{(r)} < \max\{ -\kappa_0, -(\frac{\log r}{r})^{1/4} \}$ and $\sfrak^{\max}(\dfrak^{(r)},\hat{\blam}^{(r-1)}) W^{(r)} < \max\{ -\kappa_0, -(\frac{\log r}{r})^{1/2} \}$ because we assume that $\hat{\blam}^{(r-1)}$ is updated at iteration $r$. We have by \eqref{ineq:lwbd_lin_conv} that
{\small\begin{align}\label{ineq:lwbd_lin_conv2}
		a(\hat{\blam}^{(r-1)},\overline{\Y}^{(r-1)}, \hat{\boldsymbol{\sigma}}^{2,(r-1)}) - a(\hat{\blam}^{(r)},\overline{\Y}^{(r-1)}, \hat{\boldsymbol{\sigma}}^{2,(r-1)})  \ge \min \left\{ \frac{\tau \sfrak_1 (1-\sfrak_1) (\frac{\log r}{r})^{\frac{1}{2}} }{C_\varsigma}, \sfrak_1 \left( \frac{\log r}{r} \right)^{\frac{1}{2}} \right\}.
\end{align}}

Then, we can show that $a(\hat{\blam}^{(r-1)},\boldsymbol{y}, \boldsymbol{\sigma}^2) - a(\hat{\blam}^{(r)},\boldsymbol{y}, \boldsymbol{\sigma}^2) > 0$ as follows.
\begin{align*}
	&a(\hat{\blam}^{(r-1)},\boldsymbol{y}, \boldsymbol{\sigma}^2) - a(\hat{\blam}^{(r)},\boldsymbol{y}, \boldsymbol{\sigma}^2)
	=a(\hat{\blam}^{(r-1)},\boldsymbol{y}, \boldsymbol{\sigma}^2) - a(\hat{\blam}^{(r-1)},\overline{\Y}^{(r-1)}, \hat{\boldsymbol{\sigma}}^{2,(r-1)}) \\
	&+ a(\hat{\blam}^{(r-1)},\overline{\Y}^{(r-1)}, \hat{\boldsymbol{\sigma}}^{2,(r-1)}) - a(\hat{\blam}^{(r)},\overline{\Y}^{(r-1)}, \hat{\boldsymbol{\sigma}}^{2,(r-1)}) + a(\hat{\blam}^{(r)},\overline{\Y}^{(r-1)}, \hat{\boldsymbol{\sigma}}^{2,(r-1)}) - a(\hat{\blam}^{(r)},\boldsymbol{y}, \boldsymbol{\sigma}^2).
\end{align*}
By the continuity of $a(\blam,\boldsymbol{y}, \boldsymbol{\sigma}^2)$ in $(\boldsymbol{y}, \boldsymbol{\sigma}^2)$, we have $\left| a(\blam,\boldsymbol{y}, \boldsymbol{\sigma}^2) - a(\blam,\overline{\Y}^{(r-1)}, \hat{\boldsymbol{\sigma}}^{2,(r-1)}) \right| \le C_{a} ( \|\boldsymbol{y}-\overline{\Y}^{(r-1)} \|_1 + \|\boldsymbol{\sigma}^2-\hat{\boldsymbol{\sigma}}^{2,(r-1)} \|_1 ).$ Since $\hat{\alpha}_{i,h} = \Theta(1)$, by the law of iterated logarithm, we have that $\left\|\boldsymbol{y}-\overline{\Y}^{(r-1)}\right\|_1 \le O \left( \sqrt{ (\log \log r) r^{-1} } \right)$ and $\left\|\boldsymbol{\sigma}^2-\hat{\boldsymbol{\sigma}}^{2,(r-1)}\right\|_1 \le O \left( \sqrt{ (\log \log r) r^{-1} } \right)$.
Thus, $|a(\hat{\blam}^{(r-1)},\boldsymbol{y}, \boldsymbol{\sigma}^2) - a(\hat{\blam}^{(r-1)},\overline{\Y}^{(r-1)}, \hat{\boldsymbol{\sigma}}^{2,(r-1)})|$ and $|a(\hat{\blam}^{(r)},\overline{\Y}^{(r-1)}, \hat{\boldsymbol{\sigma}}^{2,(r-1)}) - a(\hat{\blam}^{(r)},\boldsymbol{y}, \boldsymbol{\sigma}^2) |$ are of order $O \left( \sqrt{(\log \log r) r^{-1}} \right)$. Combining this order with \eqref{ineq:lwbd_lin_conv2}, we have $a(\hat{\blam}^{(r-1)},\boldsymbol{y}, \boldsymbol{\sigma}^2) - a(\hat{\blam}^{(r)},\boldsymbol{y}, \boldsymbol{\sigma}^2)$ $> 0$ for $r$ sufficiently large.

Note that $\hat{\blam}^{(r)}$ either remains unchanged or is updated by $\hat{\blam}^{(r)} = \hat{\blam}^{(r-1)} + \sfrak^{(r)} \dfrak^{(r)}$. Thus,  $a(\hat{\blam}^{(r)},\boldsymbol{y}, \boldsymbol{\sigma}^2)$ decreases with $r$. Combining this monotone property with the fact that $a(\hat{\blam}^{(r)},\boldsymbol{y}, \boldsymbol{\sigma}^2)$ is continuous in $\hat{\blam}^{(r)}$ and bounded, we have $a(\hat{\blam}^{(r)},\boldsymbol{y}, \boldsymbol{\sigma}^2)$ converges to some limiting point.

Denote the limiting point of $\hat{\blam}^{(r)}$ as $\blam^*$. Next, we show by contradiction that $\blam^*$ is a stationary point. By Lemma \ref{lem:ref}, it is sufficient to prove $\nabla a(\blam^*,\boldsymbol{y}, \boldsymbol{\sigma}^2)^\top \dfrak \ge 0$ for any feasible direction $\dfrak \in \Dcal^{(i,j)}(\blam^*)$ such that $\lambda^*_{i,j} \ge \eta$. Suppose a feasible direction $\dfrak_0 \in \Dcal^{(i,j)}(\blam^*)$ satisfies  $\nabla a(\blam^*,\boldsymbol{y}, \boldsymbol{\sigma}^2)^\top \dfrak_0 < 0$. We analyze the subsequence of $\{\hat{\blam}^{(r)}, r=1,2,\dots\}$ for which $(i,j)$ is chosen in Step 2 of CR\&S Algorithm 2 and $\hat{\blam}^{(r)}$ converges to $\blam^*$. By the continuity, there exists $c_{\dfrak} > 0$ such that $\nabla a(\hat{\blam}^{(r)},\overline{\Y}^{(r-1)}, \hat{\boldsymbol{\sigma}}^{2,(r-1)})^\top \dfrak_0 < -c_{\dfrak} <  0$ for all $r$ large enough. Moreover, since $\dfrak_0 \in \Dcal^{(i,j)}(\blam^*)$,  we have $\sfrak^{\max}(\dfrak_0,\blam^*) > 0$ and there exists a lower bound $c_{\sfrak} > 0$ such that $\sfrak^{\max}(\dfrak_0,\hat{\blam}^{(r-1)}) \ge c_{\sfrak}$ for all $r$ large enough. Thus,
\begin{align}\label{ineq:gradless}
	\sfrak^{\max}(\dfrak_0,\hat{\blam}^{(r-1)}) \nabla a(\hat{\blam}^{(r-1)},\overline{\Y}^{(r-1)}, \hat{\boldsymbol{\sigma}}^{2,(r-1)})^\top \dfrak_0 \le - c_{\sfrak} c_{\dfrak} < 0.
\end{align}
We analyze it by two cases.
\begin{itemize}
	\item[(1)]  Suppose for any $r_0 > 0$, we can find $r> r_0$ such that $(i,j)$ is chosen in Step 2 of CR\&S Algorithm 2 and $\hat{\blam}^{(r)}$ is updated. By condition (13) of Algorithm 3, $a(\hat{\blam}^{(r-1)},\overline{\Y}^{(r-1)}, \hat{\boldsymbol{\sigma}}^{2,(r-1)}) - a(\hat{\blam}^{(r)},\overline{\Y}^{(r-1)}, \hat{\boldsymbol{\sigma}}^{2,(r-1)}) \ge -\sfrak_1 \sfrak^{(r)} W^{(r)} $. By definition of $\dfrak^{(r)}$, we have $\sfrak^{\max}(\dfrak^{(r)},\hat{\blam}^{(r-1)}) W^{(r)} \le \sfrak^{\max}(\dfrak_0,\hat{\blam}^{(r-1)}) \nabla a(\hat{\blam}^{(r-1)},\overline{\Y}^{(r-1)}, \hat{\boldsymbol{\sigma}}^{2,(r-1)})^\top \dfrak_0 \le - c_{\sfrak} c_{\dfrak}$. Then $W^{(r)} \le \frac{-c_{\sfrak} c_{\dfrak}}{ \sfrak^{\max}(\dfrak^{(r)},\hat{\blam}^{(r-1)}) } \le -c_{\sfrak} c_{\dfrak}$, which yields
	$
	a(\hat{\blam}^{(r-1)},\overline{\Y}^{(r-1)}, \hat{\boldsymbol{\sigma}}^{2,(r-1)}) - a(\hat{\blam}^{(r)},\overline{\Y}^{(r-1)}, \hat{\boldsymbol{\sigma}}^{2,(r-1)}) \ge -\sfrak_1 \sfrak^{(r)} W^{(r)} \ge \sfrak_1 \sfrak^{(r)} c_{\sfrak} c_{\dfrak}$.
	Since $a(\hat{\blam}^{(r)},\boldsymbol{y}, \boldsymbol{\sigma}^2)$ converges and $a(\hat{\blam}^{(r-1)},\overline{\Y}^{(r-1)}, \hat{\boldsymbol{\sigma}}^{2,(r-1)}) - a(\hat{\blam}^{(r)},\overline{\Y}^{(r-1)},$ $ \hat{\boldsymbol{\sigma}}^{2,(r-1)})$ $\to 0$ as $r \to \infty$, we have  $\sfrak^{(r)} \to 0$ as $r \to \infty$. Note that $\sfrak^{\max}(\dfrak^{(r)},\hat{\blam}^{(r-1)}) \ge c_{\sfrak} > 0$. We have $ \sfrak^{(r)} < \sfrak^{\max}(\dfrak^{(r)},\hat{\blam}^{(r-1)})$ for all $r$ large enough.
	Then Algorithm 3 must have at least one loop and $ \hat{\blam}^{(r-1)} + \frac{\sfrak^{(r)}}{\tau} \dfrak^{(r)}$ violates at least one of conditions (13) and (14) of Algorithm 3. If condition (13) is violated, then
	$
	a(\hat{\blam}^{(r-1)} + \frac{\sfrak^{(r)}}{\tau} \dfrak^{(r)},\overline{\Y}^{(r-1)}, \hat{\boldsymbol{\sigma}}^{2,(r-1)}) - a(\hat{\blam}^{(r-1)},\overline{\Y}^{(r-1)}, \hat{\boldsymbol{\sigma}}^{2,(r-1)}) > \sfrak_1 \frac{\sfrak^{(r)}}{\tau} W^{(r)}$,
	and thus
	\begin{align*}
		\frac{a(\hat{\blam}^{(r-1)} + \frac{\sfrak^{(r)}}{\tau} \dfrak^{(r)},\overline{\Y}^{(r-1)}, \hat{\boldsymbol{\sigma}}^{2,(r-1)}) - a(\hat{\blam}^{(r-1)},\overline{\Y}^{(r-1)}, \hat{\boldsymbol{\sigma}}^{2,(r-1)})}{ \frac{\sfrak^{(r)}}{\tau} } > \sfrak_1  W^{(r)}.
	\end{align*}
	Letting $r \to \infty$, we have $\lim\inf_{r \to \infty} (1-\sfrak_1) \nabla a(\blam^{(r-1)},\boldsymbol{y}, \boldsymbol{\sigma}^2)^\top \dfrak^{(r)} \ge 0$. (Similar arguments can be found in the literature, e.g., Proposition 1.2.1 in \cite{bertsekas1999}.) If condition (14) is violated, then $\nabla a(\hat{\blam}^{(r-1)}+\frac{\sfrak^{(r)}}{\tau}  \dfrak^{(r)},\overline{\Y}^{(r-1)}, \hat{\boldsymbol{\sigma}}^{2,(r-1)})^\top \dfrak^{(r)} > \sfrak_2 |W^{(r)}| \ge 0$, which yields $ \lim\inf_{r \to \infty} \nabla a(\blam^{(r-1)},\boldsymbol{y}, \boldsymbol{\sigma}^2)^\top \dfrak^{(r)} \ge 0$ by noting that $\sfrak^{(r)}$ converges to zero.
	
	\item[(2)] If there is a $r_0$ such that $\hat{\blam}^{(r)}$ is not updated when $(i,j)$ is chosen for all $r > r_0$, we have by Step 4 of CR\&S Algorithm 2 that $W^{(r)} \ge -(\frac{\log r}{r})^{1/4}$ or $W^{(r)} \ge -\frac{1}{c_{\sfrak}} (\frac{\log r}{r})^{1/4}$, which also yields $ \lim\inf_{r \to \infty} \nabla a(\blam^{(r-1)},\boldsymbol{y}, \boldsymbol{\sigma}^2)^\top \dfrak^{(r)}$ $\ge 0$.
\end{itemize}
Summarizing the results of the two cases above, we have $ \lim\inf_{r \to \infty} \nabla a(\blam^{(r-1)},\boldsymbol{y}, \boldsymbol{\sigma}^2)^\top \dfrak^{(r)} \ge 0$. By continuity, $  \lim\inf_{r \to \infty} \sfrak^{\max}(\dfrak^{(r)},\hat{\blam}^{(r-1)}) W^{(r)} = \lim\inf_{r \to \infty} \sfrak^{\max}(\dfrak^{(r)},\hat{\blam}^{(r-1)}) \nabla a(\blam^{(r-1)},\overline{\Y}^{(r-1)},$ $\hat{\boldsymbol{\sigma}}^{2,(r-1)})^\top \dfrak^{(r)} \ge 0$. By definition of $\dfrak^{(r)}$,   $$\mathop{\lim\inf}_{r \to \infty} \sfrak^{\max}(\dfrak_0,\hat{\blam}^{(r-1)}) \nabla a(\hat{\blam}^{(r-1)},\overline{\Y}^{(r-1)}, \hat{\boldsymbol{\sigma}}^{2,(r-1)})^\top \dfrak_0 \ge  \mathop{\lim\inf}_{r \to \infty} \sfrak^{\max}(\dfrak^{(r)},\hat{\blam}^{(r-1)}) W^{(r)} \ge 0,$$ which contradicts the assumption in \eqref{ineq:gradless}. Thus, the limiting point is the stationary point to the convex program (10). Since $a(\hat{\blam}^{(r)},\boldsymbol{y}, \boldsymbol{\sigma}^2)$ decreases with $r$, there exists $r_{\varepsilon}$ such that $a(\hat{\blam}^{(r)},\boldsymbol{y}, \boldsymbol{\sigma}^2) < a_{\varepsilon}$ for $r \ge r_{\varepsilon}$. Then $\boldsymbol{\chi} (\hat{\blam}^{(r)}, \boldsymbol{y}, \boldsymbol{\sigma}^2) \in \Xi \setminus \Xi_{\varepsilon}$. By continuity, there exists $\varepsilon'$ such that the $\hat{\alpha}_{i,h}^*$ calculated at Step 5 of CR\&S Algorithm 2 satisfies $| \hat{\alpha}_{i,h}^* - \alpha_{i,h} | \le \varepsilon'$, $i=1,\dots,k$, $h=1,\dots,p$, for $r \ge r_{\varepsilon}$, where $\alpha_{i,h}$ is the optimal solution of problem (8) of the main paper and $\varepsilon'$ decreases to zero as $\varepsilon$ decreases to zero. Thus, $\hat{\alpha}_{i,h}$ converges to the optimal solution of problem (8).

% Acknowledgments here
%\ACKNOWLEDGMENT{The authors gratefully acknowledge the existence of
%the Journal of Irreproducible Results and the support of the Society
%for the Preservation of Inane Research.}

% References here (outcomment the appropriate case)

% CASE 1: BiBTeX used to constantly update the references
%   (while the paper is being written).
%\bibliographystyle{informs2014} % outcomment this and next line in Case 1
%\bibliography{<your bib file(s)>} % if more than one, comma separated

% CASE 2: BiBTeX used to generate mypaper.bbl (to be further fine tuned)
%\input{mypaper.bbl} % outcomment this line in Case 2

%If you don't use BiBTex, you can manually itemize references as shown below.

%\bibliographystyle{nonumber}

\bibliographystyle{poms}
\bibliography{reference}

%\newpage
%\appendix
%	
%\setcounter{page}{1}

%%%%%%%%%%%%%%%%%
\end{document}